\documentclass[oneside, a4paper]{amsart}
\usepackage[utf8]{inputenc}
\usepackage[english]{babel}
\usepackage{amssymb,amsmath,amsthm}
\usepackage{graphicx}

\usepackage[backref=page]{hyperref}

\usepackage{tikz-cd}

\usepackage{xspace}
\usepackage{enumerate}
\usepackage{verbatim}

\usepackage[margin=1in]{geometry}

\usepackage{mathtools}
\newcommand{\F}[1]{\ensuremath{\mathbb{F}_{#1}}}
\newcommand{\Fp}{\mathbb{F}_p}

\newcommand{\QQ}{\ensuremath{\mathbb{Q}}\xspace}
\newcommand{\CC}{\ensuremath{\mathbb{C}}\xspace}

\newcommand{\ZZ}{\ensuremath{\mathbb{Z}}\xspace}
\newcommand{\Zp}{\ensuremath{\mathbb{Z}_{(p)}}\xspace}
\newcommand{\Z}[1]{\ensuremath{\mathbb{Z}_{(#1)}}\xspace}
\newcommand{\NN}{\ensuremath{\mathbb{N}}\xspace}
\newcommand{\LL}{\ensuremath{\mathbb{L}}\xspace}
\newcommand{\inN}{\ensuremath{\in\mathbb{N}}\xspace}

\newcommand{\OO}{\ensuremath{\mathcal{O}}\xspace}
\newcommand{\rarr}{\rightarrow}

\newcommand{\xrarr}[1]{\xrightarrow{#1}}

\newcommand{\pd}{\partial}

\newcommand{\mc}[1]{\ensuremath{\mathcal{#1}}}
\newcommand{\id}{\ensuremath{\mbox{id}}}

\newcommand{\bu}{\bullet}
\newcommand{\ot}{\otimes}

\newcommand{\M}{\mc{M}}
\newcommand{\Res}{\underset{t=0}{\mathrm{Res\ }}}
\newcommand{\pphi}{\varphi}

\newcommand{\Kntop}{\mathrm{K}(n)^{top}}
\newcommand{\Kn}{\mathrm{K}(n)_{int}}
\newcommand{\Kntilde}{\widetilde{\mathrm{K}}(n)_{int}}
\newcommand{\Knspec}{\mathcal{K}(n)_{int}}

\newcommand{\oKn}{\overline{\mathrm{K}(n)}_{int}}
\newcommand{\Km}[1][n]{\mathrm{K}(#1)_{int}}
\newcommand{\CKK}{\mathrm{CK}_0}
\newcommand{\KK}{\mathrm{K}_0}

\renewcommand{\H}{\mathrm{H}}

\newcommand{\MU}{MU}

\DeclareMathOperator{\Spec}{\mathrm{Spec}}

\DeclareMathOperator{\CH}{\mathrm{CH}}
\DeclareMathOperator{\Hom}{\mathrm{Hom}}
\DeclareMathOperator{\Tor}{\mathrm{Tor}}
\DeclareMathOperator{\colim}{\mathrm{colim}}
\DeclareMathOperator{\SH}{\mathrm{SH}}

\DeclareMathOperator{\BPspec}{\mathcal{BP}}
\DeclareMathOperator{\MGL}{\mathrm{MGL}}
\DeclareMathOperator{\KGL}{\mathrm{KGL}}
\DeclareMathOperator{\BPtop}{BP^{top}}
\DeclareMathOperator{\Entop}{\mathrm{E}(n)^{top}}

\theoremstyle{definition}
\newtheorem{Def}{Definition}[section]
\newtheorem*{Def-intro}{Definition}
\newtheorem{Rk}[Def]{Remark}
\newtheorem{Exl}[Def]{Example}
\theoremstyle{plain}
\newtheorem{Th}[Def]{Theorem}
\newtheorem*{Th-intro}{Theorem}
\newtheorem{Prop}[Def]{Proposition}
\newtheorem{Cr}[Def]{Corollary}
\newtheorem{Lm}[Def]{Lemma}
\newtheorem{Not}[Def]{Notation}

\newtheorem*{BigTh-add}{Algebraic Classification of Additive Operations Theorem  (CAOT)}

\makeatletter
\newcommand\footnoteref[1]{\protected@xdef\@thefnmark{\ref{#1}}\@footnotemark}
\makeatother

\begin{document}

\title[Chern classes from Morava K-theories to $p^n$-typical oriented theories]%
{Chern classes from Morava K-theories\\
to \mbox{$p^n$-typical} oriented theories}
\author{Pavel Sechin}

\begin{abstract}
We study non-additive operations from algebraic Morava K-theories to
oriented cohomology theories in algebraic geometry. 
For oriented cohomology theory $A$ that has a {$p^n$}-typical formal group law
over a $\Zp$-algebra we construct `Chern classes' from 
the algebraic $n$-th Morava K-theory with $p$-local coefficients to $A$.
If the coefficient ring of $A$
is a free $\Zp$-module we also prove that these Chern classes 
freely generate all operations from $\Kn^*$ to $A$.

Examples of such theories are algebraic Morava K-theories $\Km[nm]^*$ for all $m\in\NN$
and Chow groups with p-local coefficients. 
The universal $p^n$-typical oriented theory is $BP\{n\}^*$
whose coefficient ring is also a free $\Zp$-module.

Chern classes from the $n$-th algebraic Morava K-theory $\Kn^*$ to itself
allow us to introduce the gamma filtration on $\Kn^*$.
This is the best approximation to the topological filtration obtained by values of operations
and it satisfies properties similar to that of the classical gamma filtration on $\KK$.
The major difference from the classical case
 is that Chern classes from the graded factors $gr^i_\gamma \Kn^*$
to Chow groups with p-local coefficients are surjective for $i\le p^n$,
which allows to estimate $p$-torsion 
in Chow groups of codimension up to $p^n$ of some varieties.
\end{abstract}



\maketitle

\section*{Introduction}

\subsection*{Orientable cohomology theories in algebraic geometry}

In algebraic geometry a cohomology theory can be said to be orientable
if it can be endowed with a suitable notion of push-forward maps for projective morphisms.
For example, push-forward maps in Chow groups can be defined using direct image of cycles,
in K-theory of smooth varieties one can define push-forward maps using the right derived functor
of the push-forward of coherent sheaves, and in Weil-type cohomology theories, e.g.\ \'etale cohomology,
push-forward maps can be described using the Poincare duality. 

It was observed by Panin and Smirnov \cite{PanSmi}
that the orientation
can also be given by the structure of Chern classes of line bundles satisfying certain axioms.
As the projective bundle theorem is usually contained in the axioms of these cohomology theories,
one can define higher Chern classes due to the classical method of Grothendieck.
Moreover, one could say that an orientable theory is a cohomology theory
which can be endowed with a suitable notion of Chern classes.
A possibility to be oriented then leads to the computation of the ring of operations from $\KK$ 
to many cohomology theories (e.g.\ for free theories, see below),
which is freely generated by Chern classes.

Altogether these observations confirm a distinguished role of the K-theory with respect 
to orientation. The scope of this paper is a search for a different notion of orientability
where the distinguished role is played by Morava K-theories $\Kn^*$ instead of $\KK$.
To introduce $\Kn^*$ let us recall how one can construct oriented cohomology theories in algebraic geometry.

There exists the universal oriented\footnote{Note that the 
 notions of orientation due to Levine-Morel and Panin-Smirnov differ in some technical details.} cohomology theory $\Omega^*$ which is called
algebraic cobordism of Levine-Morel \cite{LevMor}. 
The universality has the following meaning here: for every oriented theory $A^*$
there exist a unique morphism
of presheaves of rings from $\Omega^*$ to $A^*$ which respects both push-forward and pullback maps.
 
Another universal property of algebraic cobordism permits to introduce many 
new cohomology theories. The ring of coefficients of $\Omega^*$, at least over fields 
of zero characteristic,
is isomorphic to the Lazard ring classifying formal group laws. 
For any formal group law $F_R$ over a commutative ring $R$
the presheaf of rings $\Omega^* \ot_\LL R$ is an oriented cohomology theory, called free theory,
where the map $\LL\rarr R$ corresponds to $F_R$.
Well-known examples of free theories are Chow groups and K-theory $\KK$
obtained via the additive and the multiplicative formal group laws over $\ZZ$, respectively.

\subsection*{Morava K-theories}

If one takes $R$ to be $\Zp$ where $p$
is a prime number, and chooses a formal group law over $\Zp$
with the logarithm of the form $\log_{\Kn}(x) = x+\frac{a_1}{p}x^{p^n}+\frac{a_2}{p^2}x^{p^{2n}}+\ldots$,
where $a_1\in\Zp^\times$, then the corresponding free theory $\Kn^*$ 
is called an $n$-th algebraic Morava K-theory.
In a special case $n=1$ and $\log_{\mathrm{K}(1)_{int}}(x)=\sum_{i\ge 0} \frac{1}{p^i}x^{p^i}$
free theory $\mathrm{K}(1)^*_{int}$ is isomorphic to $\KK\ot\Zp$ as presheaf of rings.
It seems that this isomorphism was the main motivation for topologists to call
such cohomology theories similarly: Morava K-theories.

The definition of Morava K-theories given above is quite ad hoc and  demands some explanations,
which we adopt from topology. 
Algebraic cobordism $\Omega^*$ of a smooth variety together with Landweber-Novikov operations
on it defines a quasi-coherent sheaf on the moduli stack of formal groups (see e.g.\ \cite[Prop.~2.10]{Sec-Cob}).
This stack has a particularly nice description
 after base change to $\F{p}$: 
there exist a decreasing filtration
by closed substacks $\mc{M}_{fg}^{\ge n}$ which
classify formal group laws of height not less than $n$, and
$\mc{M}_{fg}^{\ge n+1}$ is a Cartier divisor in $\mc{M}_{fg}^{\ge n}$.
Moreover, $\mc{M}_{fg}^{\ge n} \setminus \mc{M}_{fg}^{\ge n+1}$
has an essentially unique geometric point.

Formal group laws yield free theories as explained above,
and isomorphic formal group laws yield multiplicatively isomorphic free theories.
Thus, a field-valued point 
of the stack of formal groups defines a free theory as a presheaf of rings
(but the orientation on it can be chosen differently).
Morava K-theories $\Kn^*/p$, as defined above, are those free theories that correspond to 
some of the $\F{p}$-points of $\mc{M}_{fg}^{\ge n} \setminus \mc{M}_{fg}^{\ge n+1}$,
but these theories become multiplicatively isomorphic with $\overline{\F{p}}$-coefficients.
Thus Morava K-theories exhaust all free theories with $\F{p}$-coefficients
except for the geometric point `at infinity', i.e. in the intersection of all substacks $\mc{M}_{fg}^{\ge n}$, $n\ge 1$.
This point corresponds to Chow groups modulo $p$, sometimes called the infinite Morava K-theory.

As the first Morava K-theory is nothing else than K-theory modulo $p$,
 the `chromatic structure' of the stack $\mc{M}_{fg}$ 
places $\Kn^*$ into an intermediate position between $\KK/p$ and $\CH^*/p$.
This is the underlying reason why Morava K-theories 
behave simpler than Chow groups, and yet more complicated than $\KK$.
Known instances of this phenomenon are the computation of operations (see theorems below) and some results about motives.
For example,
if $X$ is a projective homogeneous variety such that $M_{\Kn^*}(X)$ is a Tate motive
(here $M_{\Kn^*}$ is a functor to the category 
of pure motives corresponding to an oriented theory $\Kn^*$),
then $M_{\mathrm{K}(m)^*_{int}}(X)$ is also a Tate motive for $m<n$ (\cite[Cor. 7.11]{SechSem}).
This result gives rise to an effective `linear' strategy for the study of motives of projective homogeneous varieties
 that goes from $\mathrm{K}(1)^*_{int}*$ (aka $\KK$) to $\Kn^*$, $n>1$,
and finishes with $\CH^*\ot\Zp$ for every prime $p$.

Even though the results about Morava K-theories $\Kn^*$ might be interesting by themselves,
one also hopes to obtain some information about the more classical theories such as Chow groups
with their help.
The first step in this direction was made in \cite{Sech}
where so-called Chern classes, operations from $\Kn^*$ to $\CH^*\ot\Zp$ were constructed.
These Chern classes are similar to the classical Chern classes from $\KK$ to $\CH^*$:
they are free generators of all operations and satisfy a Cartan-type formula.
This result led us to ask the question
about the existence of the notion of Morava-orientable theories ([{\sl op.cit.}, Introduction]).
The goal of this paper is to construct new Chern classes from $\Kn^*$ to many other oriented theories
and to provide at least partial answer to the question of orientability. 
Among these operations, perhaps, the most important are the endo-operations of $\Kn^*$
which give rise to the gamma filtration on $\Kn^*$ 
and provide a tool to calculate Chern classes from $\Kn^*$ to $\CH^*\ot\Zp$ for certain varieties.

\subsection*{Results of the paper}

\subsubsection*{Morava-orientable theories and Chern classes}

We start with the description of the results of the paper
with the following definition generalizing Cartier's $p$-typical formal group laws.

\begin{Def-intro}[{for the full definition encompassing torsion-case see Section \ref{sec:pn-typ},
 cf. the notion of formal $A$-module \cite[Ch.~21]{Haz}}]
Let $A$ be a torsion-free $\Zp$-algebra.
A formal group law $F$ over $A$ is called $p^n$-typical 
if its logarithm is of the form $\log_F(x) = \sum_{i\ge 0} l_i x^{p^{ni}}$
where $l_i \in A\otimes_\ZZ \QQ$.

A free theory is called $p^n$-typical if the corresponding 
formal group law is $p^n$-typical.
\end{Def-intro}

The main result of this paper is the following.

\begin{Th-intro}[Theorem \ref{th:main}; if $A^*=\CH^*\ot\Zp$ this is {\cite[Th. 4.2.1]{Sech}}]
For every Morava K-theory $\Kn^*$ 
and every $p^n$-typical theory $A^*$
there exist a series of operations $c_j\colon\Kn^*\rarr A^*$ for $j\ge 1$
satisfying the following conditions. 
\begin{enumerate}[i)]
\item Operation $c_j$ takes values in $\tau^jA^*$,
where $\tau^\bullet$ is the topological filtration on $A^*$.
\item Denote by $c_{tot} = \sum_{i\ge 1} c_it^i$ the total Chern class in a formal variable $t$. 
Then the Cartan formula holds universally: 
$$c_{tot}(x+y)=F_{\Kn}(c_{tot}(x),c_{tot}(y)),$$
where $x,y\in \Kn^*(X)$ for a smooth variety $X$ and the identity takes place in $A^*(X)[[t]]$;
\item If $A=A^*(\Spec k)$ is a free $\Zp$-module, then 
all operations from $\Kn^*$ to $A^*$ are uniquely expressible as series in Chern classes:
$$ [\tau^1 \Kn^*, A^*] = A[[c_1,\ldots, c_i,\ldots ]],$$
where $\Kn^*=\Zp\oplus \tau^1 \Kn^*$ as presheaves of abelian groups,
and $[\mc{F},\mc{G}]$ denotes the set of natural transformations of presheaves of sets $\mc{F}, \mc{G}$
on the category of smooth varieties.
\end{enumerate}
\end{Th-intro}

Examples of $p^n$-typical theories include Morava K-theories $\mathrm{K}(mn)^*_{int}$ for $m\ge 1$,
Chow groups $\CH^*\ot\Zp$ and the universal $p^n$-typical theory $BP\{n\}^*$
whose ring of coefficients is a polynomial algebra $\Zp[v_n, v_{2n}, \ldots ]$.
Following an analogy with K-theory $p^n$-typical theories might be called Morava-orientable
even though we do not give any geometric definition of Morava orientation.

We also provide some circumstantial evidence that $p^n$-typical theories
are precisely the Morava-oriented theories: if the right notion of Morava-orientability exists,
then the intersection of Morava-orientable theories with oriented theories 
is exactly the set of $p^n$-typical theories (see Appendix \ref{app:non-exist-op}).
Nevertheless, questions of the existence of Morava-orientable but non-orientable theories
and of finding a correct and geometric definition of Morava-oriented theories remain open.

\subsubsection*{Uniqueness of Morava K-theories}

We should also mention that for each prime number $p$ 
and each number $n\in \NN$ the definition of Morava K-theory above yields infinitely many
oriented theories $\Kn^*$. It is not true that all of them are multiplicatively isomorphic 
(Appendix \ref{app:mor_not_mult}),
however we can prove the following.

\begin{Th-intro}[Theorem \ref{th:morava_unique}]
Let $\Kn^*$, $\oKn^*$ be two $n$-th Morava K-theories over $\Zp$.

Then there exist an isomorphism of presheaves of abelian groups 
$\Kn^*\xrarr{\sim} \overline{K}(n)^*$.
\end{Th-intro}

The uniqueness of Morava K-theories is not an issue in topology,
where there exist a unique spectrum with $\F{p}$-coefficients
representing the $n$-th topological Morava K-theory. The arguments used there, however, can not be applied
in the algebraic situation, and, moreover, we do not know whether
any kind of uniqueness statement for spectra representing Morava K-theories with $\Zp$-coefficients is known 
even in topology.
It may be possible, however, to prove a uniqueness result of this kind in the topological setting 
using the results or the methods of the preprint \cite{Luecke-Peterson}.

\subsubsection*{Gamma filtration on Morava K-theories}

The main application of the existence of Chern classes from $\Kn^*$ in this paper
is the construction of the gamma filtration on $\Kn^*$.

\begin{Def-intro}
Let $c_i\colon\Kn^*\rarr \Kn^*$ be Chern classes constructed in the theorem above.
Define the gamma filtration on $\Kn^*(X)$ for a smooth variety $X$ by the following formula
$$\gamma^m \Kn^*(X):= < c_{i_1}(\alpha_1)\cdots c_{i_k}(\alpha_k)| \sum_j i_j\ge m,\alpha_j \in \Kn^*(X)>, \quad m\ge 1,$$
where the brackets denote the generation as $\Zp$-module.
\end{Def-intro}

The gamma filtration satisfies
properties similar to those of the classical gamma filtration on $\KK$
which we summarize below. 
The main difference is that Chern classes from corresponding graded factors on $\Kn^*$
map surjectively to Chow groups $\CH^*\ot\Zp$
of codimension up to $p^n$ (in contrast to up to $p$ in the case of $p$-local K-theory).

\begin{Th-intro}[Prop. \ref{prop:morava_gamma_properties}]
Denote by $c_i^{\CH}$ the Chern classes from $\Kn^*$ to $\CH^*\ot\Zp$.

The gamma filtration $\gamma$ and the topological filtration $\tau$ on $\Kn^*$ satisfy the following properties:
\begin{enumerate}[i)]
\item $\gamma^i\Kn^*\subset \tau^i\Kn^*$;
\item $\gamma^\bullet \ot \QQ = \tau^\bullet \ot\QQ$;
\item $c_i^{\CH}|_{\tau^{i+1}\Kn^*}=0$;
\item the operation $c^{\CH}_i$ is additive when restricted to $\tau^i \Kn^*$,
and the following maps are isomorphisms of abelian groups:
$$ gr^i_\gamma \Kn^*\ot\QQ \rarr gr^i_\tau \Kn^*\ot\QQ \xrarr{c_i^{\CH}} \CH^i\ot\QQ;$$
\item $c_i^{\CH}\colon gr^i_\tau \Kn^*\rarr \CH^i\ot\Zp$ is an isomorphism for $i\colon 1\le i \le p^n$;
\item\label{item:intro-chern-surj} $c_i^{CH}\colon gr^i_\gamma \Kn^*\rarr \CH^i\ot\Zp$ is surjective for $i\colon 1\le i \le p^n$.
\end{enumerate}
\end{Th-intro}

The property \ref{item:intro-chern-surj}) is a new tool for obtaining estimates 
on the $p$-torsion in Chow groups of codimension up to $p^n$. Even though the problem of calculating
$gr_\gamma^\bullet \Kn^*(X)$ is not an easy one, we describe now a situation when it can be solved. 
If $X$ is a geometrically cellular variety,
such that the map $\Kn^*(X)\rarr \Kn^*(\overline{X})$ is an isomorphism (where $\overline{X}$ is 
the base change of $X$ to a base field where it becomes cellular, e.g.\ to the separable closure
of the base field),
then the graded pieces of the gamma filtration can be calculated over 
the algebraic closure and thus depend only on the `combinatorial' cellular
 structure of the variety $\bar{X}$.
First example of such variety was found by Voevodsky and it is a Pfister quadric
of dimension $2^{n+2}-2$ \cite{Voe}.
In a joint paper with Semenov \cite{SechSem} we
show that many other quadrics satisfy this assumption 
and the results of this paper are then crucial for obtaining
 new bounds on torsion in the groups $\CH^{\le 2^n}$ of them [Th.~8.14, loc.cit.], see Section~\ref{section:application_quadrics}.

\subsubsection*{Methods and ideas of the proof}

Having summarized the results, let us briefly describe the methods and the tools of the paper.
Note that algebraic cobordism admits a geometric description \cite{LevPand}:
the generators of the abelian group $\Omega^k(X)$, $k\in \ZZ$, for a smooth variety $X$
are classes of isomorphisms of projective morphisms $f\colon Y\rarr X$
of codimension $k$ with $Y$ being a smooth variety, and the relations are generated
by an analogue of the classical cobordance relation and, also, by so-called double-point relations.
Therefore the same classes generate all free theories (as a module over the ring of coefficients),
however, it is unclear whether one could write explicit geometric relations between these generators $f\colon Y\rarr X$ for any free theory apart from $\CH^*$ and $\KK$.
 Despite this lack of geometric description there exist tools to calculate operations between free 
theories in an efficient way.
Namely, the problem of constructing operations between free theories was reduced to a purely algebraic
question by Vishik in \cite{Vish1, Vish2}. More precisely, Vishik's result says that
an operation $\phi$ from $A^*$ to $B^*$ can be uniquely reconstructed
by the data of action of $\phi$ on products of projective spaces
which commutes with pullbacks along several types of morphisms between them.
For additive operations this produces a linear system of equations
which depend only on formal group laws of theories involved. For general operations
the system of equations can be considered as non-linear but still depends only on the formal group laws over $A$ and $B$.

The case of classification of operations from $\Kn^*$ to $\CH^*\ot\Zp$ which we treated in \cite{Sech}
turns out to be the simplest case, mainly because the topological filtration is split
on Chow groups by the graded components. The general case can be reduced to it by the following 
construction.

First, we prove that an operation $\phi\colon A^*\rarr B^*$ takes values in the $i$-th part of the topological
 filtration $\tau^i B^*$
if and only if it does so on products of projective spaces (Prop.~\ref{prop:operations_top_COT}).
Second, we construct an injective map of $B$-modules (Section~\ref{sec:tr_constr}):

$$ tr_i\colon [A^*,\tau^i B^*]/[A^*, \tau^{i+1}B^*]
 \hookrightarrow [A^*,\CH^i\ot B], \quad i\ge 0, $$

where 
$[A^*,\tau^j B^*]$ has a natural structure of an abelian group since $\tau^j B^*$ 
is a presheaf of abelian groups,
and we call $tr_i$ the truncation map.
For an operation $\phi\colon A^*\rarr \tau^i B^*$ the operation $tr_i \phi\colon A^*\rarr \CH^i\ot B$
can be characterized by the equation:  
$$\rho_B\circ tr_i\phi \equiv \phi \mod \tau^{i+1} B^*$$
where $\rho_B\colon\CH^i\ot B\rarr gr^i_\tau B^*$ is the canonical morphism of theories.
In other words, $tr_i\phi$ is the lift of the operation $\phi \mod \tau^{i+1}B^*$ to Chow groups
along the map $\rho_B$.

Finally, we show that when $A^*=\Kn^*$ and $B^*$ is a $p^n$-typical theory,
 the truncation map is an isomorphism,
 and we reduce the classification of operations 
to the already known case where $B^*$ is the $p$-local theory of Chow groups.
 We should emphasize that
these results are also based on Vishik's theorem
which, thus, constitutes the main tool of the paper.

\subsubsection*{Outline of the paper}

Section \ref{prelim} is a recap of oriented cohomology theories, Vishik's results on operations between them
and the properties of the topological filtration on free theories.

Section \ref{sec:trunc} contains the construction of the truncation map.
We classify operations between free theories which target a part
of the topological filtration (\ref{sec:op_topfilt}),
provide an algebraic description of the truncation map and 
discuss its properties (\ref{sec:tr_constr}).
We also consider a special refinement of the truncation map which we will need later (\ref{sec:trunc_mod}).
Finally, we investigate the general properties of modules of operations
when all the truncation maps are isomorphisms (\ref{sec:all_op_from_trunc}).

Section \ref{sec_op_mor} contains definitions of $p^n$-typical formal group laws (\ref{sec:pn-typ}),
Morava K-theories (\ref{sec:def_morava}) as well as their properties and the statement
of the main theorem on the classification of operations from $\Kn^*$ to $p^n$-typical theories (\ref{subsec_mainth}).

Section \ref{sec:proof_main} contains the proof of the main theorem
for which we need also to classify additive operations from Morava K-theories
to $p^n$-typical theories (\ref{sec:add_morava_BPn}).

In Section \ref{sec:morava_unique} we prove the uniqueness of $n$-th Morava K-theory as a presheaf of abelian groups.

Section \ref{sec:morava_gamma_filtration} contains the definition and properties 
of the gamma filtration on Morava K-theories (\ref{sec:morava_gamma_def_properties}),
the proof of its uniqueness (\ref{sec:gamma_unique})
and calculation of some computational constants related to it (\ref{sec:constant_chern}).

Appendix \ref{app} contains results which are complementary
to the rest of the paper. In \ref{app:non-exist-op} we show that there could not exist Chern classes
from the $n$-th Morava K-theory to the $m$-th Morava K-theory except for the cases treated in the main theorem.
Section \ref{app:mor_not_mult} contains examples of $n$-th Morava K-theories which are not multiplicatively isomorphic.
Finally, Section \ref{app:image_of_Chern_Chow} gives an inductive formula for the image of Chern classes
from $gr^i_\tau \Kn^*$ in $\CH^i\ot\Zp$.

\section*{Acknowledgements}

The author is extremely grateful to Alexander Vishik 
for the attentive guidance 
in the earlier stages of this work.

We thank the anonymous referee for a careful reading of the paper and for many helpful comments and questions, 
in particular those that led to the inclusion of Appendix~\ref{app:disclaimer}.

\section{Preliminaries}\label{prelim}

Fix a field $k$, throughout the paper it is assumed to have characteristic $0$.
All varieties over $k$ are assumed to be
quasi-projective.

In this section we recall basic definitions of oriented theories 
and describe the key input results of Vishik 
on the classification of operations.

\subsection{Oriented theories and formal group laws}\label{sec:def}

The following definition is a version of {\cite[Def. 1.1.2]{LevMor}} 
(cf. also \cite[2.0.1]{PanSmi})
with an additional axiom (LOC) and no grading on rings.

\begin{Def}[{\cite[2.1]{Vish1}}]\label{goct}
 
An {\it oriented cohomology theory} (or just an oriented theory) is
a presheaf $A^*$ of commutative rings on the category of smooth quasi-projective 
varieties over $k$ supplied with the data of push-forward maps for projective morphisms.
Namely, for each projective morphism of smooth varieties $f\colon X\rarr Y$,
morphisms of abelian groups $f_*\colon A^*(X)\rarr A^*(Y)$ are defined.

The structure of push-forwards has to satisfy the following axioms 
(for precise statements see {\sl ibid}):
 functoriality for compositions (A1), base change for transversal morphisms (A2),
the projection formula, the projective bundle theorem (PB),
 $\mathbb{A}^1$-homotopy invariance (EH)
and the localization axiom (LOC).\end{Def}

\begin{Def}\label{def:mult_iso}
Two oriented theories $A^*$, $B^*$ are called {\it multiplicatively isomorphic},
if there is an isomorphism between underlying presheaves of commutative rings,
i.e.\ $A^*(X)\cong B^*(X)$ as rings, functorially in $X$ with respect to pullback maps
(these isomorphism do not have to preserve push-forward maps).
\end{Def}

Recall that $A^*$ satisfies (LOC) (see \cite[2.1]{Vish1}, cf.\ \cite[3.2]{LevMor}),
if for any smooth $X$, a closed subset $Z$ in $X$ and its open complement $U$
the following sequence is right exact:

\begin{equation}\label{eq:loc}
\tag{LOC}
A^*(Z)\rarr A^*(X) \rarr A^*(U) \rarr 0,
\end{equation}

where $A^*(Z)$ is defined as a colimit of the groups $A^*(\widetilde{Z})$ 
over projective morphisms $\tilde{Z}\rarr Z$ with smooth $\tilde{Z}$.

\begin{Not}\label{not:A}
If $A^*$ is an oriented theory, we denote by $A$ its ring of coefficients,
i.e.\ $A=A^*(\Spec k)$.
\end{Not}

Recall that a pair of morphisms between smooth varieties $f\colon X\rarr Z$, $g\colon Y\rarr Z$
is called {\it transversal} if $\Tor^{\OO_Z}_q(f_*\OO_X,g_*\OO_Y)=0$ for all $q>0$,
and $X\times_Z Y$ is a smooth variety (\cite[Def. 1.1.1]{LevMor}).
The axiom (A2) says that if $f$ is a projective morphism,
then $g^*f_* = f'^*g'_*$ where base change morphisms are denoted according to the following diagram:

\begin{center}
\begin{tikzcd}
X \times_Z Y \arrow[d, "g'"] \arrow[r, "f'"] & Y \arrow[d, "g"] \\
X \arrow[r,"f"] & Z
\end{tikzcd}
\end{center}
Note that in the definition above we do not assume that $A^*$ is a presheaf of graded rings,
and it will often be the case that $A^*$ has either no grading (as in the case of $\KK$)
or has $\ZZ/m$-grading for some $m\in \NN$ (as in the case of Morava K-theories $\Kn^*$, $n\in \NN$).
If $A^*$ is a graded theory, then the push-forward maps shift the grading by codimension,
but even if $A^*$ is non-graded we still keep the ${}^*$-sign to distinguish the presheaf
from its ring of coefficients $A:=A^*(\Spec k)$ (see Notation~\ref{not:A}).

Each oriented theory $A^*$ can be endowed with Chern classes of vector bundles $c_i^A$
using the classical method due to Grothendieck.
The first Chern class allows to associate the formal group law 
$F_A\in A[[x,y]]$ to the theory $A^*$ (\cite[Lemma 1.1.3]{LevMor})
so that it satisfies the following equation: $c_1^A (L_1\ot L_2)=F_A(c_1^A(L_1),c_1^A(L_2))$
for every pair of line bundles $L_1, L_2$
over a smooth variety $X$.
Recall that the formal group laws (FGLs) over a ring $A$ are in 1-to-1 correspondence
with the ring morphisms from the Lazard ring $\LL$ to $A$. In particular,
the construction above yields a morphism
 of rings $\LL\rarr A^*(\Spec k)$ for all oriented theories $A^*$.

\begin{Th}[Levine-Morel, {\cite[1.2.6]{LevMor}}]\label{th:alg_cob_universal}
There exists the universal oriented theory $\Omega^*$ called algebraic cobordism,
i.e. for any oriented theory $A^*$ 
there exists a unique morphism of presheaves of rings
$p_A\colon\Omega^*\rarr A^*$ which respects the structure of push-forwards.

The associated formal group law of algebraic cobordism
 is the universal formal group law,
and $\Omega^*(\mathrm{Spec\ }k)$ 
is canonically isomorphic to the Lazard ring $\mathbb{L}$.
\end{Th}

Once $\Omega^*$ is known to exist, 
one can construct an oriented theory with a prescribed formal group law in the following way.

\begin{Def}[Levine-Morel,{ \cite[Rem. 2.4.14]{LevMor}}]\label{def:free_theory}
Let $R$ be a ring, let $\mathbb{L}\rightarrow R$ be a ring morphism
corresponding to a formal group law $F_R$ over $R$.

Then $\Omega^*\otimes_\mathbb{L} R$ is an oriented theory which is called a {\bf free theory}
(the morphism $\mathbb{L}\rightarrow R$ and the identification $\Omega^*(\mathrm{Spec\ }k)\cong \LL$
are hidden in the notation).
The ring of coefficients of $\Omega^*\otimes_\mathbb{L} R$ is $R$, and its associated FGL is $F_R$.
\end{Def}

Chow groups $\CH^*$ and K-theory $\KK$ are the most well-known examples of free theories (see {\sl op.cit.})
which also have alternative definitions via generators and relations.

Note that there is no condition imposed on the formal group law in this definition,
in particular, no Landweber-exactness. Although it looks wrong from the perspective of topology,
this definition makes sense in algebraic geometry and can be compared to the topological situations in many instances.
In particular, for the topological Morava K-theory $\Kntop$
there exists a motivic spectrum $\mathcal{K}(n)$ in the stable motivic category $\SH(k)$,
which is constructed from the algebraic cobordism specrum $\MGL$ by a similar procedure as $\Kntop$ is constructed from $\MU$.
Moreover, for a smooth variety $X$ over $k$ there is an isomorphism
$$ \Knspec^{2*,*}(X) \cong \Omega^*(X) \otimes_{\LL} \Knspec^{2*',*'}(k).$$
We refer the reader to Appendix~\ref{app:disclaimer} for the explanation of this phenomenon in algebraic geometry,
where we give an overview of the relevant references without claiming any originality.

\subsection{Operations and poly-operations}\label{sec:operations}

\begin{Def}\label{operation}
Let $A^*, B^*$ be presheaves of sets (resp.\ of abelian groups)
 on the category of smooth varieties over a field.

An (additive) {\it operation} $\phi\colon A^*\rarr B^*$ is a morphism of presheaves of sets (resp.\ of abelian groups). 
The set of all (additive) operations from $A^*$ to $B^*$ is denoted by $[A^*, B^*]$ (resp.\  $[A^*,B^*]^{add}$). 

If $B^*$ is a presheaf of commutative rings, the set $[A^*, B^*]$ 
 has the natural commutative ring structure
given by the multiplication and addition on the target theory:
for $\phi_1,\phi_2\colon A^*\rarr B^*$ we have 
$$ (\phi_1\cdot \phi_2)(x) = \phi_1(x)\cdot \phi_2(x), \quad (\phi_1+\phi_2)(x)=\phi_1(x)+\phi_2(x),$$
where $x\in A^*(X)$ for some smooth quasi-projective $X$ over $k$.
\end{Def}

As we will see (e.g.\ in Section~\ref{subsec_poly}) non-additive operations 
naturally give rise to poly-operations. 

\begin{Def}[{\cite[Def. 4.2]{Vish2}}]
Let $A^*, B^*$ be presheaves of sets (or abelian groups, or rings)
 on the category of smooth varieties over a field $k$.
 
An {\it external $r$-ary poly-operation} from $A^*$ to $B^*$ 
is a morphism of presheaves of sets on the $r$-product category
of smooth varieties over a field $k$
from $(A^*)^{\times r}$ to $B^*\circ \prod^r$.

Explicitly, for smooth varieties $X_1, \ldots, X_r$
an external poly-operation yields a map of sets 
$$ A^*(X_1)\times A^*(X_2) \times \ldots \times A^*(X_r) 
\rightarrow B^*(X_1 \times X_2 \times \ldots \times X_r)$$
in a functorial way.

An {\it internal $r$-ary poly-operation} from $A^*$ to $B^*$ 
is a morphism of presheaves of sets on the category
of smooth varieties over a field $k$
from $(A^*)^{\times r}$ to $B^*$.

Explicitly, for a smooth variety $X$
an internal poly-operation yields a map of sets 
$$ A^*(X)\times A^*(X) \times \ldots \times A^*(X) 
\rightarrow B^*(X)$$
in a functorial way.
\end{Def}

It is not hard to see that there is a 1-to-1 correspondence
 between these two notions (\cite[Section~4]{Vish2}).
The set of all (internal or external) $r$-ary poly-operations 
is denoted by $[(A^*)^{\times r}, B^*\circ \prod^r]$.

\subsection{Derivatives and products of poly-operations}\label{subsec_poly}

There are two straight-forward ways to produce some poly-operations from operations,
or in other words to increase the {\sl arity} of operations.

First, if $\phi_1, \phi_2$ are external $r_1$-ary and $r_2$-ary poly-operations, respectively,
from a presheaf of sets to a presheaf of rings,
then we can define an $(r_1+r_2)$-ary poly-operation $\phi_1\odot \phi_2$ 
as their external product: 
$$(\phi_1\odot \phi_2) (x_1,x_2,\ldots,x_{r_1},y_1,y_2,\ldots,y_{r_2}) = \pi_1^*\phi_1(x_1,x_2,\ldots,x_{r_1})
\cdot\pi_2^*\phi_2(y_1,y_2,\ldots,y_{r_2}),$$

where $x_i\in A^*(X_i)$ for $i\colon 1\le i\le r_1$, $y_i\in A^*(Y_i)$ for $j\colon 1\le j\le r_2$,
and \mbox{$\pi_1\colon \prod_i X_i \times \prod_j Y_j \rarr \prod_i X_i$},
\mbox{$\pi_2\colon \prod_i X_i \times \prod_j Y_j \rarr \prod_j Y_j$}
are projections.
Similarly, one can define products of internal poly-operations
which we also denote as $\phi_1 \odot \phi_2$.

If $B^*$ is a presheaf of (commutative) rings, this construction defines a morphism of (commutative) $B$-algebras
$$ [(A^*)^{\times r_1},B^*\circ \prod^{r_1}]\ot_B [(A^*)^{\times r_2},B^*\circ\prod^{r_2}] \xrarr{\odot} 
[(A^*)^{\times (r_1+r_2)}, B^*\circ \prod^{r_1+r_2}],$$
where $B=B^*(\Spec k)$ (Notation~\ref{not:A}).

In some cases this morphism of algebras turns out to be an isomorphism,
which may be seen, perhaps, as some kind of Künneth formula.
When this property is satisfied for all $r_1, r_2$ for $A^*$ and $B^*$
(see e.g.\ Th. \ref{basis}),
we will write $[(A^*)^{\times r}, B^*\circ \prod^r] = [A^*,B^*]^{\odot r}$.

Second, if $\phi$ is an external $r$-ary poly-operation,
then one can define an $(r+1)$-ary external poly-operation $\eth^1_i \phi$ 
as its derivative with respect to the $i$-th component (\cite[Def. 3.1, end of Section~4]{Vish2}).
Denote by $Z_{<i}=(z_1,\ldots,z_{i-1})$, $Z_{>i}=(z_{i+1}, \ldots, z_r)$, then 
$$ \eth^1_i \phi(Z_{<i},x,y,Z_{>i}) := 
\phi(Z_{<i},\pi_1^*(x)+\pi_2^*(y),Z_{>i})
-\phi(Z_{<i},\pi_1^*(x),Z_{>i})-\phi(Z_{<i},\pi_2^*(y),Z_{>i}),$$
where $z_j \in A^*(X_j)$ for  $j\colon 1\le j\le r$, $j\neq i$,
and $x\in A^*(X)$, $y\in A^*(Y)$, $\pi_{1}$, $\pi_{2}$ are projections from $X\times Y$ to $X$ and $Y$ respectively.
Similarly, one can define internal derivatives $\pd_i \phi$ of internal poly-operations.

If $r=1$, i.e. $\phi$ is an operation, we will omit the subscript
and write $\pd \phi$ and $\eth \phi$ to mean its derivatives.
Iterating the procedure one can easily define
$\eth^s_{(r_1,\ldots,r_s)}=\eth^1_{r_s}\circ\eth^1_{r_{s-1}}\circ \cdots \circ \eth^1_{r_1}$.
However, it is easy to see that all $s$-derivatives of an operation are symmetric
and thus derivatives do not depend on the order of derivation.
We will write $\eth^s \phi$ to denote any of them. 
Similarly, we define $\pd^s \phi$ to be the $s$-th internal derivative of $\phi$.

By definition of the derivative of $\phi$ one can express
values of $\phi$ on the sum of {\it two} elements
as the sum of values of $\phi$ and $\partial^1\phi$.
It is useful for computations to have analogous formulas
for the values on the sum of any number of elements.
We state this result for derivatives of maps between
abelian groups (\cite[Def. 3.1]{Vish2}), and it is clear 
that it can  be applied for both internal and external derivatives.

\begin{Prop}[Discrete Taylor Expansion, {\cite[Prop. 3.2]{Vish2}}]\label{prop:taylor}
Let $f\colon A\rarr B$ be a map between abelian groups.
Denote by $\partial^if\colon A^{\times i}\rarr B$ its derivatives.

For any set $\{a_i\}_{i\in I}$ of elements in $A$ the following equality holds:

$$ f(\sum_{i\in I} a_i) = \sum_{\emptyset\neq J\subset I} \partial^{|J|-1}f(a_j|j\in J). $$
\end{Prop}

\subsection{Vishik's classification of operations from theories of rational type}\label{sec:vishik}

Theories of rational type were introduced by Vishik in \cite{Vish1}
and are those oriented theories which satisfy an additional axiom (CONST)
and whose values on varieties can be reconstructed by induction on the dimension.

\begin{Def}[{\cite[Def 4.4.1]{LevMor}}]\label{def:const_axiom}
An oriented theory $A^*$ satisfies the {\it axiom (CONST)} 
if for every smooth irreducible variety $X$
the value of the theory in its generic point 
$$A^*(k(X)):=\colim_{U\subset X} A^*(U),$$
is canonically isomorphic to the ring of coefficients of the theory $A:=A^*(Spec (k))$ (Notation~\ref{not:A}).
The colimit in the above formula is taken over all non-empty open subsets of $X$.

In particular, this allows to split $A^*$ as presheaf of abelian groups into two summands:
$A^*=\tilde{A}^*\oplus \underline{A}$, 
where $\underline{A}$ is a constant presheaf\footnote{The value of this presheaf on a variety $X$ is $A^{\pi_0 X}$ 
where $\pi_0 X$ is the set of generic points of $X$.}
  and $\tilde{A}^*$
 is an ideal subpresheaf of elements which are trivial in generic points.
\end{Def}

Algebraic cobordism satisfies this axiom (\cite[Cor. 4.4.3]{LevMor}),
and therefore every free theory does as well.
As we have already mentioned theories of rational type
satisfy other properties except for (CONST),
 which are, however, very technical.
Fortunately, the following theorem allows us to skip the definition of theories of rational type.

\begin{Th}[Vishik, {\cite[Prop. 4.9]{Vish1}}]
Theories of rational type are precisely free theories.
\end{Th}

Throughout the paper we will call theories of rational type as `free theories'.

The main tool that we use in this paper
is Vishik's theorem classifying operations
from free theories to oriented theories.

\begin{Th}[{Vishik, \cite[Th. 5.1]{Vish1}, \cite[Th. 5.1]{Vish2}}]\label{th:Vish_op}

Let $A^*$ be a free theory and let $B^*$ be an oriented theory.

Then the set of operations from $A^*$ to $B^*$ preserving zero
is in 1-to-1 correspondence with the data of
pointed maps of sets (where the values of theories are pointed by zero) 
\mbox{$A^*((\mathbb{P}^\infty)^{\times l})\rightarrow B^*((\mathbb{P}^\infty)^{\times l})$}
 for $l\ge 0$
which commute with the pull-backs for:
\begin{enumerate}
\item the permutation action of symmetric groups $\Sigma_{l}$;
\item the partial diagonals;
\item the partial Segre embeddings;
\item the partial point embeddings;
\item the partial projections.
\end{enumerate}
\end{Th}

See also \cite[Th. 5.2]{Vish2} for the classification of external poly-operations.

\begin{Rk}\label{rem:grad}
If the target theory is graded, 
then the theorem allows one to compute poly-operations to each of the components of the target.

To see this note that grading on $B^*$ yields (additive) projectors $p_n\colon B^*\rarr B^n$
and an operation to a component $B^n$ is just an operation
which is zero when composed with $p_m$, $m\neq n$.
 As follows from the theorem, this property may be checked on products of projective spaces.
\end{Rk}

\begin{Rk}
A similar result was known in topology due to Kashiwabara (see \cite[Th. 4.2]{Kash}
where it is formulated for spectra representing cohomology theories).
Note that an oriented cohomology theory in the sense of Definition~\ref{goct}
is not representable by a motivic spectrum -- in the best case it is a part of a representable theory,
see Appendix~\ref{app:disclaimer}. This might explain 
the difference between Vishik's and Kashiwabara's theorems, since the latter demands 
certain additional conditions for spectra to be satisfied.
We do not know whether these conditions are fulfilled in our main cases of interest.
\end{Rk}

\subsection{Chern classes as free generators of operations from $\KK$}\label{sec:chern_op_K0}

The following is an application of Vishik's classification of operations 
that provides motivation for our results.
It shows that (classical) Chern classes freely generate all operations from $\tilde{K}_0$ 
to any oriented theory $A^*$.
The main result of this paper will be concerned with replacing $\KK$
in this statement by the $n$-th Morava K-theory $\Kn^*$
by defining new `Chern classes',
albeit restricting the class of oriented theories $A^*$ to a class of so-called $p^n$-typical oriented theories.

\begin{Th}[Vishik, for the proof see {\cite[Th. 2.1]{Sech}}]\label{basis}
Let $A^*$ be an oriented theory.

Then the ring of $r$-ary poly-operations from a presheaf $\tilde{\KK}$ to $A^*$ 
is freely generated over $A$ by products of Chern classes.
\end{Th}
\begin{Rk}
Note that there is
 no issue of convergence of a series of Chern classes for any particular element of $\KK$,
 since Chern classes $c_i^A$ are nilpotent on each variety.
 (Indeed, it is enough to show this for the universal oriented theory $\Omega^*$, 
  and in this case the claim follows because $\Omega^i(X)=0$ for $i>\dim X$.) 
\end{Rk}

Using notations from section \ref{subsec_poly}, 
we may write  $[(\tilde{K}_0)^{\times r}, A^*\circ \prod^r] = A[[c_1^A,\ldots,c_i^A,\ldots]]^{\odot r}$.

\subsection{Notation and continuity of operations}\label{sec:cont} 

Recall that by the projective bundle theorem for every oriented theory $A^*$
its value on the product of projective spaces is a quotient of a polynomial ring
generated by the first Chern class of the anti-canonical line bundle.
For the product of projective spaces $z_i^A$ denotes 
$c_1^A(\mathcal{O}(1)_i)$ where $\mathcal{O}(1)_i$ is the pull-back of $\mathcal{O}(1)$ from the $i$-th factor. 
Then we have
$$A^*(\mathbb{P}^{n_1}\times \cdots \times \mathbb{P}^{n_l})= A [z_1^A, \ldots, z_l^A]/((z_1^A)^{n_1+1}, \ldots, (z_l^A)^{n_l+1}).$$

One can form an ind-variety $\mathbb{P}^\infty$ as a formal colimit of linear inclusions of projective spaces.
The value of cohomology theories on it can be formally defined as a limit of the values on projective spaces,
thus, for an oriented theory $A^*$ we obtain
$A^*((\mathbb{P}^\infty)^{\times l}) = A [[z_1^A, \ldots, z_l^A]]$.
If the theory $A$ is clear from the context we shall drop $A$ from the notation of $z_i^A$ for brevity.

We should point out that working with infinite dimensional projective space is
a handy but totally formal convention. However, one needs to be careful with it
as the restriction of every operation to products of (infinite) projective spaces
satisfies the property of continuity which we now explain (following \cite{Vish2}).

Let $G\colon A^*\rarr B^*$ be an operation from a free theory  $A^*$ to an oriented theory $B^*$
which preserves 0.
By Theorem \ref{th:Vish_op} the operation $G$ is determined by maps of pointed sets  
\mbox{$G_{\{l\}}\colon A[[z^A_1,\ldots, z^A_l]]\rarr B[[z^B_1,\ldots, z^B_l]]$}
for all $l\ge 0$.
As $G_{\{l\}}$'s have to commute with pull-backs along partial projections  
the following diagram is commutative for any $l\ge 0$:
\begin{center}
\begin{tikzcd}
A[[z^A_1,\ldots, z^A_l]] \arrow[r, "G_{\{l\}}"] \arrow{d} & B[[z^B_1,\ldots, z^B_l]] \arrow{d} \\
A[[z^A_1,\ldots, z^A_{l+1}]] \arrow[r, "G_{\{l+1\}}"] & B[[z^B_1,\ldots, z^B_{l+1}]]\\
\end{tikzcd}
\end{center}
This allows to use only one transform, the inductive limit of maps $G_{\{l\}}$
$$ G\colon A[[\bar{z}^A]]\rarr B[[\bar{z}^B]],$$
where $A[[\bar{z}^A]]:= \cup_{l\ge 0} A[[z^A_1,\ldots, z^A_l]]$,
and similarly, for $B[[\bar{z}^B]]$.
The map $G$ uniquely determines $G_{\{l\}}$ for any $l$.
Since $G$ preserves 0 we have $G(0)=0$.

Denote by $F^k_A$ an ideal in $A[[\bar{z}^A]]$ of series of degree $\ge k$
($F^k_B$ is defined analogously).

\begin{Prop}[Vishik, {\cite[Prop. 5.3]{Vish2}}]\ \\

Let $P, P'\in A[[\bar{z}^A]]$ be s.t. $P \equiv P' \mod F^k_A$.
Then $$ G(P) \equiv G(P') \mod F^k_B. $$
\end{Prop}

This allows to calculate approximation of $G(P)$ approximating $P$.
In particular, the operation $G$ is determined by its restriction 
to the products of finite-dimensional projective spaces, or equivalently
by the maps
$$G_{r,n}\colon A^*((\mathbb{P}^n)^{\times l})=A[[z^A_1,\ldots, z^A_l]]/F_A^{n+1} 
\rarr B^*((\mathbb{P}^n)^{\times l})=B[[z_1,\ldots, z_l]]/F_B^{n+1}.$$
In other words, maps $G_{\{l\}}$ or map $G$ are determined by their restriction 
to the polynomial rings $A[z^A_1,\ldots, z^A_l] \subset A[[z^A_1,\ldots, z^A_l]]$,
and analogous statements are true for poly-operations as well.

\subsection{General Riemann-Roch theorem}

Let $A^*$ be a free theory, let $B^*$ be an oriented theory,
and let $\phi\colon A^*\rarr B^*$ be an operation.

Let $X$ be a smooth variety and let $i\colon Z\rightarrow X$ be its closed smooth subvariety.
Denote by $G^c_Z$ the composition 
$$A^*(Z)\xrarr{\cdot z_1^A\cdots z_c^A} A^*(Z\times (\mathbb{P}^\infty)^{\times c})\xrarr{\phi} 
B^*(Z\times \mathbb{P}^\infty)^{\times c})\cong B^*(Z)[[z_1^B,\ldots, z_c^B]].$$

If  $\mu_i \in B^*(Z)$ for $i\colon 1\le i \le c$ are nilpotent elements,
we will denote by $G^c_Z|_{z_i^B=\mu_i}$ the composition of $G^c_Z$
with the $B^*(Z)$-linear map $B^*(Z)[[z_1^B,\ldots, z_c^B]]\rarr B^*(Z)$ which sends $z_i^B$ to $\mu_i$.
For a formal group law $F_B$ denote by $\omega_B \in B[[t]]dt$
the unique invariant differential s.t. $\omega_B(0)=dt$.

The following result is a general form of Riemann-Roch-type theorems
for non-additive operations. 

\begin{Th}[Vishik, {\cite[Th. 5.19]{Vish2}}]\label{th:riemann_roch}
Let $\alpha\in A^*(Z)$, 
denote by $\mu_1,\ldots, \mu_c$ the $B$-roots of the normal bundle $N_{Z/X}$.

Let $L_i$ be line bundles over $Z$ for $1\le i\le k$,
 and denote by $x_i=c_1^A(L_i)$, $y_i=c_1^B(L_i)$ their first Chern classes.

Then 
$$ \phi\left(i_* (\alpha \prod_{i=1}^k x_i)\right)
= i_* \Res \frac{G^{c+k}_Z(\alpha)|_{z_i^B=t+_B\mu_i, 1\le i \le c;
z^B_{c+j} = y_j, 1\le j \le k}}
{t\cdot \prod_{i=1}^c (t+\mu_i)} \omega_B. $$
\end{Th}

\subsection{Topological filtration on oriented theories}\label{sec:top_filt}

\begin{Def}\label{def:top_filt}
Let $A^*$ be a presheaf of abelian groups on the category of smooth varieties over
a field $k$. 
Define the {\sl topological} (or sometimes called {\sl support codimensional}) filtration $\tau^\bullet$ 
on the values of $A^*$ on a variety $X$ by the formula:
$$ \tau^i A^*(X) := \cup_{codim_X X\setminus U \ge i} \mbox{Ker\ }(A^*(X)\rarr A^*(U)),$$
where the union goes over all open subvarieties $U$ in $X$ 
with the complement of codimension at least $i$.
\end{Def}

Note that, if $A^*$ is an oriented theory satisfying the (CONST) axiom (e.g.\ a free theory),
then the zero-th graded quotient of the topological filtration 
is split: $\tau^1 A^*=\tilde{A}^*$, $\tau^0A^*/\tau^1A^*=A$, $A^*=\tilde{A}^*\oplus A$ 
(see Section \ref{sec:vishik}).

We summarize well-known properties of the topological filtration in the following proposition.

\begin{Prop}\label{prop:top_filt_properties}
Let $A^*$ be a free theory.
\begin{enumerate}
\item\label{item:topfilt_free_surj}
The canonical map $(\tau^i \Omega^*)\ot_\LL A \rarr \tau^i A^*$ is surjective;
\item\label{item:topfilt_omega} $\tau^i \Omega^n=\sum_{m\le n-i} \LL^{m} \Omega^{n-m}$;
\item\label{item:push_top_filt} if $f\colon X\rarr Y$ is a projective morphism of codimension $c$,
then $f_*\tau^i A^*(X) \subset \tau^{i+c}A^*(Y)$;
\item\label{item:chern_top_filt} Chern class $c_i^A$ takes values in $\tau^iA^*$;
\item\label{item:topfilt_mult} the topological filtration on $A^*$ is multiplicative,
i.e. $(\tau^i A^*)(\tau^j A^*)\subset \tau^{i+j}A^*$;
\item\label{item:CH_surj_grtop} there exist a canonical surjective map of $A$-modules
$\rho_A\colon\CH^\bullet\ot_\ZZ A \rarr gr_\tau^\bullet A^*$,
where $A=A(\Spec k)$ (Notation~\ref{not:A});
\item\label{item:topfilt_projspace} if $X$ is a smooth variety, $E$ is a vector bundle on $X$ 
of rank $r$, $n\ge 0$, $i\ge 0$,
then 
$$\tau^i A^*(\mathbb{P}(E)) = 
\sum_{m=0}^{r-1} \tau^{i-m}A^*(X)\cdot \xi^m, $$

where $\xi = c_1^A (\mathcal{O}(1))$
on $\mathbb{P}(E)$.

In particular, we have
$$\tau^i A^*(X\times (\mathbb{P}^\infty)^{\times n}) = 
\sum_{m\ge 0} \tau^{i-m}A^*(X)\ot_A A[[z_1^A, \ldots, z_n^A]]^{\deg \ge m},$$ 
where $z^A_i$ is the first Chern class of the line bundle $\mathcal{O}(1)_i$ (see Section \ref{sec:cont}).
\end{enumerate}
\end{Prop}
\begin{proof}
(\ref{item:topfilt_free_surj}) Let $\alpha \in A^*(X)$ be supported on 
a closed subvariety $Z$, $j\colon Z \hookrightarrow X$,  of codimension $i$ (i.e. $\alpha \in \tau^i A^*(X)$).
If $Z$ is smooth, then $A^*(Z)=\Omega^*(Z)\ot_\LL A$, 
and there exist $\beta_k \in \Omega^*(Z)$, $a_k\in A$, $k\in J$,
s.t. $\sum_{k\in J} (j_*\beta_k) \ot a_k$ is a lift of $\alpha$ to $(\tau^i \Omega^*(X))\ot_\LL A$.
However, if $Z$ is not smooth, we need to consider Borel-Moore oriented cohomology theories 
to make the argument work, see \cite[Section~2.1]{LevMor}.
More precisely, one extends free theories to presheaves of abelian groups 
on all quasi-projective varieties \cite[Remark~2.1.4]{LevMor}
and proves the (LOC) axiom for them. We leave the details to the reader.

(\ref{item:topfilt_omega}) is \cite[Th. 4.5.7]{LevMor}. Note, however,
that they use homological grading for algebraic cobordism.

(\ref{item:push_top_filt}) follows for algebraic cobordism by (\ref{item:topfilt_omega}),
since pushforwards are $\mathbb{L}$-linear; and for arbitrary free theory by (\ref{item:topfilt_free_surj}).  

(\ref{item:chern_top_filt}) follows from the construction of Chern classes
and from the fact that $c^A_1(L) \in \tau^1 A^*$ for a line bundle $L$.
 The latter property
can be checked using the (CONST) axiom and the fact that every line bundle is trivial in the generic point.

(\ref{item:topfilt_mult}): It follows from (\ref{item:topfilt_free_surj})
that it is enough to consider the case $A^*=\Omega^*$, and the 
statement then follows from the explicit description of the
topological filtration in (\ref{item:topfilt_omega}).

(\ref{item:CH_surj_grtop}) In the case $A^*=\Omega^*$ the map is constructed in \cite[Cor. 4.5.8]{LevMor}.
From (\ref{item:topfilt_free_surj}) it follows, 
that there is a canonical surjective map $(gr^\bullet_\tau \Omega^*)\ot_{\LL} A\rarr gr^\bullet_\tau A^*$.
We define $\rho_A$ as the composition of this map 
with $\rho_\Omega\ot_{\LL} A\colon \CH^*\ot A \rarr (gr^\bullet_\tau \Omega^*)\ot_{\LL} A$.

(\ref{item:topfilt_projspace}): It follows from 
 (\ref{item:chern_top_filt}) and (\ref{item:topfilt_mult}) that the RHS is contained in the LHS.
Let $\alpha \cdot \xi^k \in \tau^i A^*(\mathbb{P}(E))$.
Then by (\ref{item:CH_surj_grtop}) there exist an element $z\in \CH^i(\mathbb{P}(E))\ot A$
which maps to $\alpha \cdot \xi^k$. By the projective bundle theorem for Chow groups
we see that $z=x\cdot \xi^k$ for some $x\in \CH^{i-k}(X)$, and its image in $\tau^i A^*(\mathbb{P}(E))/\tau^{i+1}$
is of the form $\beta \cdot \xi^k$ where $\beta \in \tau^{i-k} A^*(\mathbb{P}(E))$.
Lifting this element to $\tau^i A^*(\mathbb{P}(E))$ we can argue by induction on $i$
that $\alpha \in \tau^{i-k} A^*(\mathbb{P}(E))$. 
\end{proof}

For convenience of readers, we include the proof of the following
well-known statement for which we were not able to find a reference.

\begin{Lm}\label{lm:rho_morphism_of_theories}
Let $B^*$ be a free theory.
Then $gr_\tau^\bullet B^*:=\oplus_{c\ge 0} gr_\tau^c B^*:=\oplus_{c\ge 0} \tau^c B^*/\tau^{c+1} B^*$ 
is an oriented theory 
in the sense of Levine-Morel (\cite[Def. 1.1.2]{LevMor}), i.e.\ it satisfies Definition \ref{goct} 
except for the (LOC) axiom.

Moreover,  the map $\rho_B\colon \CH^\bullet\ot_\ZZ B\rarr gr_\tau^\bullet B^*$ is a morphism of oriented theories,
where $B$ is the coefficient ring of $B^*$ (Notation~\ref{not:A}).
\end{Lm}
\begin{proof}
To prove that $gr_\tau^\bullet B^*$ is an oriented theory
we need to show the following: 
\begin{enumerate}
\item the topological filtration on $B^*$ is respected by pullbacks;
\item if $f\colon X\rarr Y$ is a projective morphism of codimension $c$, then $f_* \tau^i B^*(X) \subset \tau^{c+i} B^*(Y)$ 
for all $i\ge 0$;
\item $gr_\tau^\bullet B^*$ satisfies the projective bundle theorem (PB);
\item If $p\colon E\rarr X$ is a vector bundle on $X$, then $p^*$ strictly preserves the topological filtration.
\end{enumerate}

(1) and (2) allow to define the structure of push-forward and pullbacks on $gr_\tau^* B^\bullet$.
Properties (A1), (A2) and the projection formula for $gr_\tau^\bullet B^*$ 
 follow directly from the corresponding properties of $B^*$.
(PB) and (EH) follow from (3) and (4) above, respectively.

Properties (1) and (2) follow from Prop.~\ref{prop:top_filt_properties}, (\ref{item:topfilt_omega}), (\ref{item:push_top_filt}).
 To show property (4) it suffices to treat the case of $B^*=\Omega^*$ 
due to Prop. \ref{prop:top_filt_properties}, (\ref{item:topfilt_free_surj}),
and it follows from Prop.~\ref{prop:top_filt_properties} (\ref{item:topfilt_omega})
since $p^*\colon\Omega^*(X)\rarr \Omega^*(E)$ is an isomorphism of $\LL$-modules.
Also, (3) follows from Prop. \ref{prop:top_filt_properties} (\ref{item:topfilt_projspace}).

It is left to show that the map $\rho_B$ commutes with pullbacks and push-forwards.
Again it is enough to treat the case $B^*=\Omega^*$.
However, the map $(\rho_\Omega)^{\deg i}\colon\CH^i \rarr \tau^i\Omega^i/\tau^{i+1}\Omega^i=\Omega^i/\tau^{i+1}\Omega^{i+1}$
 is the inverse of the isomorphism of theories $\Omega^* \ot_\LL \ZZ \cong \CH^*$, 
and commutes with push-forwards and pullbacks. The whole map $\rho_\Omega$
is just the $\LL$-linearisation of this map, and since pullbacks and push-forwards
are morphisms of $\LL$-modules we obtain the claim.
Similarly, it follows that $\rho$ is multiplicative.
\end{proof}
\begin{Rk}\label{rem:gr_oriented}
Even though $gr^\bullet_\tau \Omega^*$ is an oriented theory in the sense of Levine-Morel,
it does not satisfy the localization axiom\footnote{Vishik kindly provided a counter-example
to this statement via personal communication.} and therefore is not an oriented theory (Def. \ref{goct}).
In particular, Vishik's results on classification of operations (Th. \ref{th:Vish_op})
 can not be applied to operations with the target theory $gr^\bullet_\tau B^*$ where $B^*$ is a free theory.
\end{Rk}

\subsection{The gamma-filtration on $\KK$}

Chern classes $c^{\KK}_i$ from $\KK$ to $\KK$ are closely related 
to $\lambda$-operations (\cite[Exp. 0]{Gro}), which are more well-studied in the literature.
In particular, Theorem \ref{basis} actually can be restated as saying that all endo-operations of $\KK$
are freely generated by $\lambda$-operations. Nevertheless we prefer to work with Chern classes.
The gamma filtration is defined by the following formula:

\begin{equation}
\gamma^i \KK(X) := < c^{\KK}_{i_1}(\alpha_1)\cdots c^{\KK}_{i_k}(\alpha_k)| \sum_j i_j\ge m, \alpha_j \in \KK(X)>.
\end{equation}

We summarize well-known properties of the gamma filtration below,
so that thereafter we could compare them with analogous statement about the gamma filtration
on Morava K-theories.  We do not provide proofs, as we will not use these results,
however, proofs can be obtained in the same manner
as it will be done for Chern classes from $\Kn^*$ in Section \ref{sec:morava_gamma_filtration}.

\begin{Prop}\label{prop:gamma_K0}
\phantom{a}
\begin{enumerate}
\item $\gamma^i\KK \subset \tau^i\KK$ for $i\ge 0$;
\item $\gamma^1=\tau^1$, $\gamma^2=\tau^2$;
\item $\gamma^i\ot\QQ = \tau^i\ot\QQ$ as filtrations on $\KK\ot\QQ$;
\item The $i$-th Chern class $c_i^{\CH}\colon\KK\rarr \CH^i$ induces additive morphisms
$$c^{\CH}_i\colon gr^i_\tau \KK \rarr \CH^i, \quad c^{\CH}_i\colon gr^i_\gamma \KK \rarr \CH^i;$$
\item $c^{\CH}_i\ot id_{\QQ}$ yields an isomorphism between $gr^i_\gamma \KK\ot \QQ$
(or $gr^i_\tau \KK\ot \QQ$) and $\CH^i\ot \QQ$;
\item $c^{\CH}_1\colon gr^1_\gamma \KK \rarr \CH^1$ is an isomorphism,
and $c^{\CH}_2\colon gr^2_\gamma \KK \rarr \CH^2$ is surjective.
\item $c^{\CH}_i\colon gr^i_\gamma \KK\ot\Zp \rarr \CH^i\ot\Zp$ is surjective for $i\le p$.
\end{enumerate}
\end{Prop}

The gamma filtration has a description that does not use the existence or particular properties of Chern classes.
Namely, the gamma filtration is the best `operational' approximation of the topological filtration.

\begin{Prop}
For any $m\ge 0$
$$\gamma^m \KK(X) := < \phi(\alpha_1,\ldots, \alpha_k) | \phi\in [(\KK)^{\times k},\tau^m \KK \circ \prod^k], \alpha_i \in \KK(X)>,$$

i.e.\ the $m$-th part of the gamma filtration is generated by the image of all internal poly-operations 
whose codomain is the $m$-th part of the topological filtration on $\KK$.
\end{Prop}

\section{Truncation of operations and the topological filtration}\label{sec:trunc}

The Chern class $c^B_i$ considered as an operation from $\KK$ 
to some oriented theory $B^*$ is an example of an operation
which takes its values in $\tau^iB^*$, the $i$-th part of the topological filtration (Prop. \ref{prop:top_filt_properties},
(\ref{item:chern_top_filt})).
There is also a surjective additive map $\rho_B\colon\CH^i\ot B\rarr \tau^iB^*/\tau^{i+1}B^*$ (Prop. \ref{prop:top_filt_properties},
(\ref{item:CH_surj_grtop})),
and one can easily check that $\rho_B\circ c_i^{\CH}= c_i^B$ as operations to $\tau^iB^*/\tau^{i+1}B^*$.
Thus, $c_i^{\CH}$ is a lift of $c_i^{B} \mod \tau^{i+1}$ along the map $\rho_B$.

The goal of this section is to provide a construction of such a lift for all operations.
Namely, for an operation $\phi\colon A^*\rarr \tau^i B^*$ from a free theory to an oriented theory
we construct its truncation $tr_i \phi\colon A^*\rarr \CH^i\ot B$
s.t. $\rho_B\circ (tr_i\phi)=\phi$ as operations to $\tau^i B^*/\tau^{i+1}B^*$.
Moreover, the truncation map gives an inclusion of $B$-modules
$[A^*, \tau^i B^*]/[A^*, \tau^{i+1}B^*]\subset [A^*, \CH^*\ot B]$.
The problem of calculating
all operations from $A^*$ to $B^*$ then may be solved in two steps: 1. calculating all operations
from $A^*$ to $\CH^*\ot B$, 2. calculating images of truncation maps.

In Section \ref{sec:op_topfilt} we characterize operations from $A^*$ to $B^*$
which take values in $\tau^i B^*$ in terms of their action on products of projective spaces.
In Section \ref{sec:tr_constr} we explain the construction of the truncation
map and the action of the truncation of an operation on products of projective spaces. 
In short, series $G_{\{l\}}\in B[[z_1,\ldots, z_l]]$ which determine the operation (see Section \ref{sec:cont})
are truncated to polynomials of degrees $i$ with respect to variables $z_j$ 
by forgetting the part of these series of higher degree.
Section \ref{sec:trunc_mod} contains a variation of the truncation construction for operations between $p$-local theories.
All this is based on Vishik's Theorem \ref{th:Vish_op}. 
In Section \ref{sec:all_op_from_trunc} we show that if truncation maps are isomorphisms,
and the $B$-module $[A^*, \CH^*\ot B]$ is free, then lifts of (either $B$-module or $B$-algebra) generators 
of the latter are generators of $[A^*, B^*]$.

The results about truncations are one of the main technical tools for the rest of the paper.
In Section~\ref{sec:proof_main} we will apply them to the case when $A^*$ is the $n$-th Morava K-theory $\Kn^*$,
and $B^*$ is a $p^n$-typical oriented theory. Note that the problem of classifying all operations
from $\Kn^*$ to $\CH^*\ot\Zp$ was already solved in \cite{Sech}.

\subsection{Operations which target the $j$-th part of the topological filtration}\label{sec:op_topfilt}

\begin{Prop}\label{prop:operations_from_tilde}
Let $A^*$, $B^*$ be oriented theories,
and assume that $A^*$ satisfies the (CONST) axiom (see Def. \ref{def:const_axiom}).
Let $A$ denote the coefficient ring of $A^*$ (Notation~\ref{not:A}).

Then there are natural isomorphisms of the following sets:
$$ [A^*, B^*] = \prod_A [\tilde{A}^*, B^*], \qquad [A^*, B^*]^{add} = \Hom(A,B)\oplus [\tilde{A}^*, B^*].$$
\end{Prop}
\begin{proof}
As $A^*=A\oplus \tilde{A}^*$ by the (CONST) property,
there is a natural map $\prod_A [\tilde{A}^*, B^*] \rarr [A^*,B^*]$
which sends a tuple of operations $(\phi_a)_{a\in A}$
to an operation $(a,x)\rarr \phi_a(x)$. 
The inverse map is just the restriction of an operation
 to subspaces $a\oplus \tilde{A}^*$ for all $a\in A$.

The case of additive operations can be treated similarly.
\end{proof}

For an oriented theory $A^*$ satisfying (CONST) axiom 
the first piece of the topological filtration $\tau^1 A^*$ equals to $\tilde{A}^*$.
Thus, the proposition above simplifies a description of operations from $A^*$.
Note that since any operation which preserves zero
respects topological filtration, we see that classifying operations
from $A^*$ to $\tau^j B^*$, $j\ge 1$, is equivalent to classifying operations from
 $\tau^1 A^*=\tilde{A}^*$ to $\tau^j B^*$.
Thus, dealing with $\tilde{A}^*$ instead of $A^*$ in the following proposition
does not reduce the generality of the statement.

\begin{Prop}\label{prop:operations_top_COT}
Let $A^*$ be a free theory, 
 $B^*$ be an oriented theory  
 and let $\phi\colon\tilde{A}^*\rarr B^*$ be an operation.
Then $\phi$ takes values in $\tau^i B^*$ for some $i>0$
if and only 
if its restriction to products of projective spaces 
does. 
\end{Prop}
\begin{proof}
The only if part is straight-forward. For the converse let $i>0$, 
and assume that $\phi$ takes values in $\tau^i B^*$ on products of projective spaces.
Note that in particular this means that $\phi$ is zero over a point,
i.e.\ $\phi$ preserves zero. 

There exists a free theory $\Omega^*\ot_{\LL, F_B} B$
and the canonical morphism of theories $\Omega^*\ot_{\LL, F_B} B \rarr B^*$ 
which preserves the topological filtration.
Since an operation from $A^*$ to $B^*$ lifts to the free theory $\Omega^*\ot_{\LL, F_B} B$
by Vishik's Theorem \ref{th:Vish_op},
we may assume without loss of generality that $B^*$ is also a free theory.

We follow Vishik's proof of the Th. \ref{th:Vish_op} (\cite[5.4]{Vish2}) in which 
the operation is reconstructed on arbitraty smooth variety by induction
on its dimension. The key ingredient in this argument is the Riemann-Roch theorem (Th. \ref{th:riemann_roch})
for closed smooth subvarieties.

The idea of the construction of an operation by induction is the following:
if $X$ is a variety of dimension $d+1$, and $\alpha\in \tilde{A}^*(X)$
is supported on a smooth divisor $i\colon D\hookrightarrow X$,
then $\phi(i_*(\alpha) z_1\cdots z_n)\in B^*(X\times (\mathbb{P}^\infty)^{\times n})$ 
can be calculated by the Riemann-Roch formula
using the value of $\phi$ on a product of projective spaces with $D$, while the latter is a variety of dimension $d$.
There are, however, two technical issues which make the following proof a little bit cumbersome.
 First, the divisor $D$ does not have to be smooth,
and if it is not, the Riemann-Roch formula can not be applied.
Second, the fact that $\phi$ sends additive generators of $A^*(X\times (\mathbb{P}^\infty)^{\times n})$
to the $j$-th part of the topological filtration does not directly imply
that $\phi$ sends every element to this part of the filtration as the operation
does not have to be additive. To deal with the first problem one uses resolutions of singularities,
and the second one is dealt with by considering derivatives of the operations.

{\bf Induction assumption.} For an integer $d\ge 0$,
every smooth variety $X$ of dimension not greater than $d$ and every number $n\ge 0$ 
the map
\begin{center}
\begin{tikzcd}
\phi\colon A^*(X\times (\mathbb{P}^\infty)^{\times n}) \arrow{r} & B^*(X\times (\mathbb{P}^\infty)^{\times n})
\end{tikzcd}
\end{center}
takes values in $\tau^i B^*(X\times (\mathbb{P}^\infty)^{\times n})$.

{\bf Base of induction} ($d=0$) is the assertion of the Proposition for the base field,
and the case of finite extension of the base field 
clearly reduces to the assertion with the help of the (CONST) and the (PB) properties.

{\bf Induction step} ($d\rightarrow d+1$).
Let $X$ be a smooth variety of dimension $d+1$.

Every element of $\tilde{A}^*(X\times (\mathbb{P}^\infty)^{\times n})$
can be represented as a sum of two elements $\alpha$ and $a$
where $\alpha$ lies in the subgroup $\tilde{A}^*(X) [[z_1,\ldots, z_n]]$
and $a\in A[[z_1,\ldots, z_n]]^{\deg \ge 1}$. Here $A[[z_1,\ldots, z_n]]$
is considered as a subgroup of $A^*(X\times (\mathbb{P}^\infty)^{\times n})$ 
via the map $p_{X,n}^*$,
where $p_{X,n}\colon X\times (\mathbb{P}^\infty)^{\times n}\rarr (\mathbb{P}^\infty)^{\times n}$ is the projection morphism. 

By the definition of a derivative we have $\phi(a+\alpha)=\phi(a)+\phi(\alpha)+\pd^1\phi(a,\alpha)$,
and we are going to show that each of the latter three summands lies in $\tau^i B^*(X\times (\mathbb{P}^\infty)^{\times n})$.

1. The value of $\phi$ on  $a\in A[[z_1,\ldots, z_n]]$ 
is equal to $p_{X,n}^*\phi_{(\mathbb{P}^\infty)^{\times n}}(a)$.
By the base induction assumption $\phi(a)$ lies in $(z_1^B,\ldots, z_n^B)^i$
and therefore $p_{X,n}^* \phi(a) \in (z_1^B,\ldots, z_n^B)^i \subset \tau^i B^*(X\times (\mathbb{P}^\infty)^{\times n})$.

2.  For $\alpha \in \tilde{A}^*(X) [[z_1,\ldots, z_n]]$,
there exist a divisor $D$ in $X$, s.t. $\alpha$
restricts to $0$ over $X\setminus D \times (\mathbb{P}^\infty)^{\times n}$.
By the result of Hironaka \cite{Hiro} (recall that we always work in characteristic $0$ in this paper)
 there exist a resolution 
of singularities of $D$ inside $X$,
i.e.\ a birational morphism $p\colon\tilde{X}\rarr X$,
which is an isomorphism outside of $D$ 
and the preimage of $D$ is a divisor $E$ with strict normal crossings.

Let us show that $\phi(p^*(\alpha))$
lies in the $i$-th part of the topological filtration 
of $B^*(\tilde{X}\times (\mathbb{P}^\infty)^{\times n})$.
Indeed, $p^*\alpha$ restricts to zero in the complement to $p^{-1}(D)\times (\mathbb{P}^\infty)^{\times n}$,
and therefore 
there exists an element $\beta\in A^*(E\times (\mathbb{P}^\infty)^{\times n}))$ 
s.t. $p^*\alpha=f_* \beta$ where 
 $f\colon E\times (\mathbb{P}^\infty)^{\times j})\rarr \tilde{X}\times (\mathbb{P}^\infty)^{\times n})$ 
and $E\subset X$ is a divisor with 
strict normal crossings. 
Denote by $J$ the set of irreducible components of $E$, i.e.\ for $r\in J$ the divisor $E_r\subset E$
is smooth,
and for $S\subset J$ denote by $E_S$ the intersection of components $E_s$ where $s\in S$.
Note that by the definition of a divisor with simple normal crossing $d_S\colon E_S\hookrightarrow X$ 
is a closed smooth subvariety of codimension $|S|$. 

There is a well-defined push-forward map from the values on irreducible components of $E$ to the value on $E$:
$\oplus_r A^*(E_r\times (\mathbb{P}^\infty)^{\times n}))\rarr A^*(E\times (\mathbb{P}^\infty)^{\times n}))$ 
which is a surjection (see e.g.\ \cite[Section 2.2]{Vish1}),
and,
thus, we may assume that $p^*\alpha = \sum_{r\in J} (d_r)_* \beta_r$ where $\beta_r \in A^*(E_r\times (\mathbb{P}^\infty)^{\times n})$.

\begin{Lm}[{\cite[Formula (10)]{Vish2}}]\label{lm:riemann_roch_snc}
\begin{equation}\label{eq:riemann_roch_snc}
 \phi(p^*\alpha) = \sum_{S\subset J} (d_S)_* \Res
 \frac{\pd^{|S|-1} \phi(z_r^A j_r^*\beta_r, r\in S)|_{z_r^B=t+_B j_r^*\lambda_r}}{t\cdot \prod_{r\in S} (t+_B j_r^*\lambda_r)}
 \omega_t^B,
\end{equation}
where $\lambda_r=c_1^B(\OO(E_r))\in B^*(E_r)$ and $j_r\colon E_S\rarr E_r$ is a closed embedding.
\end{Lm}
We provide the proof of this Lemma for the sake of completeness. Note that the statement differs from \cite{Vish2} only in notation
and that there it is used to construct an operation, while we are dealing with an already existing operation $\phi$.
\begin{proof}[Proof of the Lemma.]

Using Discrete Taylor Expansion (Prop. \ref{prop:taylor}) 
we have 
$$\phi(p^*\alpha)=\sum_{S\subset J} \pd^{|J|-1}\phi((d_r)_* \beta_r, r\in S).$$
The internal derivative $\pd^{|J|-1}\phi$ on $X$ is equal
to $\Delta_X^* \eth^{|J|-1}\phi$ on $X^{\times |J|}$
where $\Delta_X\colon X\hookrightarrow X^{\times |J|}$ is the diagonal.
On the other hand by the Riemann-Roch formula for external poly-operations
applied to each of the `variables' separately
we obtain 
$$\eth^{|J|-1}\phi((d_r)_* \beta_r, r\in S)=
(\times d_r, r\in S)_* \Res \frac{\eth^{|J|-1}\phi(\beta_r z_r^A, r\in S)|_{z_r^A=t+_B \lambda_r}}{t\cdot \prod_{r\in S} (t+_B \lambda_r)}
\omega_t^B\in B^*(X^{\times |J|-1}).$$

In order to get rid of external derivatives note
 that the following square is a transversal cartesian square
(see Section \ref{sec:def} for the definition):
\begin{center}
\begin{tikzcd}
E_S \arrow[r, "d_S"] \arrow[d, "j_S"] & X \arrow[d, "\Delta_X"] \\
\prod_{r\in S} E_r \arrow[r,"{\times d_r, r\in S}"] & X^{\times |J|-1}\\
\end{tikzcd}
\end{center}
and, thus, by axiom (A2) for the theory $B^*$ 
we have $\Delta^*_X (\times d_r, r\in S)_*=(d_S)_*(j_S)^*$.
Applying this equality 
to the expression of $\phi(p^*\alpha)$ in terms of external derivatives on products of divisors
we obtain that 
$$\phi(p^*\alpha) = \sum_{S\subset J} (d_S)_*  j_S^* \Res \frac{\eth^{|J|-1}\phi(\beta_r z_r^A, r\in S)|_{z_r^A=t+_B \lambda_r}}{t\cdot \prod_{r\in S} (t+_B \lambda_r)}\omega_t^B.$$

Poly-operation commute with pullbacks, 
thus factoring $j_S^*=\Delta_{E_S}^*\circ (\times j_r)^*$
we see that the following holds: 
$j_S^* \eth^{|J|-1}\phi(\beta_r z_r^A, r\in S)= 
\Delta_{E_S}^* \eth^{|J|-1}\phi(z_r^A j_r^* \beta_r, r\in S)$,
and the latter is actually an internal derivative $\pd^{|J|-1}\phi(z_r^A j_r^* \beta_r, r\in S)$.
This finishes the proof of the Lemma.
\end{proof}

To conclude that $\phi(p^*\alpha)$ lies in the $i$-th part of the topological filtration
 we need to deal with fractions in the RHS of the equation (\ref{eq:riemann_roch_snc}). 
This is done in the following Lemma using the induction assumption.

\begin{Lm}\label{lm:induction_pd_top}
Let $\phi$ be an operation from $A^*$ to $B^*$, let $i\ge 1$, $d\ge 0$,
and assume that for each variety of dimension less than $d+1$ 
we have $\phi(A^*(X\times (\mathbb{P}^\infty)^{\times n})))\subset \tau^i B^*(X \times (\mathbb{P}^\infty)^{\times n}))$
for each $n\ge 0$.

Let $r\ge 1$ and let $D$ be a smooth variety of dimension less than $d+1$,
let $\beta_i\in A^*(D)$, $\lambda_i \in B^*(D)$ for $i\colon 1\le i \le r$.
Then  
$$\Res \frac{(\pd^{|S|-1}\phi)(z^A_r \beta_r, r\in S)|_{z_r^B=t+_B \lambda_r}}
{t\cdot \prod_{r\in S} (t+_B \lambda_r)}\omega^B_t \in \tau^{i-|S|} B^*(D).$$
\end{Lm}
\begin{proof}[Proof of the Lemma]
We need to show that $(\pd^{|S|-1}\phi)(z^A_r \beta_r, r\in S)$
equals to $F(\beta_r, r\in S) z_1^B\cdots z_{|S|}^B$
for some $F(\beta_r, r\in S) \in (\tau^{i-|S|}B^*(D))[[z_1^B,\ldots, z_{|S|}^B]]$.
Then the expression in question will be equal
to $\Res \frac{F(\beta_r, r\in S)|_{z_i^B=t+_B \lambda_i} \omega_t^B}{t}$,
i.e.\ to $F(\beta_r, r\in S)|_{z_i^B=\lambda_i}$.

We know that $(\pd^{|S|-1}\phi)(z^A_r \beta_r, r\in S)$
is divisible by $z_1^B\cdots z_{|S|}^B$ which is an instance of
continuity of operations (see Section \ref{sec:cont}).
On the other hand by definition the derivative 
\mbox{$(\pd^{|S|-1}\phi)(z^A_r \beta_r, r\in S)$} is equal 
to $\sum_{I\subset S} (-1)^{|I|} \phi(\sum_{i\in I} z_i^A \beta_i)$
and the only summand of it which can be divisible by a product of variables $z^B_r$
is $\phi(\sum_{r\in S} \beta_r z_r^A)$, because others do not contain all variables $z_i^B$ together
and therefore are cancelled.

By the inductive assumptions we have
$\phi(\sum_{r\in S} \beta_r z_r^A)\in \tau^i B^*(D\times (\mathbb{P}^\infty)^{\times |S|}))$
and from Prop.~\ref{prop:top_filt_properties},~(\ref{item:topfilt_projspace})
it follows that those summands of $\phi(\sum_{r\in S} \beta_r z_r^A)$
which are divisible by  $z_1^B\cdots z_{|S|}^B$ have coefficients
in $\tau^{i-|S|}B^*(D)$. 
\end{proof}

Thus, we have proved $p^*\phi(\alpha)\in \tau^i B^*(\tilde{X})$.
By projection formula we have $p_*p^* \phi(\alpha) = \phi(\alpha) p_* 1_{\tilde{X}}$.
However, by Prop.~\ref{prop:top_filt_properties}~(\ref{item:push_top_filt})
we have that $\phi_*$ preserves the topological filtration,
this filtration is multiplicative, and from the (LOC) it easily follows 
that $p_* 1_{\tilde{X}}$ is an invertible element in $B^*(X)$.
Therefore $\phi(\alpha)$ also lies in $\tau^i B^*(X)$.

3. For $a\in A[[z_1,\ldots, z_n]]$, $\alpha \in \tilde{A}^*(X) [[z_1,\ldots, z_n]]$
we need to show that $\pd \phi(a,\alpha)$ lies in 
$\tau^i B^*(X \times (\mathbb{P}^\infty)^{\times n}))$.
The proof is very similar to that of 2, so we only provide a sketch.

Denote by $\Delta\colon X\times (\mathbb{P}^\infty)^{\times n} \rarr X\times (\mathbb{P}^\infty)^{\times 2n}$
the morphism $\id_X\times \Delta_{(\mathbb{P}^\infty)^{\times n}}$
where $\Delta_{(\mathbb{P}^\infty)^{\times n}}$ is the diagonal morphism.
Then $\pd \phi(a,\alpha) = \Delta^* \eth \phi (a,\alpha)$,
and if $\eth \phi (a,\alpha)\in \tau^i B^*(X\times (\mathbb{P}^\infty)^{\times 2n}))$
the claim follows.

To study $\eth \phi (a,\alpha)$ we can apply the Riemann-Roch formula,
and acting as in 2 above we see that everything reduces
to the study of $\pd^{|S|} \phi(a, z_r^A \beta_r, r\in S)$
for some $\beta_r \in A^*(D)$, $r\in S$ and smooth variety $D$ of dimension less than $d+1$.
Similar to Lemma \ref{lm:induction_pd_top} one
then shows that \mbox{$\pd^{|S|} \phi(a, z_r^A \beta_r, r\in S)$}
equals to $F(a, \beta_r) z_1^B\cdots z_{|S|}^B$
where $F(a, \beta_r)$ lies in 
\mbox{$\tau^{i-|S|} B^*((\mathbb{P}^\infty)^{\times n})\times D\times (\mathbb{P}^\infty)^{\times |S|})$.}
(Indeed, one has to consider only $\phi(a+\sum \beta_r z_r^A)$
and the coefficient of $z_1^B\cdots z_{|S|}^B$ lies in $\tau^{i-|S|}$ by induction assumptions.)  

This finishes the inductive step and the proof of the Proposition.
\end{proof}

\begin{Rk}
A similar statement with a similar proof holds for poly-operations.
\end{Rk}

\subsection{Construction of the truncation of an operation}\label{sec:tr_constr}

Let $\mathbb{P}_n:= (\mathbb{P}^\infty)^{\times n}$
 be a product of $n$ copies of an infinite-dimensional projective space.
Let $B^*$ be an oriented theory.
Denote by $\pi_n$ the morphism of $B$-algebras 
$B^*(\mathbb{P}_n)\rarr (\CH^*\ot B)(\mathbb{P}_n)$
which sends $z_i^B$ to $z_i^{\CH}$ (see Section \ref{sec:cont} for the notation). 
Denote by $\pi_n^c$ the composition of $\pi_n$ 
with the projection to $(\CH^c\ot B)(\mathbb{P}_n)=B[z_1^{\CH},\ldots, z_n^{\CH}]^{\deg =c}$.

\begin{Prop}\label{prop:trunc_op}
Let $A^*$ be a free theory and let $B^*$ be an oriented theory.
Let $\phi\colon\tilde{A}^*\rarr \tau^c B^*$ be an operation for some $c\ge 1$.

Then there exist an operation $tr_c \phi\colon\tilde{A}^*\rarr \CH^c\ot B$,
s.t.\ for any $n\ge 0$ and any $\alpha \in \tilde{A}^*(\mathbb{P}_n)$
\begin{equation}\label{eq_trunc}
tr_c \phi (\alpha) = \pi_n^c (\phi (\alpha)).
\end{equation}

This defines a map between groups of operations called {\bf truncation}
$tr_c\colon[\tilde{A}^*,\tau^c B^*]\rarr [\tilde{A}^*, \CH^c\ot B]$
and induces an inclusion 
$$ gr^c_\tau[\tilde{A}^*, B^*]=[\tilde{A}^*,\tau^c B^*]/[\tilde{A}^*, \tau^{c+1}B^*]
 \xhookrightarrow{tr_c} [\tilde{A}^*,\CH^c\ot B]. $$
\end{Prop}
\begin{proof}
Due to Vishik's classification of operations (Th. \ref{th:Vish_op})
the formula (\ref{eq_trunc}) defines an operation from $\tilde{A}^*$ to $\CH^c\ot B$
if these maps respect pull-backs along several types of morphisms between $\mathbb{P}_n$'s.
This follows from the next Lemma.

\begin{Lm}\label{lm:trunc_correct}
Let $f\colon\mathbb{P}_n\rarr \mathbb{P}_{m}$ be one of the morphisms
between products of projective spaces appearing in the list of Th. \ref{th:Vish_op}.
Then for any $c\ge 1$ the following diagram is commutative:
\begin{center}
\begin{tikzcd}
\tau^cB^*(\mathbb{P}_m) \arrow[r, "\pi_n"] \arrow[d,"f_{B^*}^*"]  & (\CH^c\ot B)(\mathbb{P}_m) \arrow[d, "f_{\CH^*\ot B}^*"] \\
\tau^cB^*(\mathbb{P}_n) \arrow[r, "\pi_m"] & (\CH^c\ot B)(\mathbb{P}_n).
\end{tikzcd}
\end{center}
\end{Lm}
\begin{proof}
We consider only the case when $f\colon\mathbb{P}_n\rarr \mathbb{P}_{n+1}$ is a partial Segre embedding
acting on the last two components,
other cases can be treated similarly.

The pull-backs along Segre maps: $f^*_{B^*}$ sends $z^B_{n}$ to $F_B(z^B_n, z^B_{n+1})$,
$f^*_{CH^*\ot B}$ sends $z_n$ to $z_n+z_{n+1}$. We can check the commutativity of the diagram
on a monomial $b (z_1^B)^{r_1}\cdots (z_n^B)^{r_n}$ where $b\in B$, $r_j\ge 0$, $\sum r_j =c$:

\begin{center}
\begin{tikzcd}
   b (z_1^B)^{r_1}\cdots (z_n^B)^{r_n} \arrow{r} \arrow{d} &  b (z_1^{\CH})^{r_1}\cdots (z_n^{\CH})^{r_n} \arrow{d} \\
   b (z_1^B)^{r_1}\cdots (F_B(z_n^B, z_{n+1}^B))^{r_n} \arrow{r} & b (z_1^{\CH})^{r_1}\cdots (z_n^{\CH}+z_{n+1}^{\CH})^{r_n} \\
\end{tikzcd}
\end{center}

As $F_B(z^B_n, z^B_{n+1}) \equiv z^B_n + z^B_{n+1} \mod (z^B_nz^B_{n+1})$
the claim is checked by a straight-forward computation.
\end{proof}

The topological filtration on the oriented theory $B^*$
induces a decreasing filtration 
on the set of all operations from the theory $\tilde{A}^*$ to $B^*$.
We denote the graded factors of it by $gr^i_\tau [\tilde{A}^*,B^*]$.
The truncation map $tr_c$ sends an operation to zero
whenever the values of operation on the products of projective
spaces lie in $\tau^{c+1} B^*$, which is the same
as the operation taking values in $\tau^{c+1} B^*$ for all varieties by Prop. \ref{prop:operations_top_COT}.
It proves the last claim of the Proposition.
\end{proof}

\begin{Exl}
It is easy to see that $c_i^A$ takes values in $\tau^i A^*$,
and $tr_i c_i^A= c_i\ot\id\colon\KK\rarr \CH^i\ot A$.

One can check this using the construction of Chern classes in 
an arbitrary oriented theory,
or by calculating 
 the action of Chern classes on products of projective spaces.
Namely, it is enough to show that series $c_i^A(z_1\cdots z_j)$ 
have degree at least $i$ for any $j\le i$,
and that its $i$-th degree summand does not depend
on $A^*$ (i.e. on the formal group law of $A^*$).
We leave this to the reader.
\end{Exl}

\begin{Prop}\label{prop:op_mod_tau_vs_trunc}
Let $\phi\colon\tilde{A}^*\rarr \tau^c B^*$
be an operation between free theories.

Then the following diagram of presheaves is commutative:
\begin{equation}\label{diag:tr}
\begin{tikzcd}
\tilde{A}^* \arrow[r, "\phi"] \arrow[d, "tr_c \phi"] & \tau^c B^* \arrow{d} \\
\CH^c \ot B \arrow[r, "\rho_B"]  & \tau^c B^*/\tau^{c+1} B^*\\
\end{tikzcd} 
\end{equation}
where the map $\rho_B$ is defined in Prop.~\ref{prop:top_filt_properties},~(\ref{item:CH_surj_grtop}).
\end{Prop}
\begin{proof}
The proof goes by induction procedure 
which is analogous to the one in the proof of Prop.~\ref{prop:operations_top_COT}.

{\bf Induction assumption (for $d\ge 0$)}. 
The diagram (\ref{diag:tr}) commutes for values on $X\times (\mathbb{P}^\infty)^{\times n}$
for every smooth variety $X$ of dimension not greater than $d$, and for every $n\ge 0$.

Base of induction follows by the construction of the truncation of an operation
and the property $(LOC)$.

{\bf Induction step}. Let $a\in \tilde{A}^*(X)$.
Then using Hironaka's resolution of singularities there exist a smooth variety $\tilde{X}$
and a birational morphism $p\colon\tilde{X}\rarr X$
s.t. $p^*(a)$ is supported on a divisor $E$ with smooth normal crossings.
Note that $p^*$ is injective for every free theory. Since $p_*p^*(a)=p_*(1) a$ and $p_*(1)$ is an invertible element,
it follows from Prop.~\ref{prop:top_filt_properties} \ref{item:push_top_filt} that
$p^*$ strictly preserves the topological filtration. Thus, the corresponding map $p^*$ is also
injective for $\tau^c B^*/\tau^{c+1} B^*$, and to prove the commutativity
of (\ref{diag:tr}) suffices to look at the values on $\tilde{X}$.

Denote by $J$ the set of irreducible components of $E$, i.e.\ for $r\in J$ the divisor $E_r\subset E$
is smooth, and for $S\subset J$ denote by $E_S$ the intersection of components $E_s$ where $s\in S$.
Denote by $d_S\colon E_S\hookrightarrow X$ the inclusion of a closed smooth subvariety of codimension $|S|$. 
Then $p^*(a)= \sum_r (d_r)_*(\beta_r)$, and it is enough to check the commutativity
of (\ref{diag:tr}) for values on $p^*(a)$.
For a subset $S\subset J$ denote by $\Theta(\phi;p^*(a);S)$ the element $\Res
 \frac{\pd^{|S|-1} \phi(z_r^A j_r^*\beta_r, r\in S)|_{z_r^B=t+_B j_r^*\lambda_r}}{t\cdot \prod_{r\in S} (t+_B j_r^*\lambda_r)}
 \omega_t^B$, and similarly $\Theta(tr_c \phi;p^*(a);S)$.

Then by Lemma \ref{lm:riemann_roch_snc}
we have

\begin{equation}\label{eq:phi_p_alpha}
 \phi(p^*\alpha) = \sum_{S\subset J} (d_S)^B_* \Theta(\phi;p^*(a);S), 
\qquad (tr_c \phi)(p^*\alpha) = \sum_{S\subset J} (d_S)_*^{\CH} \Theta(tr_c \phi;p^*(a);S).
\end{equation}

Let us show that $\Theta(\phi;p^*(a);S)  \mod \tau^{c+1-|S|} \equiv \rho_B (\Theta(tr_c \phi;p^*(a);S))$.
Since $\rho_B$ commutes with push-forwards by Lemma \ref{lm:rho_morphism_of_theories} and $(d_S)^B_*$
increases the topological filtration by $|S|$,
it would follow that $(d_S)^B_* \Theta(\phi;p^*(a);S) \mod \tau^{c+1} \equiv 
\rho_B (d_S)^{\CH}_*(\Theta(tr_c \phi;p^*(a);S))$. Thus, dealing 
with the each summand of $\phi (p^* \alpha)$ in the formula (\ref{eq:phi_p_alpha})
we will prove the commutativity of (\ref{diag:tr}) on the element $p^*(\alpha)$.

Recall that internal derivatives are defined using values of the operation (Section \ref{subsec_poly}),
and therefore by induction assumption 
we have the following relation in $\tau^c B^*(E_S \times (\mathbb{P}^\infty)^{\times |S|})/\tau^{c+1}$:
\begin{equation}\label{eq:pd_rho}
\pd^{|S|-1} \phi(z_r^A j_r^*\beta_r, r\in S)\mod \tau^{c+1} =
\rho_B \left(\pd^{|S|-1} (tr_c\phi)(z_r^A j_r^*\beta_r, r\in S)\right).
\end{equation}

Note that by Prop. \ref{prop:top_filt_properties}, (\ref{item:topfilt_projspace})
we have 
$$\tau^c B^*(E_S \times (\mathbb{P}^\infty)^{\times |S|})/\tau^{c+1} 
= \sum_{i=0}^c \tau^i B^*(E_S)/\tau^{i+1} B^*(E_S) [[z_r^B, r\in S]]^{\deg = c-i}.$$
Moreover, by continuity of operations
$\pd^{|S|-1} \phi(z_r^A j_r^*\beta_r, r\in S)$ is divisible by $\prod_{r\in S} z_r^B$,
and similarly for $tr_c \phi$.
Thus, dividing both expressions in (\ref{eq:pd_rho})
by a product of $z_r$
we obtain an equality modulo $\tau^{c+1-|S|}$.

Note, however, that 

$$ \Theta(\phi;p^*(a);S) =  \left(\frac{\pd^{|S|-1} \phi(z_r^A j_r^*\beta_r, r\in S)}{\prod_{r\in S} z_r^B}\right)_{z_r=j_r^* \mu_r^B} $$

because $\omega_t^B(0)=dt$ and positive powers of $t$ do not intervene with the residue.
Since $\rho_B$ sends $j_r^*\mu_r^B$ to $\mu_r^{\CH}$
we obtain that

$$\Theta(\phi;p^*(a);S)  \mod \tau^{c+1-|S|} \equiv \rho_B (\Theta(tr_c \phi;p^*(a);S))$$

 which is enough for the claim as explained above.
 \end{proof}

If a presheaf $\bigoplus_c \tau^c B^*/\tau^{c+1} B^*$ were an oriented theory, then it would be enough
to check the claim of Prop. \ref{prop:op_mod_tau_vs_trunc}
 on products of projective spaces using Vishik's classification theorem. 
However, this is not true, see Remark \ref{rem:gr_oriented}.

The following proposition explains 
how to calculate any operation between free theories
on the graded pieces of the topological filtration
via classes of closed (non-neccesarily-smooth) subvarieties.

\begin{Prop}\label{prop:operations_chow_topfilt}
Let $\phi\colon\tilde{A}^*\rarr \tau^c B^*$ be an operation between free theories.
Then the composition of $\phi$  restricted to $\tau^c A^*$ 
with the projection $\tau^c B^*\rarr \tau^c B^*/\tau^{c+1} B^*$
factors through the quotient over $\tau^{c+1} A^*$
yielding an operation which fits in the following commutative diagram
\begin{center}
\begin{tikzcd}
	\CH^c\ot A \arrow{r} \arrow[d, "\phi_c^{\CH} " ]& \tau^c A^*/\tau^{c+1} A^* \arrow[d, "\phi"] \\
	\CH^c\ot B \arrow{r} & \tau^c B^*/\tau^{c+1}B^* \\
\end{tikzcd}
\end{center}
where  $\phi_c^{\CH} \colon \CH^c \ot A \rarr \CH^c \ot B$
is the composition of the map $tr_c \phi\colon\tau^c A^*/\tau^{c+1} A^*\rarr \CH^c\ot B$ with the map $\rho_A$.
\end{Prop}
\begin{proof}
Let $x\in A^*(X)$, $y\in \tau^{c+1} A^*(X)$ for some smooth variety $X$.
Then there exist an open subvariety $U\subset X$, s.t. $y|_U=0$,
and codimension of $X\setminus U$ in $X$ is greater or equal to $c+1$.
The element $\phi(x+y)-\phi(x)$ is zero when restricted to $U$
since $\phi$ commutes with pullbacks,
i.e.\ $\phi(x+y)\equiv \phi(x) \mod \tau^c B^*(X)$
and $\phi$ factors through $\tau^{c+1} B^*$.

The commutativity of each face of the following diagram follows either by definition or from Prop. \ref{prop:op_mod_tau_vs_trunc}.
\begin{center}
\begin{tikzcd}
 & \tau^c A^* \arrow[d, twoheadrightarrow]  \arrow[r, "\phi"] & \tau^c B^* \arrow{dd} \\
\CH^c\otimes A \arrow[r, twoheadrightarrow] \arrow[rd, "\phi_c^{\CH}"'] & \tau^c A^*/\tau^{c+1} A^* \arrow[d, "tr_c \phi"] \arrow [rd, "\phi"] & \\
& \CH^c\otimes B \arrow[r, twoheadrightarrow] & \tau^c B^*/\tau^{c+1} B^* 
\end{tikzcd}
\end{center}
The commutativity of the diagram in the statement of the proposition follows.
\end{proof}

\subsection{Truncation modulo ideal}\label{sec:trunc_mod}

In the construction of Chern classes from the $n$-th Morava K-theory
we will need a modification of the truncation process
which takes an operation from $A^*$ to $B^*\ot\QQ$ satisfying certain conditions
and produces an operation from $A^*$ to $\CH^*\ot B/p$.
This truncation is described here purely algebraically with the use of Vishik's
classification theorem (Th. \ref{th:Vish_op}) and is not provided with a geometric interpretation.

\begin{Prop}\label{prop:trunc_mod_ideal}
Let $p$ be a prime.
Let $A^*$ and  $B^*$ be free theories defined over torsion-free $\Zp$-algebras $A, B$, respectively.
Let $G\colon\tilde{A}^*\rarr B^*\ot\QQ$ be an operation, let $k\ge 1$.

Assume that on products of projective spaces $\phi$ acts integrally modulo the $k+1$-th 
part of the topological filtration
and the action is zero modulo the $k$-th part of the topological filtration and modulo $p$,
 i.e.\ $\forall n\ge 1$:
$$G_{\{n\}}\colon A^*[[z_1^A,\ldots, z_n^A]]\rarr 
\sum_{s=1}^{k-1} pB(z^B_1,\ldots, z^B_n)^s + B(z^B_1,\ldots, z^B_n)^{k}
 + B\ot\QQ (z^B_1,\ldots,z^B_n)^{k+1},$$
where the series $G_{\{n\}}$ is defined in Section~\ref{sec:cont}.

Then maps $(\pi_k\circ G_{\{n\}}) \mod p\colon A^*[[z_1^A,\ldots, z_n^A]] \rarr B/p[z_1^B,\ldots, z_n^B]^{deg=k}$,
 $n\ge 1$, define an operation $tr^{\mathrm{mod\ } p}_{k}G\colon\tilde{A}^* \rarr \CH^{k}\ot B/p$. 
If $G$ is additive, so is the operation $tr^{\mathrm{mod\ }p}_{k} G$.
\end{Prop}
\begin{proof}
The proof follows the same strategy as Prop. \ref{prop:trunc_op} by defining 
operation $tr^{\mathrm{mod\ } p}_{k} G$ by its values on the products of projective space.
(due to Th. \ref{th:Vish_op}). We leave the details to the reader. 
\end{proof}

\subsection{Lifting operations via truncations}\label{sec:all_op_from_trunc}

The truncation map $tr_c$ (see Propositon~\ref{prop:trunc_op}) is not always an isomorphism,
and we will see that this is not the case
for example if $A^*=\Kn^*$ and $B^*=\mathrm{K}(m)_{int}^*$ where $m \nmid n$ (see Appendix \ref{app:non-exist-op}).
However, in the case of our interest when $A^*=\Kn^*$ 
and $B^*=BP\{n\}^*$ is the universal $p^n$-typical theory 
(see Sections \ref{sec:pn-typ}, \ref{sec:def_morava}) 
the truncation map will be shown to be an isomorphism.
In this situation the following results allow to reduce the problem of constructing operations
to the classification of operations to the Chow groups.

We should note the main feature of sets of operations 
which is used in the following. If $\phi_j\colon A^*\rarr B^*$, $j\in J$ is
a set of operations between two theories such that for every $i\ge 0$ all but a finite number of
operations in this set map to $\tau^i B^*$,
then the sum $\sum_j \phi_j$ is well-defined.
The reason for it is that the topological filtration
is finite on each variety, and therefore all except of finite summands in the sum $\sum_j \phi_j(a)$ 
are zero for $a\in A^*(X)$.

\begin{Prop}\label{prop:all_operations_from_truncations}
Suppose that $[A^*, \CH^i\ot B]$ is a free $B$-module
of finite rank for each $i\ge 0$, 
and assume that the map $tr_i\colon gr^i_\tau[A^*,B^*]\rarr [A^*, \CH^i\ot B]$ is an isomorphism 
 for all $i\ge 0$.
Denote by $\psi^{(i)}_j \in [A^*,\tau^iB^*]$, $j\in J_i$, a finite set of operations
for each $i\ge 0$, s.t. their $i$-th truncations are a basis 
of the corresponding free module.

Then operations $\psi^{(i)}_j$ 
freely\footnote{\label{footnote:free}
Here `free' actually means
'topologically free' with respect to the adic topology defined by the topological filtration on $B^*$, 
i.e.\ infinite sums of operations whose target space has growing index
of the topological filtrations are allowed.}
generate the $B$-module of all operations from $[A^*, B^*]$.

The same is true if one considers only additive operations instead of all operations.
\end{Prop}
\begin{proof}
Let $\phi$ be an operation, and let us construct its representation
as $\sum_{k\ge 0} \sum_{j\in J_k} b_j^{(k)} \psi^{(k)}_j$.
In order to do it we need
the operation $\phi - \sum_{k=0}^i \sum_{j\in J_k} b_j^{(k)} \psi^{(k)}_j$
to take values in $\tau^{i+1} B^*$,
since the residual takes values in $\tau^{i+1}B^*$.
 Let us find coefficients $b_j^{(k)}$ by induction on $k$.

If $\chi_i:=\phi - \sum_{k=0}^i \sum_{j\in J_k} b_j^{(k)} \psi^{(k)}_j$ takes values in $\tau^{i+1} B^*$,
choose coefficients $b_{j}^{(i+1)}$
as coefficients of the representation of $tr_{i+1}\chi_i$
in the basis $tr_{i+1} \psi^{(i+1)}_j$, $j\in J_{i+1}$.
Thus, the $i+1$-th truncation 
of the operation $\chi_{i+1}=\phi - \sum_{k=0}^{i+1} \sum_{j\in J_k} b_j^{(k)} \psi^{(k)}_j$
is zero, and therefore the operation $\chi_{i+1}$ takes values in $\tau^{i+1} B^*$.
The process converges due to the fact that infinite sums of operations $\psi^{(i)}_j$
are defined. Thus, operations $\psi_j^{(k)}$ generate $B$-module of operations $[A^*, B^*]$.

Let us also check that there are no relations between infinite sums of operations $\psi_j^{(k)}$.
Assume that $\phi=\sum_{k\ge 0} \sum_{j\in J_k} b_j^{(k)} \psi^{(k)}_j$ is a zero operation,
and suppose that $b_j^{(k)}=0$ for $k<i$, $j\in J_k$. 
Then $tr_i \phi = \sum_{j\in J_i} b_j^{(i)} tr_i(\psi^{(k)}_j)$ is zero,
and therefore $b_j^{(i)}=0$ for all $j\in J_i$.
\end{proof}

\begin{Lm}\label{lm:trunc_mult}
Let $A^*, B^*$ be free theories.
The multiplication of operations from $A^*$ to $B^*$ defined by the ring structure of $B^*$
yields a structure of $B$-algebra on the set $\bigoplus_i gr^i_\tau[\tilde{A}, B^*]$.

Moreover, the map $tr\colon \bigoplus_i gr^i_\tau[\tilde{A}, B^*]\rarr \bigoplus_i [\tilde{A}^*, \CH^i\ot B]$
is a map of rings.
\end{Lm}
\begin{proof}
The first claim follows from the multiplicativity of the topological 
filtration (Prop. \ref{prop:top_filt_properties}, (\ref{item:topfilt_mult})).

Let $\phi\colon A^*\rarr \tau^iB^*$, $\psi:A^*\rarr \tau^jB^*$ be operations,
and we need to check that \mbox{$tr_{i+j}(\phi \psi)=tr_i(\phi)tr_j(\psi)$.}
By Vishik's classification theorem (Th. \ref{th:Vish_op})
it is enough to prove the equality of values on products of projective spaces.
According to the construction of the truncation 
the claim is equivalent to the following for each $n\ge 0$:
if $f\in B(z_1^B,\ldots, z_n^B)^i$, $g\in B(z_1^B,\ldots, z_n^B)^j$,
then $\pi_{i+j}(f\cdot g) = \pi_i(f)\cdot \pi_j(g)$ (see Section \ref{sec:tr_constr} for the notation).
This is straight-forward.
\end{proof}

\begin{Prop}\label{prop:relations_between_truncations}
Let $A^*$ be a free theory, and let $B^*$ be an oriented theory.

Assume that the truncation maps $tr_i\colon gr^i_\tau[\tilde{A}^*,B^*]\rarr [\tilde{A}^*,\CH^i\ot B]$ 
are isomorphisms for each $i\ge 1$.
Assume that the ring $[\tilde{A}^*,\CH^*\ot B]$ is freely\footnoteref{footnote:free}
 generated as $B$-algebra 
by operations $\bar{t}_i\colon \tilde{A}^*\rarr \CH^{m_i}\ot B$, $i \in I$,
 and assume that $[\tilde{A}^*, \CH^i\ot B]$ is a finite rank (free) $B$-module.

Denote by $t_i\in [\tilde{A}^*, \tau^{m_i}B^*]$ a lift of the operation $\bar{t}_i$
with respect to the truncation map.

Then $[\tilde{A}^*, B^*]$ is freely\footnoteref{footnote:free} generated as
 $B$-algebra by operations $t_i$.
\end{Prop}

\begin{proof}
It follows from Prop. \ref{prop:all_operations_from_truncations}
that $t_i$  generate the ring of all operations,
i.e.\ any operation can be represented as a $B$-series in $t_i$.
Let us prove that there are no algebraic relations between these operations.

Assume that $P\in B[[t_1,\ldots, t_i,\ldots]]$ defines a zero operation, $P\neq 0$.
Define degree of a monomial $\prod_k t_{i_k}^{r_k}$ 
to be $\sum_k m_{i_k}r_k$, 
i.e.\ this monomial defines an operation which takes values in the degree-th part
of the topological filtration on the target.
Let $j$ be the minimal degree of monomial summands of $P$, i.e.\ $j$ is finite.

By Lemma \ref{lm:trunc_mult} the operation $tr_j P$ 
is equal to the polynomial in operations $\bar{t}^{(i)}_j$
obtained by truncating $P$ as a series.
As there are no polynomial relations between operations $\bar{t}^{(i)}_j$ by the assumption, 
we come to a contradiction.
\end{proof}

\section{Morava K-theories and operations to $p^n$-typical theories}\label{sec_op_mor}

Fix a prime number $p$. In what follows all rings will be $\Zp$-algebras 
if not specified otherwise.

\subsection{$p^n$-typical formal group laws}\label{sec:pn-typ}

\begin{Def}
A series $\gamma \in A[[x]]$ is called {\it $p^n$-gradable} (w.r.to $x$)
if it is of the form $\sum_{j\ge 0} a_jx^{1+j(p^n-1)}$.

A series $\eta \in A[[x_1,\ldots, x_n]]$ is called {\it $p^n$-gradable}
if it is $p^n$-gradable w.r.to any variable $x_i$.

We will also say that a series $\gamma$ (or $\eta$) is $p^n$-gradable up to degree $k$
if it is of the form $\sum^{j=k}_{j\ge 0} a_jx^{1+j(p^n-1)}$ modulo $(x^{k+1})$
(of the form $\sum^{j=k}_{j\ge 0} P_j(x_1,\ldots,\hat{x_i}, \ldots, x_n) x_i^{1+j(p^n-1)}$ 
modulo $(x_i^{k+1})$
for each $i\colon 1\le i \le n$, respectively).
\end{Def}

Note that a $p^n$-gradable series over $A$
can be made a homogenous series of degree 1 over the ring $A[v_n]$ 
with $\deg v_n=1-p^n$, and $\deg x=1$.

We record the following straight-forward properties of $p^n$-gradable series.

\begin{Prop}\label{prop:pn_grad_series}
\begin{enumerate}
\item Whenever composition of $p^n$-gradable series is defined, it is $p^n$-gradable.
\item If a series is $p^n$-gradable and invertible w.r.to composition,
then its inverse is also $p^n$-gradable.
\item If a series if $p^{kn}$-gradable, then it is also $p^n$-gradable.
\end{enumerate}
\qed
\end{Prop}

Formal group laws (FGLs) are formal power series in two variables which 
model a multiplication law on a formal one-dimensional neighbourhood.
We have already mentioned them in Section \ref{sec:def} where they appeared in the connection to oriented cohomology theories.
For general facts about formal group laws we refer the reader to \cite{Haz}.

Recall that there exist a notion of a $p$-typical FGL due to Cartier,
which over a torsion-free ring is simplified to the following:
a FGL is $p$-typical iff its logarithm has the form $\sum_{i=1}^\infty l_ix^{p^i}$ (\cite[Th.~4]{Cart}).

There exist the universal $p$-typical FGL $F_{BP}$ over the ring $BP$,
and the corresponding free theory is called the {\it Brown-Peterson cohomology}.
The ring $BP$ is non-canonically isomorphic to the graded ring $\Zp[v_1,v_2,\ldots]$
where the variable $v_i$ has degree $1-p^i$.
The Araki generators are one choice of set of generators $v_i$ for $BP$.
It is common to include $v_0=p$ as a `variable' in notations.
If we denote the coefficients of the logarithm of $F_{BP}$ as $l_i$ so that
 $\log_{BP}=\sum_i l_ix^{p^i}$,
then  $pl_n = \sum_{i=0}^{n} l_iv_{n-i}^{p^i}$ (\cite[I.4 (6.12)]{Araki}).
For these generators we also have $p\cdot_{BP} x = \sum^{BP}_{i\ge 0} v_ix^{p^i}$ ([{\sl op.cit.}, Th. 6.5]). 

Here is a generalization of this notion.

\begin{Def}
Formal group law $F$ over a $\Zp$-algebra is called {\it $p^n$-typical},
if it is $p$-typical and $p\cdot_F x$ is a $p^n$-gradable series.

An oriented theory is called {\it $p^n$-typical},
if the corresponding FGL is $p^n$-typical.
\end{Def}

\begin{Rk}
After this paper was written we learned that
a similar definition has appeared under a different name as early as \cite[21.5.5]{Haz}
and studied in \cite[Def.~2.5]{Rav_A-formal}, \cite[Def.~2.3]{Salch_Moduli}, \cite{Salch_stack}
under the name of $A$-typical formal $A$-modules. 

The importance of this notion was justified in these papers by certain algebraic computations related
to the stack of formal groups and the second page of the Adams-Novikov spectral sequence,
i.e.\ not by the ``topological significance'', see e.g.\ \cite[Introduction]{Rav_A-formal}.

In what follows the notion of $p^n$-typical formal group law
plays an important role in the classification of unstable operations from $p$-local algebraic Morava K-theory.
We do not know if similar results have been obtained in topology before.
\end{Rk}

The following Proposition~\ref{prop:universal_pn-typical_fgl} and Corollaries~\ref{cr:fgl_pn_gradable_logarithm},~\ref{cr:pn-typical-torsion-free}
are standard results. 
We have included their proof for the sake of completeness.

\begin{Prop}[{\cite[Th.~2.6]{Rav_A-formal}}]\label{prop:universal_pn-typical_fgl}
There exist a graded ring $BP\{n\}$ classifying $p^n$-typical FGL's.

The ring $BP\{n\}$ can be naturally identified with a factor ring of $BP\cong \Zp[v_n, v_{2n}, \ldots]$
where Araki generators $v_i$ are sent to zero for $i\nmid n$
and $v_{mn}$ is sent to $v_{mn}$ for $m\ge 1$.

In particular, $BP\{1\}$ is $BP$, and $BP\{kn\}$ is a natural factor-ring of $BP\{n\}$ for any $k,n\inN$.
\end{Prop}
\begin{proof}
It is clear that the ring $BP\{n\}$ exists
 and it can be identifed with a factor of the ring $BP$ 
over the ideal generated by those coefficients of the series $p\cdot_{F_{BP}} x$
which stand to `non-$p^n$-gradable' monomials. 
Denote by $F_{BP\{n\}}$ the universal $p^n$-typical FGL.

Denote by the ring map $\phi\colon BP\rarr BP\{n\}$ the canonical map
classifying the universal \mbox{$p^n$-typical} FGL over $BP\{n\}$ which is also $p$-typical.
Our goal now is to show that $\phi$ sends the Araki generator $v_i$ to zero for $i\nmid n$.
Suppose that $i_0=\min\{j>0\colon  \phi(v_j)\neq 0, j\nmid n\}$ is finite.
Recall that $p\cdot_{BP} x = \sum_i^{BP} v_ix^{p^i}$,
and therefore $p\cdot_{BP\{n\}} x = \sum_i^{F_n} \phi(v_i)x^{p^i}$.
By our assumption 
\begin{equation}\label{eq:p-sum-bpn}
p\cdot_{BP\{n\}} x = 
px+_{BP\{n\}}\sum_{j>0}^{BP\{n\}} \phi(v_{jn})x^{p^{jn}}+_{BP\{n\}}\phi(v_{i_0})x^{p^{i_0}}
+_{BP\{n\}}\mbox{higher degree terms}.
\end{equation}

Note that $\log_{BP\{n\}}$ is $p^n$-gradable up to degree $p^{i_0}-1$ as follows from the assumptions,
 and therefore $F_{BP\{n\}}$ is $p^n$-gradable up to degree $p^{i_0}-1$ as well.
It follows that the sum of two first summands of the RHS of (\ref{eq:p-sum-bpn}) 
is $p^n$-gradable up to degree $p^{i_0}-1$. 

The rightmost summands of (\ref{eq:p-sum-bpn}) have degree in $x$ strictly bigger than $p^{i_0}$,
 and therefore in $p\cdot_{BP\{n\}} x$ there appears the monomial $\phi(v_{i_0})x^{p^{i_0}}$.
Thus, this series is not $p^n$-gradable and we get a contradiction with the finiteness of $i_0$.

To finish the proof one needs to check that 
the $p$-typical FGL $F$ over the ring $BP/(v_i, i\nmid n)$
defined by the canonical map from $BP$ is $p^n$-typical.
As the series $p\cdot_{F} x$
is a homogeneous series of degree 1 (where $\deg x=1$),
and the degrees of all of the generators of $BP/(v_i, i\nmid n)$ are divisible by $p^n-1$,
it follows that this series is $p^n$-gradable. 
\end{proof}

\begin{Cr}\label{cr:fgl_pn_gradable_logarithm}
A formal group law $F$ over a ring $R$ is $p^n$-typical
iff it is $p$-typical and \mbox{$F=\sum a_{ij}x^iy^j$} is $p^n$-gradable.
\end{Cr}
\begin{proof}
The universal $p^n$-typical formal group law $F_{BP\{n\}}$ is 
homogeneous of degree 1 (with $\deg x=\deg y=1$) 
and since the degrees of elements in $BP\{n\}$ are always divisible by $p^n-1$
it follows that $F_{BP\{n\}}$ is $p^n$-gradable.
Since a $p^n$-typical formal group law $F_A$ over a ring $A$
can be obtained as $\phi(F_{BP\{n\}})$ where $\phi\colon BP\{n\}\rarr A$
is a ring map, the power series $F_A$ is also $p^n$-gradable.

The converse is straight-forward: if $F$ is $p^n$-gradable, 
then $[k]\cdot_F x$ is $p^n$-gradable for any $k\in\ZZ$.
\end{proof}

\begin{Cr}\label{cr:pn-typical-torsion-free}
If $A$ is a torsion-free $\Zp$-algebra,
then a FGL $F$ over $A$ 
is $p^n$-typical iff its logarithm is of the form $\sum_{i=0}^\infty l_i x^{p^{ni}}$.
\end{Cr}
\begin{proof}
By Cartier's result mentioned above we know that $p$-typical FGLs over torsion-free $\Zp$-algebras 
are precisely those which have the logarithm of the form $\sum_{i=0}^\infty l_i x^{p^{i}}$.
Thus, $F$ having logarithm of the form prescribed in Corollary is $p$-typical.
Note that $p\cdot_F x = \log^{-1}(\log (px))$.
Thus, if the series $\log_F$ is $p^n$-gradable, 
then by Prop. \ref{prop:pn_grad_series} so is the series $p\cdot_F x$ and $F$ is $p^n$-typical.

Conversely, to prove that the logarithm of a $p^n$-typical FGL is of the form
$\sum_{i=0}^\infty l_i x^{p^{ni}}$
it suffices to check that for the universal $p^n$-typical FGL over the ring $BP\{n\}$.
To do this recall the formulas linking coefficients of the logarithm of $BP$
in terms of Araki generators $v_i$:
$ pl_m = \sum_{i=0}^m l_i v_{m-i}^{p^i},$
where $v_0 = p$.

Denote by $\bar{l_j}\in \QQ[v_n, v_{2n}, \ldots]$
 the coefficients of the logarithm of $F_{BP\{n\}}$,
i.e.\ \mbox{$\log_{BP\{n\}} = \sum_{j=0}^\infty \bar{l_j}x^{p^j}$}.
Let $i_0:=\min\{j>0\colon  l_j\neq 0, n\nmid j\}$, and assume that it is finite.
Then by the definition of Araki generators we have an equality 
$pl_{i_0} = \sum_{i=0}^{m-1} l_i v_{m-i}^{p^i}+l_{i_0} p^{p^{i_0}}$ in $BP$,
and applying the map $\phi\colon BP\rarr BP\{n\}$ to it 
we see that $p\bar{l}_{i_0} = p^{p^{i_0}}\bar{l}_{i_0}$.
Indeed, if $n\nmid i$ then $\phi(l_i)$ is zero for $i<i_0$, 
and if $n|i$, then $n\nmid (i_0-i)$ and $\phi(v_{i_0-i})$ is zero. 
Thus, $\bar{l}_{i_0}=0$ which is a contradiction.
\end{proof}

\begin{Def}
Denote by $BP\{n\}^*$ a free theory 
with the ring of coefficients $BP\{n\}$ and the corresponding formal group law 
is the universal $p^n$-typical FGL.
\end{Def}

\begin{Prop}\label{prop:universal_BPn}
Free theory $BP\{n\}^*$ is the universal $p^n$-typical oriented theory,
i.e.\ for any $p^n$-typical oriented theory\footnote{Note 
that we indirectly impose here that the ring of coefficients of $A^*$ is a $\Zp$-algebra.}
 $A^*$
there exist a unique morphism of oriented theories $BP\{n\}^*\rarr A^*$.
\end{Prop}
\begin{proof}
Follows from the universality of algebraic cobordism (Th. \ref{th:alg_cob_universal}).
\end{proof}

\subsection{Definition of Morava K-theories}\label{sec:def_morava}

\begin{Def}[{\cite[Def. 4.1.1]{Sech}}]\label{mor} Let $n\ge 1$.
A $p^n$-typical free theory $\Kn^*:=\Omega^*\ot_{\LL} \Zp$ 
is called an algebraic\footnote{We will drop `algebraic' from the name
as topological Morava K-theories will not be considered in this paper
except in a few comparison discussions.}
 {\it $n$-th Morava K-theory} $\Kn^*$ 
if the formal group law $F_{\Kn} \mod p$ over $\F{p}$ has height $n$.
\end{Def}

\begin{Rk}\label{rem:graded_morava}
Formal group laws which can be associated with oriented spectra in topology (and motivic topology)
are defined over a graded ring of coefficients and are graded themselves,
i.e.\ formal power series of the FGL is homogeneous of degree 1 where variables have degree 1. 
If $F$ is a $p^n$-typical FGL over $\Zp$,
there exist a graded FGL $\hat{F}$ over the ring $\Zp[v_n,v_n^{-1}]$ where $\deg v_n=1-p^n$
such that $\hat{F}|_{v_n=1} = F$. Indeed, if $\log_F = \sum_{j=0}^\infty a_j x^{p^{nj}}$,
then the logarithm $\log_{\hat{F}}=\sum_{j=0}^\infty a_j v_n^{\frac{p^{nj}-1}{p^n-1}} x^{p^{nj}}$
is homogeneous of degree 1 and defines a graded formal group law. 

In other words, for every $n$-th Morava K-theory
$\Kn^*$ there exist a graded free theory $G\Kn^*$ with the ring of coefficients $\Zp[v_n,v_n^{-1}]$
such that $\Kn^*=G\Kn^*/(v_n-1)$. It is $G\Kn^*$ which may be a part of a cohomology theory
represented by a motivic spectrum. 
We prefer to work with the theory $\Kn^*$, but all the results on operations
can be reformulated for $G\Kn^*$.
\end{Rk}

\begin{Rk}
In topology {\sl the} $n$-th Morava K-theory usually is a unique spectrum
whose homotopy groups (i.e.\ the value of the corresponding cohomology theory on a point)
is the {\sl graded field} $\F{p}[v_n,v_n^{-1}]$
where $\deg v_n=2(p^n-1)$, and there exists an orientation of it
such that the corresponding graded formal group law has height $n$ (see e.g.\ \cite[Section 4.2]{Rav}).
Any such formal group law is liftable to $\Zp[v_n,v_n^{-1}]$,
and if the relation $v_n=a\in\Zp^\times$ is imposed this will be a formal group
of a \mbox{$n$-th} Morava K-theory as defined above.
However, it is not clear to us which of these lifts can be performed
to yield a (motivic) spectrum representing a cohomology theory, 
and whether the lift is in any sense unique. 

For a uniqueness statement
concerning algebraic Morava K-theories $\Kn^*$ above see Section \ref{sec:morava_unique}.
\end{Rk}

\subsection{Grading on Morava K-theories}

We will need the following standard result on the logarithm of the FGL of $\Kn$.

\begin{Prop}[see e.g.\ {\cite[Prop. 4.3.2]{Sech}}]\label{prop:log_morava}
Let $F$ be a FGL of an $n$-th Morava K-theory.

Then its logarithm has the form
$$ \log_{\Kn}(x) = x+\frac{a_1}{p}x^{p^n}+\frac{a_2}{p^2}x^{p^{2n}}+\ldots$$
where $a_i\in\Zp^{\times}$ for all $i\ge 1$.
Moreover, $a_k \equiv (a_1)^k \mod p$.
\end{Prop}

One can easily check 
that the first terms of the formal group law of any Morava K-theory $\Kn^*$ look like this: 
$$F_{\Kn}(x,y)=
x+y-a_1\frac{1}{p}\sum_{i=1}^{p^n-1} \binom{p^n}{i} x^iy^{p^n-i}+\textrm{higher  degree terms}.$$

\begin{Rk}\label{rem_k0-mor}
The Artin-Hasse exponential establishes an isomorphism 
between formal group laws $F_m=x+y+xy$ and a $p$-typical FGL over $\Zp$ of height 1.
This implies that $\KK\ot\Zp$ is isomorphic 
to a first Morava K-theory as a presheaf of rings.

It is not true, however, 
 that every two $n$-th Morava K-theories as defined above are multiplicatively isomorphic
 (see Appendix \ref{app:mor_not_mult}).
\end{Rk}

\begin{Prop}[{\cite[Prop. 4.1.5]{Sech}}]\label{prop:mor_grad}
\begin{enumerate}
\item Morava K-theories $\Kn^*$ are $\ZZ/(p^n-1)$-graded.

\item The grading on $\Kn^*$ is respected by Adams operations.

\item\label{item:push-forward} The grading is compatible with push-forwards,
i.e.\ for a proper morphism $f\colon X\rarr Y$ of codimension $c$
push-forward maps increase grading by $c$,
 $f_*\colon \Kn^i(X)\rarr \Kn^{i+c}(Y)$.
 
In particular, $c_1^{\Kn}(L)\in \Kn^1(X)$ for any line bundle $L$ over a smooth variety $X$.
\end{enumerate}
\end{Prop}

We denote the graded components of $n$-th Morava K-theories
 as $\Kn^1$, $\Kn^2$, $\ldots$, $\Kn^{p^n-1}$,
and freely use the following expressions
 $\Kn^i$, $\Kn^{i \mod{p^n-1}}$, $\Kn^{i+r(p^n-1)}$
to denote the component $\Kn^j$ where $j\equiv i \mod p^n-1$, $1\le j\le p^n-1$.
The reason we denote the component $\Kn^{0}$ as $\Kn^{p^n-1}$
is mainly because we will work with $\Kntilde^*$
instead of $\Kn^*$, $\Kntilde^{p^n-1}$ 
contains classes of codimensions $p^n-1 + r(p^n-1)$ for all $r\ge 0$.

The grading on the $n$-th Morava K-theory splits the topological filtration
``modulo a period of $p^n-1$ steps'' as explained in the next proposition.

\begin{Prop}\label{prop:top_morava_grading}
The topological filtration on each graded component of the $n$-th Morava K-theory
changes only every $p^n-1$ steps, i.e.
$$ \tau^{j+s(p^n-1)+1} \Kntilde^j = \tau^{j+s(p^n-1)+2} \Kntilde^j = \ldots =
\tau^{j+(s+1)(p^n-1)} \Kntilde^j,$$
where $j\in [1, p^n-1]$, $s\ge 0$.

In particular, $gr^j_\tau \Kntilde^* = \Kntilde^j/\tau^{j+p^n-1} \Kntilde^j$
for $j\colon 1\le j\le p^n-1$, 
and \mbox{$gr^j_\tau \Kntilde^*= gr^j_\tau \Kntilde^j$} for every $j$.
\end{Prop}
\begin{proof}
The topological filtration on algebraic cobordism has
a description in terms of the structure of $\LL$-module by \cite[Th. 4.5.7]{LevMor}:
$\tau^i\Omega^n = \cup_{m\le n-i} \LL^m \Omega^{n-m}$,
and the morphism of theories $\Omega^*\rarr \Kn^*=\Omega^*\ot_{\LL} \Kn$
yields a surjection $\tau^i \Omega^* \ot \Kn \rarr \tau^i \Kn^*$ by 
Proposition \ref{prop:top_filt_properties}, \ref{item:topfilt_free_surj}).

It follows from Proposition \ref{prop:mor_grad}, (\ref{item:push-forward})
that $\Omega^k$ maps to $\Kn^{k \mod p^n-1}$ for all $k$,
 i.e.\ \mbox{$\Kn^k = \oplus_{s\in\ZZ} \Omega^{k+s(p^n-1)}\ot_\LL \Kn$.}
Thus, for $j\in [1, p^n-1]$ the group $\tau^i \Kn^j$ is the image of $\bigoplus_{s\in\ZZ} \tau^i \Omega^{j+s(p^n-1)}$ 
which is equal to $\bigoplus_{s\in\ZZ} \cup_{m\le j+s(p^n-1)-i} \LL^m \Omega^{j+s(p^n-1)-m}$.

Since the formal group law of $\Kn^*$
is $p^n$-typical, the map $\LL \rarr \Kn$ can be factored
as $\LL\xrarr{\phi} \Zp[v_n, v_n^{-1}]\xrarr{v_n=1} \Zp=\Kn$ where $\phi$ is a graded map, $\deg v_n=1-p^n$
(cf. Remark \ref{rem:graded_morava}). 
Therefore $\LL^j$ maps to zero in $\Kn$ if $j\neq 0 \mod p^n-1$.

Thus, we see that $\LL^m \Omega^{j+s(p^n-1)-m}$ maps to zero in $\Kn^j$ if $m\neq 0 \mod p^n-1$,
and therefore the image of the group $\bigoplus_{s\in\ZZ} \cup_{m\le j+s(p^n-1)-i} \LL^m \Omega^{j+s(p^n-1)-m}$
in $\Kn^j$ changes only each $p^n-1$ steps with the change of $i$.
\end{proof}

\subsection{Chern classes: statement of the main theorem}\label{subsec_mainth}

In \cite[Th. 4.2.1]{Sech} we have shown
that for every $n$-th Morava K-theory $\Kn^*$ there exist operations from 
$\Kn^*$ to $\CH^*\ot\Zp$ which we called Chern classes.
Even though these operations were not unique, 
let us fix and denote by $c_j^{\CH}$ any choice of them.
Now we are going to generalize this result to a wider class of theories in place of Chow groups,
namely, to the $p^n$-typical oriented theories.

\begin{Th}\label{th:main}
For every Morava K-theory $\Kn^*$ 
and every $p^n$-typical theory $A^*$
there exist a series of operations $c_j\colon \Kn^*\rarr A^*$ for $j\ge 1$
satisfying the following conditions. 
\begin{enumerate}[i)]
\item Operation $c_j$ restricted to the graded components of $\Kn^*$
is zero except for $\Kn^{j\mod (p^n-1)}$.
\item\label{item:c_j-tau^j} Operation $c_j$ takes values in $\tau^jA^*$,
and $tr_j c_j$ is the Chern class $c_j^{\CH}\ot id_A$.
\item Denote by $c_{tot} = \sum_{i\ge 1} c_it^i$ the total Chern class in a formal variable $t$. 
Then the Cartan formula holds universally: 
$$c_{tot}(x+y)=F_{\Kn}(c_{tot}(x),c_{tot}(y)),$$
where $x,y\in \Kn^*(X)$ for any smooth variety $X$ and the identity takes place in $A^*(X)[[t]]$;
\item If $A$ is a free $\Zp$-module, then 
all operations from $\Kn^*$ to $A^*$ are uniquely expressible as series in Chern classes:
$$ [\Kntilde^*, A^*] = A[[c_1,\ldots, c_i,\ldots ]].$$
Moreover, the analogous statement is true for poly-operations.
\end{enumerate}
\end{Th}

We call operations $c_j$ described in this Theorem {\it Chern classes}, even though
it may lead to confusion especially if the usual Chern classes are involved.
For the clarity we suggest to use the notation $c_j^{\Kn\rarr A^*}$, $c_i^{\KK\rarr A^*}$
to distinguish the notions in writing when needed.

\begin{Rk}
Petrov and Semenov introduced
 operations $c_1, c_2,\ldots,  c_{p^n}$ from a specific Morava \mbox{K-theory $\Kn^*$}
 to Chow groups modulo $p$-torsion in \cite{PetSem}. 
 One may choose lifts of their operations $c_1, c_2, \ldots c_{p^n-1}$ to $\CH^*\ot\Zp$
 as the starting point of the construction of operations in \cite{Sech}.
However, the natural choice of the operation $c_{p^n}$ then will
differ by a sign to the choice of Petrov-Semenov.
\end{Rk}

\begin{Rk}
No uniqueness of Chern classes 
is claimed in the theorem even when the Chern classes to Chow groups are fixed. 
Moreover, one can check through the proof of Theorem \ref{th:main}
that having constructed operations $c_1,\ldots, c_i$ 
one can define $c_i^{new}=c_i+\phi_i$ where $\phi_i$ 
is any additive operation to $\tau^{i+1}A^*$, and then construct operations $c^{new}_j$ for $j>i$
so that all properties of the theorem are satisfied. 
\end{Rk}

\begin{Rk}
The ring structure on $[\Kntilde^*, A^*] $ is induced from the multiplication on $A^*$ (see Definition~\ref{operation}).
In the case when $A^*=\Kn$ there exists a different (not necessarily commutative) ring structure on the set $[\Kntilde^*, \Kn^*]$,
as its elements are endo-operations of $\Kn^*$ and can be composed. 
Although we study these endo-operations in Section~\ref{sec:morava_gamma_filtration},
the composition of Chern classes $c_j^{\Kn\rarr \Kn}$ are not investigated there.
Nevertheless, any such composition can be uniquely written as a series in Chern classes  $c_j^{\Kn\rarr \Kn}$
as follows from Theorem~\ref{th:main}.
We also do not study the action of the Morava stabilizer group on $[\Kntilde^*, \Kn^*]$ in this paper.
\end{Rk}

\begin{Rk}
As we explain in Appendix~\ref{app:disclaimer} $\Kn$ 
is the geometric part of a bi-graded cohomology theory represented by the motivic spectrum $\Knspec$,
i.e.\ $\Kn^*\cong \Knspec^{2*,*}$.
If a similar property holds for $A$ in the theorem above, i.e.\ $A^*\cong \mathcal{A}^{2*,*}$ for some motivic spectrum $\mc{A}$,
then the computation of operations from $\Kn^n$ to $A^m$ is nothing else than
the computation of the set $[\Omega^\infty \Sigma^{2n,n} \Knspec, \Omega^\infty \Sigma^{2m,m} \mc{A}]$.

Note that the operations of Theorem~\ref{th:main} exist over any field of characteristic $0$,
and, in particular, over $\CC$. By using the topological realization applied to the motivic spectra
we get a map
$$ [\Omega^\infty \Sigma^{2n,n} \Knspec, \Omega^\infty \Sigma^{2m,m} \mc{A}] \rarr 
[\Omega^\infty \Sigma^{2n} \Knspec^{top}, \Omega^\infty \Sigma^{2m} \mc{A}^{top}].$$
This map is injective, since all operations act non-trivially on the products of projective spaces,
and the values of oriented cohomology theories on them are essentially determined by the ring of coefficients,
both in the motivic and in the topological worlds. 

We do not investigate the question, whether the map is surjective,
and we do not study the topological realizations of the operations that we construct.
\end{Rk}

Even though constructed operations are not unique,
one can define the gamma filtration on the $n$-th Morava K-theory using them
and it does not depend on any choices (Section \ref{sec:morava_gamma_filtration}).

It should also be noted that one can substitute 
the formal group law of $\Kn^*$ in the Cartan formula
by any $p^n$-typical FGL over $\Zp$ of height $n$, i.e.\ to any formal group law
defining $n$-th Morava K-theory.  
At the moment it seems that there is no advantage of using one or another FGL 
for the Cartan formula for Chern classes from $\Kn^*$,
however, using the same FGL as that of the orientation of $\Kn^*$ 
make it look similar to the classical case of $\KK$.

\begin{Rk}
If $A$ is not a free $\Zp$-module, 
then it is not true that all operations from $\Kn^*$ to $A^*$ 
are expressible in terms of Chern classes for $n>1$. 

In particular, we have shown in \cite[End of Sec. 4.5]{Sech} 
that the ring of operations from $\Kn^*$ to $\CH^*/p$
generated by Chern classes is not stable under the action of the Steenrod algebra.
\end{Rk}

\section{Proof of Theorem \ref{th:main}}\label{sec:proof_main}

In this section an $n$-th Morava K-theory $\Kn^*$ is fixed, 
its FGL is denoted by $F_{\Kn}$ 
and its logarithm is $\log_{\Kn}(x) = x+ \sum_{i=1}^\infty \frac{a_i}{p^i}x^{p^{ni}}$
for some $a_i\in\Zp^\times$ (see Prop. \ref{prop:log_morava}).

First, we classify additive operations from $\Kn^*$ to $BP\{n\}^*$ in Section \ref{sec:add_morava_BPn},
and use them to construct Chern classes with values in $BP\{n\}^*$ in
Section \ref{sec:construction_chern_classes}.
Then for any $p^n$-typical theory $A^*$ Chern classes from $\Kn^*$ are defined as compositions
of constructed operations $c_i\colon \Kn^*\rarr BP\{n\}^*$ 
with the unique morphism of oriented theories $BP\{n\}^*\rarr A^*$.
Finally, in Section \ref{sec:chern_generate_all} based on the results about the truncation of operations 
from Section \ref{sec:trunc}
we prove that if $A$ is a free $\Zp$-module then Chern classes generate all operations to $A^*$.

\subsection{Additive operations from Morava K-theory to $BP\{n\}^*$}\label{sec:add_morava_BPn}

We will need to work with additive operations from $\Kn^*$ to $BP\{n\}^*/p$,
however, their classification is more difficult than of integral operations.
We restrict ourselves to so-called gradable operations 
which are much more amenable to investigations (cf. \cite[Section 4.5]{Sech}).

\begin{Def}[cf. {\cite[Def. 4.5.2]{Sech}}]\label{def:grad}
An additive operation $\phi\colon \Kn^*\rarr A^*$ to an oriented theory $A^*$ 
is called {\it gradable} if series $\phi_{\{l\}}\in A[[z_1,\ldots,z_l]]$ defining it (see Section \ref{sec:cont})
is $p^n$-gradable for all $l\ge 1$.
\end{Def}

\begin{Prop}[cf. {\cite[Prop. 4.5.4]{Sech}}]\label{prop:add_grad}
Let $A^*$ be a $p^n$-typical oriented theory s.t. $A$ is a torsion-free $\Zp$-module.
Then all additive operations from $\Kn^*$ to $A^*$ are gradable.
\end{Prop}
\begin{proof}
Let $\phi\colon \Kn^*\rarr A^*$ be an additive operation.
Its composition with the Chern character $ch_A\colon A^*\rarr \CH^*_\QQ\ot A$
is  an additive operation to $\CH^*_\QQ\ot A$. 
However, as a presheaf of abelian groups $\CH^*_\QQ\ot A$  
is a direct sum of presheaves $\CH^*_\QQ$, and by \cite[Prop. 4.5.4]{Sech} 
any additive operation to this presheaf is gradable. Thus, $ch_A\circ \phi$ is gradable as well.

Note that the operation $(ch_A)^{-1}\colon \CH^*_\QQ\ot A\rarr A^*\ot\QQ$ is gradable.
Indeed, it is a multiplicative operation, and therefore 
$G_l(1)=\prod_{i=1}^l G_1(1)|_{z^{\CH}=z_i^{\CH}}$,
so it is enough to check that $G_1$ is gradable. 
However, $G_1(1)=\log_A(z^{\CH})$ which is gradable by definition of $p^n$-typical theory 
and Cor. \ref{cr:pn-typical-torsion-free}.

Thus, the operation $\phi\ot \id_\QQ=(ch_A)^{-1}\circ (ch_A \circ \phi)$
from  $\Kn^*$ to $A^*\ot\QQ$
is gradable, and as $A$ is a torsion-free ring therefore so is $\phi$.
\end{proof}

\begin{Prop}\label{prop:add_op_morava_BPn}
The truncation map  (see Prop. \ref{prop:trunc_op})
$$tr_i\colon  gr^i_\tau [\Kn^*,\tau^i BP\{n\}^*]^{add} \hookrightarrow [\Kn^*,\CH^i\ot BP\{n\}]^{add}$$
defines an isomorphism of free $BP\{n\}$-modules of rank 1 for each $i\ge 0$.

Moreover, for each $i\ge 0$
there exists an additive operation $\phi_i$ supported on $\Kn^i$ 
and taking values in $\tau^i BP\{n\}^i$ 
which generates the module above.
\end{Prop}
\begin{proof}
The group $[\Kn^*, \CH^i\ot\Zp]^{add}$ is a free $\Zp$-module of rank 1 by \cite[Prop. 3.6]{Sech}.
Since $BP\{n\}$ is a free $\Zp$-module, the presheaf of abelian groups $\CH^*\ot BP\{n\}$
is isomorphic to a direct sum of presheaves $\CH^*\ot\Zp$.
Combining these two facts together
we obtain that the group $[\Kn^*,\CH^i\ot BP\{n\}]^{add}$ 
is a free $BP\{n\}$-module of rank 1. 

Thus, to prove the proposition it suffices to find an additive operation
from $\Kn^i$ to $BP\{n\}^i=\tau^i BP\{n\}^i$ such that
its $i$-th truncation comes from a generator of the $\Zp$-module $[\Kn^*, \CH^i\ot\Zp]^{add}$ 
which is a submodule of the $BP\{n\}$-module $[\Kn^*, \CH^i\ot BP\{n\}]^{add}$.
Let $p_i\colon BP\{n\}^i\rarr \CH^i\ot\Zp$ be the $i$-th component of the unique morphism of theories
from $BP\{n\}^*$ to $\CH^*$. Direct computation on products of projective spaces shows that 
that for an operation $\phi\colon \Kn^*\rarr BP\{n\}^i$ we have $p_i \circ \phi = tr_i (\phi)$.

In what follows 
we will repeatedly and implicitly use Vishik's Theorem~\ref{th:Vish_op} 
to identify the sets
of operations from $\Kn^*$ to $\CH^*\ot\Zp$ (or to $BP\{n\}^*$)
as subsets of those operations to $\CH^*\ot\QQ$ (or to $BP\{n\}^*\ot\QQ$, respectively)
which act on the products of projective spaces integrally.

Let us first construct an additive operation $\psi_j\colon \Kn^j\rarr BP\{n\}^j_\QQ=\tau^jBP\{n\}^j_\QQ$
 for any $j\ge 1$
s.t. its truncation $tr_j\psi_j\colon \Kn^*\rarr \CH^j\ot\QQ$
is a generator of additive operations to $\CH^j\ot\Zp$.
We have an isomorphism $[\Kn^*,BP\{n\}^*_\QQ]^{add}= [\Kn^*\ot\QQ, BP\{n\}^*\ot \QQ]^{add}$,
and there exist the Chern character isomorphism $BP\{n\}^*_\QQ\xrarr{\sim} \CH^*_\QQ\ot BP\{n\}$
which identifies $\tau^k BP\{n\}^j_\QQ$ with $\oplus_{s\ge k} (\CH^s_\QQ\ot BP)^{\deg = j}$.
It follows that the truncation map 
$$tr_j\colon  gr^j_\tau [\Kn^*,BP\{n\}^j_\QQ]^{add} \rarr [\Kn^*,\CH^j_\QQ]^{add}$$
 is an isomorphism.
Thus, a generator of additive operations from $\Kn^*$ to $\CH^j\ot\QQ$ 
which is supported on $\Kn^j$ (\cite[Prop. 4.1.6]{Sech}) 
can be lifted to an additive operation $\psi_j$ from $\Kn^j$ to $BP\{n\}^j_\QQ$.

Second, let us define the additive operation $\phi_i$
as an infinite $BP\{n\}\ot\QQ$-linear combination of operations
 $\psi_j$, $j\ge i$, $j\equiv i \mod (p^n-1)$.
Such an operation will be supported on $\Kn^i$.
The linear combination is chosen by the following inductive procedure.

On the $k$-th step of the induction we construct an additive operation 
\mbox{$\phi_i^k\colon \Kn^i\rarr BP\{n\}^i\ot\QQ$,}
s.t. $tr_i \phi_i$ is a generator of additive operations to $\CH^i\ot\Zp$
and $\phi_i^k$ acts integrally on products of projective spaces 
modulo the $(i+k(p^n-1)+1)$-th piece of the topological filtration.
The latter condition could also be written
as series $G_l$ defining the operation
for all $l\ge 1$ satisfy 
$$G_l \mod (z_1,\ldots, z_l)^{i+k(p^n-1)+1} \in BP\{n\}[z_1,\ldots,z_l]/(z_1,\ldots, z_l)^{i+k(p^n-1)+1}.$$

We should also note that if a gradable operation $\phi$ on $\Kn^i$
is integral modulo the \mbox{$(i+k(p^n-1)+1)$-th} piece of the topological filtration,
then it is also integral modulo the the $(i+(k+1)(p^n-1))$-th piece of the topological filtration.
Indeed, the series $G_k$ defining the operations has non-trivial summands of degrees 
$k+r(p^n-1)$ only, where $r\ge 0$. However, for the operation supported on $\Kn^i$ 
we have $G_k =0 $ if $k\not\equiv i \mod p^n-1$.

{\bf Base of induction ($k=0$).} Take $\phi_i^0$ to be equal $\psi_i$.
This operation takes values in $\tau^i$, and its $i$-th truncation is integral,
which means that it acts integrally modulo $(i+1)$-th piece of the topological filtration
on products of projective spaces.

{\bf Induction step ($k\rarr k+1$).} The operation $\phi_i^k$
acts integrally on the projective spaces modulo $(i+k(p^n-1)+1)$-th part of the topological filtration.
Since series $G_l$ defining the operation $\phi_i^k$ are gradable for each $l\ge 1$,
they are integral modulo $(i+(k+1)(p^n-1))$-th part of the topological filtration.

Denote by $m$ the maximal $p$-power in the denominators of degree $i+(k+1)(p^n-1)$ summands in
series $G_l$ for all $l\ge 1$. In other words, $m$ is the minimal number $M$,
s.t. $p^M \phi_i^k$ acts\ integrally\
on\ projective\ spaces\ modulo\ $(i+(k+1)(p^n-1)+1)$-th\ part\ of\ the\ topological\ filtration.
If $m=0$, then one can define $\phi_i^{k+1}=\phi_i$.

Assume now that $m>0$. The operation $p^m \phi_i^k$ satisfies 
the assumptions of Proposition \ref{prop:trunc_mod_ideal}, and therefore
 there is a well-defined additive operation $tr^{\mod p}_{i+(k+1)(p^n-1)+1} (p^m \phi_i^k)$
 from $\Kn^i$ to $\CH^{i+k}\ot BP\{n\}/p$.
This operation acts on the products of projective spaces
via the series \mbox{$(p^m G_l \mod (z_1, \ldots, z_l)^{i+(k+1)(p^n-1)+1}) \mod p$}, $l\inN$,
which are in fact polynomials of degree $i+(k+1)(p^n-1)$ in variables $z_i$.

Clearly, this operation is also gradable. Since the module of gradable additive operations to $\CH^{i+(k+1)(p^n-1)}/p$ 
is of rank 1 by \cite[Cor. 4.5.12]{Sech}, 
the operation $tr^{\mod p}_{i+(k+1)(p^n-1)+1} (p^m \phi_i^k)$ is proportional to the truncation of the operation $\psi_{i+(k+1)(p^n-1)}$.
Therefore there exist \mbox{$b\in BP\{n\}^{i-k-1}$}
s.t. $p^m\phi_i^k-b\psi_{i+(k+1)(p^n-1)}$ 
has denominators at most $p^{m-1}$
in the \mbox{$i+(k+1)(p^n-1)$}-th part of filtration. This induction process reduces $m$ 
to zero in finitely many steps, 
and one defines $\phi_i^{k+1}$ as $\phi_i^k+x\psi_{i+(k+1)(p^n-1)}$ 
for some $x\in BP\{n\}^{i-k-1}\ot\QQ$. 
Thus, $tr_i \phi_i^{k+1} = tr_i \phi_i^{k}$ is a generator of additive operations.

Note that since the operation $\psi_{i+k(p^n-1)+1}$ takes values in $\tau^{i+k(p^n-1)+1}$,
infinite linear combinations of these operations are well-defined
and the induction process on $k$ converges.

Thus, the induction process yields the additive operation $\phi_i\colon \Kn^*\rarr \tau^i BP\{n\}^*$
such that its $i$-th truncation is a generator of the free $BP\{n\}$-module
$[\Kn^*, \CH^i\ot BP\{n\}]^{add}$ of rank one. 
\end{proof}

Let us fix the notation $\phi_i\colon \Kn^*\rarr \tau^i BP\{n\}^*$ for the operation
constructed in the proof above throughout this section.

\begin{Cr}\label{cr:add_op_Kn-pntyp}

Let $A^*$ be a $p^n$-typical oriented theory.
Denote by $\phi_i^A$ the composition of the additive operation $\phi_i$ 
with the unique morphism of theories $p\colon BP\{n\}^*\rarr A^*$.

\begin{enumerate}
\item If $A$ is $\F{p}$-algebra, then 
the $A$-module of gradable additive operations from $\Kn^*$ to $A^*$ 
is freely generated by $\phi_i^A$, 
i.e.
$$ [\Kntilde^*, A^*]^{add, grad}= \{\sum_{i=1}^\infty a_i\phi^A_i| a_i\in A \}.$$
\item
If $A$ is a free $\Zp$-module, then
the $A$-module of all additive operations from $\Kn^*$ to $A^*$ 
is freely generated by $\phi_i^A$, 
i.e.\ 
$$ [\Kntilde^*, A^*]^{add}= \{\sum_{i=1}^\infty a_i(\phi^A_i \mod p)| a_i\in A \}.$$
\end{enumerate}
\end{Cr}
\begin{proof}
It follows from the construction of the truncation,
that $tr_i (\phi_i^A) = tr_i (\phi_i)\ot \id_A$. 

If $A$ is a $\F{p}$-algebra, then $\CH^*\ot A$ as a presheaf of abelian groups
is isomorphic to a direct sum of presheaves $\CH^*/p$. In this case $tr_i (\phi_i^A)$ is 
an additive generator of gradable additive operations to $\CH^i/p$ (\cite[Cor. 4.5.12]{Sech}),
and therefore is a generator of the $A$-module of gradable additive operations to $\CH^i\ot A$.
The first part of this Proposition is proved then exactly 
in the same way as Proposition \ref{prop:all_operations_from_truncations}
where one considers only gradable additive operations instead of all additive operations.

If $A$ is a free $\Zp$-module,
 then $\CH^*\ot A$  as a presheaf of abelian groups
is isomorphic to a direct sum of $\CH^*\ot\Zp$. 
By construction $tr_i (\phi_i^A)$ is a generator of the $A$-module of all additive operations to $\CH^i\ot A$.
The claim follows from Proposition \ref{prop:all_operations_from_truncations}.
\end{proof}

\begin{Rk}
A particular case of Corollary \ref{cr:add_op_Kn-pntyp}
is a classification of additive endo-operations of Morava K-theories $\Kn^*$.
Note that additive endo-operations of $\KK$ were classified by Vishik (\cite[Th. 6.8]{Vish1})
and are infinite linear combinations of special finite sums of Adams operations.
In other words one could say that additive endo-operations of $\KK$ 
(and by essentially the same argument of $K(1)^*$) are generated by Adams operations.

This is not the case for $\Kn^*$, $n>1$. More precisely, one can show
that for a sufficiently big $i$ the operation $\phi_i^{\Kn}$ constructed above can not be equal 
to a `convergent' infinite $\Zp$-linear combination of Adams operations. In Topology a similar fact 
about topological Morava K-theories is well-known.
See \cite{Rav_Morava_stabilizer} for the multiplicative endo-operations of 
the topological Morava K-theory  with $\F{p}$-coefficients,
and also \cite[Remark on p. 437]{Yag}.
\end{Rk} 

\subsection{The construction of Chern classes}\label{sec:construction_chern_classes}

\begin{Prop}
There exist operations $c_i\colon \Kn^i\rarr BP\{n\}^i$ for any $i\ge 1$
s.t. 
\begin{enumerate}
\item $c_{tot}=\sum_{i\ge 1} c_it^i$ satisfies a Cartan-type formula:
$$ c_{tot}(x+y)=F_{\Kn}(c_{tot}(x), c_{tot}(y));$$
\item operation $tr_i c_i$ is equal to $c^{\CH}_i\ot \id_{BP\{n\}}$.
\end{enumerate}
\end{Prop}
\begin{proof}
We will construct operations $c_i$ by induction,
in the same way as it was done for operations to Chow groups in \cite[Sec. 4.4]{Sech}.
This is possible as the Cartan formula restricted to the coefficient of the monomial $t^i$
describes non-additivity of the operation $c_i$
in terms of operations $c_j$ where $j<i$. 
For example, Chern classes $c_i$ (as well as $c_i^{\CH}$)
 have to be additive for $i\colon 1\le i \le p^n-1$.

{\bf Base of induction.}
Note that  operations $c_i^{\CH}$, $i\colon 1\le i \le p^n-1$,
  are chosen in \cite{Sech}
 to be generators of modules $[\Kn^*, \CH^i\ot\Zp]^{grad}$.
Thus, by Prop. \ref{prop:add_op_morava_BPn} 
 there exist numbers $a_i\in\Zp^\times$ for $i\colon 1\le i \le p^n-1$ 
 s.t. $tr_i (a_i\phi_i) = c_i^{\CH}$.
  We choose $c_i$ to be equal to $a_i\phi_i$ for $i\colon 1\le i \le p^n-1$.

{\bf Induction step.}
Assume that operations $c_1,\ldots, c_{i-1}$ are constructed
and satisfy properties of the proposition.

Similar to the case of Chern classes with values in Chow groups
we define the operation $c_i$ as a sum of a polynomial in $c_1,\ldots, c_{i-1}$
and an additive operation. In the current situation this additive operation is taking values in $\tau^i BP\{n\}^*$.

Denote by $P_i\in \QQ[c_1,\ldots, c_{i-1}]$
the coefficient of $t^i$ in the formal power series $c_i - (\log_{\Kn}(c_{tot}))$ in $t$.
Let $\mu_i = \max(0, -\nu_p(P_i))$ where $\nu_p$ is the minimal $p$-valuation of coefficients of $P_i$. 

Note that the polynomial $P_i$ has degree $i$ in variables $c_j$ (where $\deg c_j =j$),
and since the topological filtration is multiplicative on $BP\{n\}^*$ 
(Prop. \ref{prop:top_filt_properties}, (\ref{item:topfilt_mult})),
by the induction assumption the operation defined by the polynomial $P_i$ 
takes values in $\tau^i BP\{n\}^*$.

It follows from Lemma~\ref{lm:trunc_mult} 
that $tr_i P_i (c_1,\ldots, c_{i-1})$, $P_i(tr_1 c_1,\ldots, tr_{i-1}c_{i-1})$
and $P_i(c_1^{CH},\ldots, c_{i-1}^{CH})$ are equal.
By \cite[Lemma~4.4.1]{Sech} $c_i^{\CH}$ as an operation to $\CH^i\ot\QQ$
equals to $P_i(c_1^{\CH}, \ldots, c_i^{\CH})+\frac{\psi_i}{p^{\mu_i}}$
where $\psi_i$ is a generator of additive operations from $\Kn^*$ to $\CH^i\ot\Zp$. 
By Proposition \ref{prop:add_op_morava_BPn} there exist $a_i\in\Zp^\times$ such that $tr_i (a_i\psi_i) =\phi_i$,
and therefore an operation $P_i(c_1,\ldots,c_{i-1}) +\frac{a_i\phi_i}{p^{\mu_i}}$
takes values in $\tau^i BP\{n\}^*\ot\QQ$
and its $i$-th truncation is equal to $c_i^{\CH}$.
To define $c_i$ we now will add to this operation a linear combination of additive 
operations so that it acts integrally on products of projective spaces.
The following induction process is rewriting of
 \cite[p. 33, Lemma 4.4.1 follows from Lemma 4.4.2]{Sech} with minor changes.
 
{\bf Claim.}
There exist an additive operation $\psi_i^{(r)}\colon \Kn^*\rarr \tau^i BP\{n\}^*$
such that 1) the operation $tr_i \psi_i^{(r)}$ is a generator of additive operations to $\CH^i\ot BP\{n\}$,
2) the operation $p^{\mu_i-r}P_i(c_1,\ldots,c_{i-1}) +\frac{\psi^{(r)}_i}{p^{r}}$
acts integrally on products of projective spaces.

{\bf Base of induction ($r=0$).}
By definition of $\mu_i$ the polynomial $p^\mu_iP_i$ is integral,
and one can choose $\psi^{(0)}$ to be the operation $\phi_i$ of Prop. \ref{prop:add_op_morava_BPn}.

{\bf Induction step $(r\rarr r+1)$.}
Let $p^{\mu_i-r}P_i(c_1,\ldots,c_{i-1}) +\frac{\psi^{(r)}_i}{p^{r}}$
be an operation acting integrally on products of projective spaces.

Note that the derivative of this operation is equal 
to $p^r\pd^1 P_i$, and $\pd^1 P_i$ is an integral polynomial in operations $c_1,\ldots, c_i$
by [{\sl loc. cit.}]. Thus, if $r>0$, then this operation modulo $p$ is an additive operation
which takes values in $\tau^i BP\{n\}^*/p$.
Moreover, it is gradable as follows by copying the argument of [{\sl op.cit.}, Prop. 4.5.6].
Thus, by Corollary \ref{cr:add_op_Kn-pntyp}, (1)
 it is equal to a sum $\sum_{s\ge i} (b_s\phi_s \mod p)$ modulo~$p$,
and one can choose $\psi^{(r+1)}_i:=\psi^{(r)}_i - p^r \sum_{s\ge i} b_s\phi_s$.
Note that $tr_i \psi^{(r+1)}_i = tr_i \psi^{(r)}_i-p^r b_i tr_i \phi_i$,
and if $r>0$, then by induction $tr_i \psi^{(r+1)}_i$ is a generator of additive operations. 
In the induction step $r=0\rarr r=1$ similar arguments apply as 
in \cite[p. 33, Lemma 4.4.1 follows from Lemma 4.4.2]{Sech}
to show that $tr_i \psi^{(1)}_i$ is a generator of additive operations.

Having finished the induction
we define $c_i$ as $P_i+\frac{\psi^{(\mu_i)}_i}{p^{\mu_i}}$
which satisfied demanded properties.
\end{proof}

Thus, we have constructed Chern classes from $\Kn^*$ to the universal $p^n$-typical theory $BP\{n\}^*$.
As already mentioned we define then Chern classes $c_i^A$ from $\Kn^*$ to any $p^n$-typical theory $A^*$
as compositions of $c_i^{BP\{n\}}$ with the unique morphism of theories $\pi_A\colon BP\{n\}^*\rarr A^*$.
Operations $c_i^A$ satisfy the Cartan-type formula 
 since any morphism of theories is a multiplicative map.
To finish the construction we need to check that $tr_i c_i^A = c_i^{\CH}$.

\begin{Prop}
Operations $c_i^A$ take values in $\tau^iA^*$,
and $tr_i c_i^A = c_i^{\CH}\ot \id_A$. 
\end{Prop}
\begin{proof}
Every morphism of theories preserves the topological filtration.
In particular, $\pi_A$ maps $\tau^i BP\{n\}^*$ to $\tau^i A^*$.

On products of projective spaces $\pi_A$ sends $b z^{BP\{n\}}_1\cdots z^{BP\{n\}}_l$
to $\pi_A(b) z_1^A\cdots z_l^A$,
and thus by the construction of the truncation map on products of projective spaces
we have $tr_i (\pi_A \circ c_i^{BP\{n\}})=\pi_A(tr_i c_i^{BP\{n\}})$.
However, $tr_i c_i^{BP\{n\}}$ takes values in $\CH^i\ot\Zp \subset \CH^i\ot BP\{n\}$,
i.e.\ coefficients of the corresponding polynomials are integral.
Being a multiplicative map, $\pi_A$ acts on these coefficients canonically (if not to say trivially).
\end{proof}

\subsection{Chern classes generate all operations}\label{sec:chern_generate_all}

The following proposition finishes the proof of Theorem \ref{th:main}.

\begin{Prop}
Let $A^*$ be a $p^n$-typical theory s.t. $A$ is a free $\Zp$-module.

Then all operations from $\Kn^*$ to $A^*$ 
are uniquely expressible as series in Chern classes:
$$ [\Kntilde^*, A^*] = A[[c_1,\ldots, c_i,\ldots]].$$

Moreover, the analogous statement is true for poly-operations.
\end{Prop}
\begin{proof}
If $A$ is a free $\Zp$-module, then presheaf of abelian groups $\CH^*\ot A$
 is isomorphic to a direct sum of $\CH^*\ot \Zp$.
Thus, $A$-module of operations $[\Kntilde^*,CH^*\ot A]$ 
is isomorphic to $A\ot[\Kntilde^*,CH^*\ot \Zp]$.

The $\Zp$-algebra $[\Kntilde^*,CH^*\ot \Zp]$ is freely generated by Chern classes
$c_i^{CH}$ as was shown in \cite[Th. 4.2.1]{Sech},
and in Section \ref{sec:construction_chern_classes} the lifts of these operations 
with respect to the truncation map were constructed.
By Proposition \ref{prop:relations_between_truncations} the claim now follows.

A similar proof works for poly-operations based on the same statement \cite[Th. 4.2.1]{Sech}
for the classification of poly-operations to Chow groups.
\end{proof}

\section{On the uniqueness of Morava K-theories}\label{sec:morava_unique}

In this section we construct an additive isomorphism 
between every two $n$-th Morava K-theories. In order to do this
recall the classification of additive endooperations in $\Kn^*$ 
from Corollary \ref{cr:add_op_Kn-pntyp}.

There exist additive operations $\phi^{\Kn}_i\colon \Kn^i\rarr \tau^i \Kn^i$ for each $i\ge 0$,
such that every additive endo-operations in $\Kntilde^*$ can be 
uniquely represented as an infinite sum of operations $\phi_i$ with $\Zp$-coefficients.
We can take $\phi_0$ to be a projection to the canonical summand $\Zp$ of $\Kn^*$
(in other words, $\phi_0$ is an analogue of virtual rank for K-theory).

\begin{Prop}\label{prop:kn_invertible_endo_add}
An additive endo-operation $\phi:=\sum_{i\ge 0} a_i \phi^{\Kn}_i$ is an isomorphism
if and only if $a_i \in \Zp^{\times}$ for $0\le i\le p^n-1$.
\end{Prop}
\begin{proof}
Since any additive operation preserves the topological filtration,
operations $\phi_i$, $i\ge 1$ are supported on $\Kntilde^*=\tau^1 \Kn^*$.
It is clear then, that the operation $\phi$
is an isomorphism if and only if the operation $\sum_{i\ge 1} a_i \phi^{\Kn}_i$
is an isomorphism on $\Kntilde^*$ and $a_0\in\Zp^\times$.

To prove the proposition it is useful to understand
the `coordinates' of operations $(\phi^{\Kn}_i)^{\circ 2}$ 
in the `basis' $\phi^{\Kn}_i$.

\begin{Lm}\label{lm:comp_phi}
For each $i\ge 1$ there exist numbers $\beta_i\in\Zp$
s.t. for some $b^i_k\in \Zp, k>i$ the following equation holds:
\begin{equation}\label{eq:phi_comp}
(\phi^{\Kn}_i)^{\circ 2}= \beta_i\phi^{\Kn}_i + \sum_{k>i} b^i_k \phi^{\Kn}_k.
\end{equation}

Moreover, $\beta_i\in\Zp^\times$ for $i<p^n$, and $\beta_i\in p\Zp\setminus\{0\}$ for $i\ge p^n$.
\end{Lm}
\begin{proof}[Proof of the Lemma.]

An endo-operation $(\phi^{\Kn}_i)^{\circ 2}$ takes values in $\tau^i \Kn^*$
and therefore is represented by an infinite sum of operations $\phi_k$, $k\ge i$,
with $\Zp$-coefficients. Denote the coefficients in this sum as in (\ref{eq:phi_comp}).

Denote by $\eta_i\in\Zp$ the coefficient of the monomial
 $z_1\cdots z_i$ in the series $G_i$ corresponding
to the operation $\phi_i$. We claim that $\eta_i = \beta_i$.
Indeed, the series $G_i$ of the composition $\phi_i^{\circ 2}$
has the coefficient $\eta_i^2$. Since for operations $\phi_j$, $j>i$,
 the corresponding coefficient is equal to 0, we have $\eta_i^2 = \beta_i \eta_i$.
Therefore either $\eta_i=0$ ot $\beta_i=\eta_i$.

By the construction the $i$-th truncation of $\phi_i$ 
has the polynomial $G_i$ equal to $\eta_i z_1\cdots z_i$.
Therefore, in order to prove the lemma we need to show that
for a generator of additive operations from $\Kn^*$ to $\CH^i\ot\Zp$
the coefficient $\eta_i\neq 0$ for all $i\ge 1$,
$\eta_i\in \Zp^\times$ for $i\le p^n-1$, and $\eta_i\in p\Zp$ for $i\ge p^n$.

An additive operation $ch_i\colon \Kn^*\rarr \CH^i\ot\QQ$
has the polynomial $G_i$ equal to $z_1\cdots z_i$.
However, the vector space over $\QQ$ of additive operations $[\Kn^*,\CH^i\ot\QQ]^{add}$
is 1-dimensional by \cite[Prop.~3.6]{Sech}, and therefore a generator of integral additive operations
is proportional to $ch_i$. In particular, $\eta_i \neq 0$.

One checks by a direct computation that $ch_i$ acts integrally on products 
of projective spaces for $i\le p^n-1$, and therefore $\eta_i \in \Zp^\times$ for $i\le p^n-1$.

On the other hand $ch(z)=\log^{-1}_{\Kn}(z)=z-\frac{a_1}{p}z^{p^n}+\ldots$ for some $a_1\in\Zp^\times$,
and the polynomial $G_{i-{p^n-1}}$ of the operation $ch_i$
is equal to $-\frac{a_1}{p}\sum_{j=1}^{i-{p^n-1}} z_1\cdots z_j^{p^n} \cdots z_{i-{p^n-1}}$
which is not integral. Thus, a generator of additive operations to $\CH^i\ot\Zp$
is proportional to $ch_i$ with the coefficient in $p\Zp$, and $\eta_i\in p\Zp$.
\end{proof}

Let $\phi:=\sum_{i=1}^\infty a_i \phi^{\Kn}_i$ be an additive endo-operation of $\Kntilde^*$
which is an isomorphism.  The inverse of $\phi$ is also an operation, denote it by $\psi$.

The series $G_i$ of the operation $\phi$ for $i\colon 1\le i\le p^n-1$
equals to $a_i\beta_i z_1\ldots z_i+\mathrm{higher\ degree\ terms}$,
where $\beta_i$ is defined in Lemma \ref{lm:comp_phi}.
Indeed, operations $\phi_j^{\Kn}$ do not contribute to the series $G_i$ if $j>i$,
since they take values in higher part of the topological filtration.
Also, operation $\phi_j^{\Kn}$ is supported on $\Kn^{j \mod p^n-1}$,
which means that series $G_r$ corresponding to it are zero for $r\not\equiv j \mod p^n-1$.
Therefore only $\phi_i^{\Kn}$ contributes to the starting monomial of 
the series $G_i$ of the operation $\phi$.

Suppose that the series $G_i$ of the operation $\psi$
is equal to $\alpha_i z_1\ldots z_i+\mathrm{higher\ degree\ terms}$ for some number $\alpha_i\in\Zp$.
Therefore the series $G_i$, $i\le p^n-1$ of the composition $\psi\circ \phi=\id_{\Kn}$
starts with $\alpha_i a_i \beta_i z_1\ldots z_i$, and on the other hand
is equal to $z_1\ldots z_i$, since $\psi\circ \phi=\id$.
 Thus, we may conclude that $a_i \in \Zp^\times$,
and one part of the proposition is proved.

Now assume that $a_i\in\Zp^\times$ for $i\le p^n-1$ and let us
 construct the inverse of the operation $\phi$.
 We do this by induction approximating the inverse of $\phi$
 modulo parts of the topological filtration.

{\bf Claim.} There exists an operation $\psi_k\colon \Kntilde^*\rarr \Kn^*$ 
s.t. $\psi_k\circ \phi -\id$ takes values in $\tau^k \Kn^*$.

{\bf Base of induction ($k=p^n$).}

Define the operation $\psi_{p^n}:=\sum_{i=1}^{p^n-1} \frac{1}{a_i\beta_i^2}\phi_i$.
Considerations as above show that for the composition $\phi\circ \psi_{p^n}$
the series $G_i$ starts with $a_i\beta_i\cdot \beta_i \cdot (a_i\beta_i^2)^{-1} z_1\cdots z_i=z_1\cdots z_i$
for $i\le p^n-1$. On the other hand series $G_i$ are gradable for every additive operation
(Prop. \ref{prop:add_grad}), 
and for the operation $\phi\circ \psi_{p^n} - \id$
the series $G_i$ starts with monomials of degrees higher than $i$ for $i\le p^n-1$,
and therefore by monomials of degree higher or equal to $p^n$. 
Thus, by Proposition \ref{prop:operations_top_COT} the operation
$\phi\circ \psi_{p^n} - \id$ takes values in $\tau^{p^n} \Kn^*$.

{\bf Induction step ($k \rarr k+1$, $k\ge p^n$).}

By the induction assumption we have $\psi_k\circ \phi =\id + \sum_{i\ge k} \alpha_i \phi^{\Kn}_i$ 
for some $\alpha_i\in \Zp$.
Let us find $x\in\Zp$ such that the operation 
$\psi_{k+1}=(\id+x\phi^{\Kn}_k)\circ \psi_k$
is the next approximation of the inverse of $\phi$,
i.e.\ $\psi_{k+1}\circ \phi - \id$ takes values in $\tau^{k+1}\Kn^*$.

Using Lemma \ref{lm:comp_phi} we have 
$$\psi_{k+1}\circ \phi = (\id+x\phi^{\Kn}_k) \circ \psi_k\circ \phi 
= \id + x\phi^{\Kn}_k +\alpha_k\phi^{\Kn}_k+x\alpha_k \beta_k \phi^{\Kn}_k
+\sum_{i>k}\delta_i\phi^{\Kn}_i,$$
where $\delta_i\in\Zp$.
Therefore setting $x=-\frac{\alpha_k}{1+\alpha_k\beta_k}$ the claim is obtained.
Note that $x\in\Zp$ since $\beta_k\in p\Zp$ for $k\ge p^n$ by Lemma \ref{lm:comp_phi}.

The infinite induction process `converges' as infinite linear combinations
of operations $\phi_i$ are well-defined.
Thus, we have constructed the left inverse of $\phi$,
however, the same argument applies to construct the right inverse of $\phi$,
and therefore $\phi$ is an isomorphism.
\end{proof}

\begin{Th}\label{th:morava_unique}
Let $\Kn^*$, $\oKn^*$ be two $n$-th Morava K-theories over $\Zp$.

Then there exist an isomorphism of presheaves of abelian groups 
$\Kn^*\xrarr{\sim} \oKn^*$.
\end{Th}
\begin{proof}
In what follows we prove the isomorphism
between $\Kntilde^*$ and $\tilde{\oKn}^*$,
however, we omit $\sim$ in what follows for the clarity of reading.
Clearly, this implies the statement of the theorem.

By Corollary \ref{cr:add_op_Kn-pntyp} for $i\ge 1$ there exist additive
operations $\pphi_i:=\phi_i^{\Kn\rarr\oKn}\colon \Kn^i\rarr \oKn^i$
and $\psi_i:=\phi_i^{\oKn \rarr \Kn}\colon \oKn^i\rarr \Kn^i$
such that $\pphi_i$, $\psi_i$ take values in the $i$-th part of the topological filtration
and their truncations are generators of additive operations to $\CH^i\ot\Zp$.

Denote $\pphi=\sum_{i=1}^{p^n-1} \pphi_i$ and $\psi = \sum_{i=1}^{p^n-1} \psi_i$.
We will show now that compositions of these operations are isomorphisms
of $\Kn^*$ and $\oKn^*$, respectively. It would follow that both these operations are isomorphisms.
Clearly, it is enough to consider only one composition by symmetry of the argument.

Let us show that for $i\colon 1\le i\le p^n-1$ the series $G_i$ of operations $\pphi$ and $\psi$ 
start with $a_i z_1\ldots z_i$ and $b_i z_1\ldots z_i$, respectively,
where $a_i, b_i\in \Zp^\times$.
Let us consider only the operation $\pphi$. 
Since operations $\pphi_i$ are supported on $\Kn^i$,
the series $G_i$ of $\pphi$ for $i\colon 1\le i\le p^n-1$
equals to the series $G_i$ of $\pphi_i$.
The truncation of $\pphi_i$ is a generator of additive operations from $\Kn^i\rarr \CH^i\ot\Zp$,
and its corresponding polynomial $G_i$ is exactly $a_i z_1\ldots z_i$.
However, $ch_i$ is an integral generator of additive operations to $\CH^i\ot\Zp$ for $i\le p^n-1$,
its polynomial $G_i$ is equal to $z_1\cdots z_i$.
The operation $tr_i \pphi_i$ has to be proportional to $ch_i$ by a number in $\Zp^\times$,
and therefore $a_i\in\Zp^\times$.

We can represent $\psi\circ \pphi$ as a sum $\sum_{i\ge 1} \alpha_i \pphi_i$ for some $\alpha_i \in \Zp$.
The series $G_i$ of this operation starts with $\alpha_i a_i z_1\ldots z_i$ for $i\colon 1\le i\le p^n$.
However, calculating the series $G_i$ using the composition of corresponding series of $\psi$ and $\phi$,
one obtains that it starts with $a_ib_i z_1\ldots z_i$.
Thus, $\alpha_i=b_i\in\Zp^\times$ for $i\colon 1\le i\le p^n$, 
and it follows by Proposition~\ref{prop:kn_invertible_endo_add} that $\psi\circ \pphi$  is an isomorphism.
\end{proof}

\begin{Rk}
Not all of $n$-th Morava K-theories are multiplicatively isomorphic,
see  Definition~\ref{def:mult_iso} and Appendix \ref{app:mor_not_mult}.
Note, however, that every two formal group laws over $\overline{\mathbb{F}}_p$
are isomorphic by a theorem of Lazard, and therefore 
there exist a multiplicative isomorphism between free theories 
$\Kn^*\ot\overline{\mathbb{F}}_p$ and $\oKn^*\ot\overline{\mathbb{F}}_p$
for every two $n$-th Morava K-theories $\Kn^*, \oKn^*$.
\end{Rk}

\begin{Rk}
Theorem \ref{th:morava_unique} suggests
that there is a uniquely defined motivic spectrum representing $n$-th Morava K-theory.
Since there exist many different multiplicative structures 
on the $n$-th Morava K-theory $\Kn^*$, it is also possible that this hypothetical
unique spectrum admits several multiplicative structures, 
or admits none of them. These questions are addressed in topology \cite{Rob, Ang},
however, are open in motivic homotopy, as to the author's knowledge.
\end{Rk}

\section{The gamma filtration on Morava K-theories}\label{sec:morava_gamma_filtration}

In this section we define and investigate properties 
of a functorial filtration on $\Kn^*(X)$ for all $n\ge 1$
which we call the gamma filtration.
The definition of this filtration
is verbatim the definition of the gamma filtration on $\KK$
as defined with Chern classes. 

The gamma filtration on $\Kn^*$ 
has the universal property as the best approximation to the topological filtration
which is defined by values of poly-operations (Prop. \ref{prop:gamma_unique}).
Thus, even though Chern classes which we have constructed in Theorem~\ref{th:main} are not unique,
the gamma filtration does not depend on their choice
and is strictly compatible with additive isomorphisms between $n$-th Morava K-theories (Prop. \ref{prop:gamma_isom_compatible}).
In other words, the gamma filtration is unique on the presheaf of abelian groups $\Kn^*$.
As expected the gamma filtration on $K(1)^*$ coincides with the gamma filtration on $\KK\ot\Zp$ 
under an additive isomorphism between them (Prop. \ref{prop:gamma_K(1)_K0}).

An important difference between properties of gamma filtrations on $\KK$ and $\Kn^*$, $n\ge 2$,
is the fact that 
Chern classes $c_i^{\Kn\rarr \CH}$ are
additive surjective homomorphisms from the $i$-th graded piece $\gamma^i \Kn^*/\gamma^{i+1} \Kn^*$
 to $\CH^i\ot\Zp$ for $i\le p^n$ (in the case of $\KK$ we have surjectivity only for $i\le p$), see Prop. \ref{prop:morava_gamma_properties}.
This yields a combinatorial tool
for the study of $p$-torsion in Chow groups
for geometrically cellular varieties $X$ such that
the pullback morphism $\Kn^*(X)\rarr \Kn^*(X\times_k \bar{k})$ is an isomorphism.
Pfister quadric is a well-known example of a variety satisfying this property, and we show
in Section~\ref{sec:constant_chern} how the calculations work out in this case.
New examples of varieties satisfying the property above are studied in \cite{SechSem}.

\subsection{Definitions and properties}\label{sec:morava_gamma_def_properties}

In what follows we denote by $c_i$, $c_i^{\CH}$ Chern classes from $\Kn^*$ to itself and to Chow groups, respectively,
which were constructed in Theorem~\ref{th:main}.

\begin{Def}
Define the gamma filtration on $\Kn^*$ of a smooth variety by the following formulas:
$$ \gamma^0 \Kn^*(X)=\Kn^*(X),$$
$$\gamma^m \Kn^*(X):= < c_{i_1}(\alpha_1)\cdots c_{i_k}(\alpha_k)| \sum_j i_j\ge m,\alpha_j \in \Kn^*(X)>, \quad m\ge 1, $$ 
where $<,>$-brackets denote generation as $\Zp$-modules.
\end{Def}

It is clear from the definition that $\gamma^m \Kn^*$ is an ideal subpresheaf of $\Kn^*$ 
for all $m\ge 0$.

\begin{Prop}\label{prop:morava_gamma_properties}
The gamma and the topological filtrations on $\Kn^*$ satisfy the following properties:

\begin{enumerate}[i)]
\item\label{item:gr_gamma_support} $gr_\gamma^i\Kn^*=gr_\gamma^i\Kn^{i\mod p^n-1}$;
\item\label{item:gamma_topological} $\gamma^i \subset \tau^i$;
\item\label{item:c_i-is-zero-on-tau^i+1} $c_i^{CH}$ is zero on $\tau^{i+1}\Kn^*$;
\item\label{item:c_i-is-additive-on-tau^i}
 the operation $c^{CH}_i$ is additive when restricted to $\tau^i \Kn^*$;
\item\label{item:c_i-gr-tau} the additive map $c^{CH}_i\colon gr_\tau^i\Kn^*\rarr \CH^i\ot\Zp$ 
 \begin{enumerate}[1)]
\item\label{item:c_i-is-isom-Q-tau} is an isomorphism after tensoring with $\QQ$;
\item\label{item:c_i-is-isom-tau} is an isomorphism for $1\le i\le p^n$
and $\rho_{\Kn}$ is its inverse;
\end{enumerate}
\item\label{item:c_i-gr-gamma} the additive map $c^{CH}_i\colon gr_\gamma^i\Kn^*\rarr \CH^i\ot\Zp$ 
 \begin{enumerate}[1)]
\item\label{item:c_i-is-isom-Q-gamma} is an isomorphism after tensoring with $\QQ$;
\item\label{item:c_i-is-surj-gamma} is an isomorphism for $i=1$ and surjective for $1\le i\le p^n$;
\end{enumerate}
\item\label{item:gamma_top_Q} $\gamma^i\ot\QQ = \tau^i\ot\QQ$.
\end{enumerate}
\end{Prop}

\begin{proof}
\ref{item:gr_gamma_support}) Recall that the Chern class $c_i\colon \Kn^*\rarr BP\{n\}^i$
takes values in the $i$-th graded part, and as the classifying morphism of theories
from $BP\{n\}^*$ to $\Kn^*$ respects the grading, 
the $i$-th Chern class from $\Kn^*$ to itself takes values in $\Kn^i$.

The space of all polynomials in Chern classes 
is split into $p^n-1$ summands by their degree modulo $p^n-1$.
It is clear that the filtration by the degree of polynomials 
jumps on each summand every $p^n-1$ steps, and thus the same is true
for the gamma filtration. In parituclar, $gr^i_\gamma \Kn^j=0$ if $j\not\equiv i \mod p^n-1$.

\ref{item:gamma_topological}) The Chern class $c_i$ takes values
 in $\tau^i\Kn^*$ by Theorem~\ref{th:main},~\ref{item:c_j-tau^j}).
Since the topological filtration is multiplicative (Prop. \ref{prop:top_filt_properties}),
this property follows.

\ref{item:c_i-is-zero-on-tau^i+1}) Follows
from the fact that any operation (which preserves 0) preserves the topological filtration,
and that $\tau^{i+1} \CH^i =0$.

\ref{item:c_i-is-additive-on-tau^i}) By the Cartan formula
the internal derivative of the operation $c_i$
is expressible as a polynomial in operations $c_j, j<i$.
However, these operations vanish on $\tau^i$ by the property \ref{item:c_i-is-zero-on-tau^i+1}),
and therefore $c_i$ is additive.

\ref{item:c_i-gr-tau}) We will need the following Lemmata.
\begin{Lm}\label{lm:c_i-on-gr-top-Kn}
Let $i\ge 1$. Denote by $a_i \in \Zp$ 
the coefficient of the monomial $z_1\ldots z_i$
of $c_i(z_1\ldots z_i)$ in the notation of Section \ref{sec:cont}.

Then the following diagram commutes:
\begin{center}
\begin{tikzcd}
\CH^i\ot\Zp \arrow[r, "\cdot a_i"] \arrow{d} & \CH^i\ot\Zp \arrow{d} \\
gr^i_\tau \Kn^* \arrow[r, "c_i"] & gr^i_\tau \Kn^*
\end{tikzcd}
\end{center}
Moreover, $a_i \in \Zp^\times$ for $i\colon 1\le i \le p^n$, and $a_i\neq 0$ for all $i\ge 1$.
\end{Lm}
\begin{proof}[Proof of the Lemma.]
This is a version of Proposition~\ref{prop:operations_chow_topfilt} with the claim that
the corresponding operation on Chow groups is the multiplication by $a_i$.
The operation $gr^i_\tau \Kn^* \xrarr{c_i} gr^i_\tau \Kn^*$ is additive by \ref{item:c_i-is-additive-on-tau^i}),
and therefore the corresponding operation on $\CH^i\ot\Zp$ is also additive.
By Vishik's Theorem \ref{th:Vish_op} an additive operation $\CH^i\ot\Zp\rarr \CH^i\ot\Zp$
is defined by the additive map $G_i\colon \Zp\rarr \Zp\cdot z_1\cdots z_i$, i.e.
it is a multiplication by the coefficient of the monomial $z_1\cdots z_i$ in $G_i(1)$.
A computation on products of projective spaces shows that this is the coefficient $a_i$
as specified in the Lemma.

Note that the coefficient of the monomial $z_1\cdots z_i$
in $c_i(z_1\cdots z_i)$ is the same as the coefficient of 
the monomial $z_1^{\CH}\cdots z_i^{\CH}$ for $(tr_i c_i = c_i^{\CH})(z_1\cdots z_i)$ by the construction of the truncation.
Thus, we may work with Chern classes to Chow groups instead.
For $i\colon 1\le i \le p^n-1$ we have $a_i \in \Zp^\times$ 
because in this range Chern classes are additive operations and $c_i^{\CH}$ can be chosen as $ch_i$ as an operation to $\CH^i\ot\QQ$
(this was already discussed in the proof of Lemma \ref{lm:comp_phi}).
For $i=p^n$ the construction of Chern classes yields $c^{\CH}_{p^n}=\frac{(c^{\CH}_1)^{p}+\phi_{p^n}}{p}$
and one checks that a generator of integral additive operations $\phi_{p^n}$ is equal to $a_{p^n}p\cdot ch_{p^n}$
where $a_{p^n}\in\Zp^{\times}$ (see e.g.\ Appendix \ref{app:image_of_Chern_Chow}, Example \ref{app:exl_add_image}). 
Since $(c_1^{\CH})^{p^n}$ is zero on $z_1\cdots z_{p^n}$  we see that this number $a_{p^n}$ is the coefficient we
are looking for.

The last claim is that $a_i\neq 0$ for every $i\ge 1$.
As above we can consider the operation $c_i^{\CH}$.
By \cite[Lemma 4.4.1]{Sech} the following relation $c_i^{\CH}=P_i(c^{\CH}_1,\ldots, c^{\CH}_{i-1})+\frac{\phi_i}{p^{\mu_i}}$
holds for the action of operations on projective spaces where $P_i$ is a certain polynomial, $\phi_i$ is a 
generator of integral additve operations $\Kn^*\rarr \CH^i\ot\Zp$.
Chern classes $c^{\CH}_j$ send $z_1\cdots z_i$ to zero if $j<i$ because of the continuity of operations,
and therefore so does the polynomial  $P_i(c^{\CH}_1,\ldots, c^{\CH}_{i-1})$. Thus, it is enough 
to show that the operation $\phi_i$ sends $z_1\cdots z_i$ to $a_ip^{\mu_i} z_1\cdots z_i$
where $a_i\neq 0$. However, the space of additive operations from $\Kn^*$ to $\CH^i\ot\QQ$ is 1-dimensional
by \cite[Prop.~3.3.1]{Sech}, and therefore $\phi_i$ is rationally proportional to $ch_i$ with a non-zero coefficient,
and $ch_i$ sends $z_1\cdots z_i$ to $z_1\cdots z_i$. This proves the claim.
\end{proof}

\begin{Lm}\label{lm:comp_CH_c_i_CH}
The composition $\CH^i\ot\Zp \twoheadrightarrow gr^i_\tau \Kn^* \xrarr{c^{\CH}_i} \CH^i\ot\Zp$
is multiplication by $a_i$ (the same number as in Lemma \ref{lm:c_i-on-gr-top-Kn}).
\end{Lm}
\begin{proof}[Proof of the Lemma.] 
The map $\CH^i\ot\Zp \twoheadrightarrow gr^i_\tau \Kn^*$ is a surjective additive operation
by Lemma \ref{lm:rho_morphism_of_theories}, and thus the composition
is an additive operation. Calculating the composition on the element $z_1\cdots z_i \in \CH^i((\mathbb{P}^\infty)^{\times i})$
we obtain the result.
\end{proof}

\ref{item:c_i-is-isom-Q-tau}), \ref{item:c_i-is-isom-tau}) now follow
 from Lemmata \ref{lm:c_i-on-gr-top-Kn} and \ref{lm:comp_CH_c_i_CH}.

\ref{item:c_i-gr-gamma}) We will first prove the surjectivity statement \ref{item:c_i-is-surj-gamma}).
The statement \ref{item:c_i-is-isom-Q-gamma}) will follow from \ref{item:c_i-is-isom-Q-tau})
and \ref{item:gamma_top_Q}).

\begin{Lm}\label{lm:comp_CH_c_i_c_i_CH}
The composition 
\begin{center}
\begin{tikzcd}
\CH^i\ot\Zp \arrow[two heads, r, "\rho_{\Kn}"{yshift=2pt}]  & gr^i_\tau \Kn^* \arrow[r, "c_i"] 
& gr^i_\tau \Kn^* \arrow[r, "c^{\CH}_i"] &\CH^i\ot\Zp
\end{tikzcd}
\end{center}
is multiplication by $a_i^2$ (the number $a_i$ as in Lemma \ref{lm:c_i-on-gr-top-Kn}).
\end{Lm}
\begin{proof}[Proof of the Lemma.]
The composition is an additive operation, and it is enough to consider its value on 
the element $z_1^{\CH}\cdots z_i^{\CH}$ of $\CH^i((\mathbb{P}^\infty)^{\times i})\ot\Zp$ (see Section \ref{sec:cont} 
for the notation).
 The morphism of the theories $\rho_{\Kn}$ sends it to $z_1 \cdots z_i$,
 the operation $c_i$ sends this elements to $a_i z_1\cdots z_i$ 
 as in the proof of Lemma~\ref{lm:c_i-on-gr-top-Kn}. Similarly, $c_i^{\CH}$ sends
$z_1 \cdots z_i$ to $a_i z_1 \cdots z_i$, and, thus, sends $a_i z_1 \cdots z_i$
to $a_i^2 z_1 \cdots z_i$.
\end{proof}

By Lemma~\ref{lm:comp_CH_c_i_c_i_CH} if  $a_i\in\Zp^\times$, then 
the map $gr^i_\tau \Kn^* \xrarr{c_i} gr^i_\tau \Kn^*$ is an isomorphism.
In particular, this is true for $i\le p^n$ by Lemma~\ref{lm:c_i-on-gr-top-Kn}.
The image of this map lies in $\gamma^i/\tau^{i+1}\Kn^i\subset gr^i_\tau \Kn^i$, 
and thus we see that $\tau^i=\gamma^i$ on $\Kn^i$ in this range.
 Together with \ref{item:c_i-is-isom-tau})
this implies \ref{item:c_i-is-surj-gamma}).

\ref{item:gamma_top_Q}) The inclusion $\gamma_\QQ\subset \tau_\QQ$ is known
by \ref{item:gamma_topological}). 
Note that the map $\CH^i\ot\QQ\rarr gr^i_\tau \Kn^*\ot\QQ$ is an isomorphism,
and it follows from Lemma \ref{lm:c_i-on-gr-top-Kn} and \ref{item:c_i-is-isom-tau})
that $c_i$ acts a multiplication by a non-zero number $a_i$ on $gr^i_\tau \Kn^*\ot\QQ$.
In particular, if $x\in \tau^i \Kn^*(X)_\QQ$ for a smooth variety $X$,
then $a_ix-c_i(x)\in \tau^{i+1}\Kn^*(X)_\QQ$.

For a variety $X$ of dimension $d$ we prove by a decreasing induction 
that $\gamma^i\Kn^*\ot\QQ=\tau^i\Kn^*\ot\QQ$.

{\bf Base of induction.} The equality holds for $i=d+1$,
since $\tau^{d+1}\Kn^*(X)=0$ for dimensional reasons, and 
polynomials of degree $\ge d+1$ in Chern classes $c_i^{BP\{n\}}$ are zero because $BP\{n\}^{\ge d+1}(X)=0$
and therefore $\gamma^{d+1}\Kn^*(X)=0$.

{\bf Induction step.} Let $x\in \tau^i \Kn^*(X)_\QQ$
and define $y_i(x) = a_ix-c_i(x)\in \tau^{i+1}\Kn^*(X)_\QQ$.
The element $x+\frac{1}{a_i}y_i(x)$ lies in $\gamma^i\Kn^*(X)_\QQ$,
therefore it is enough to show that $y_i(x) \in \gamma^i\Kn^*(X)_\QQ$.
However, we know that $y_i(x)\in \tau^{i+1}\Kn^*(X)_\QQ$
and by induction assumption we know that $y_i(x)\in \gamma^{i+1}\Kn^*(X)_\QQ$.
Thus, $x\in \gamma^i_\QQ \Kn^*(X)$.
\end{proof}

\subsection{Application: $\Kn$-motives of Tate type.}

For the following corollary we need the notion of $A^*$-motive
where $A^*$ is an oriented theory.
This is an analogue of Grothendieck pure motives 
with Chow groups replaced by $A^*$.
We refer the reader to \cite{Manin, NenZai} for details.
Note that we call $A^*$-motive $M$ {\sl Tate} if it is a direct sum
of powers of the Lefschetz motive in the terminology of \cite{NenZai}.
For a smooth projective variety $X$ we denote by $\M_{A}(X)$ its $A^*$-motive.
Note that by \cite{VishYag}
for any oriented cohomology theory $A^*$ there is a canonical map
from the set of isomorphism classes of indecomposable objects of the category of Chow motives
to the set of isomorphism classes of objects of the category of $A^*$-motives,
 i.e.\ there is a way to ``specialize''
an irreducible Chow motive to $A^*$-motive, but it may further decompose then.

\begin{Cr}\label{cr:Kn_tate_small_dim}
Let $X$ be a smooth projective variety such that $\dim X\le p^n-2$, 
then $$\mathrm{End\ } \M_{\Kn}(X)=\mathrm{End\ } \M_{\CH\ot\Zp}(X),$$
and therefore the decompositions of motives $\M_{\Kn}(X)$, $\M_{\CH\ot\Zp}(X)$ are the same,
i.e.\ the canonical lift of indecomposable summands of $M_{\CH\ot\Zp}(X)$ 
stays irreducible in $\Kn^*$-motives.

Moreover, $\Kn^*$-motive of $X$ is Tate if and only if $\CH^*\ot\Zp$-motive of $X$ is Tate.
\end{Cr}

Before the proof we unravel the condition that the motive of $X$ is Tate
in the following straight-forward lemma.

\begin{Lm}\label{lm:tate_motive}
 Let $A^*$ be an oriented theory, let $X$ be a smooth projective variety over $k$.
Denote by $\pi\colon  X\rarr \Spec k$ the structural morphism, 
and by $p_1, p_2\colon  X\times X \rarr X$ the canonical projections.

Then the $A^*$-motive of $X$ is Tate if and only if there exist  
elements $a_i, b_i \in A^*(X)$ for a finite set of indices $i\in I$,
such that 
\begin{enumerate}

\item ($(p_1^*)(a_i)\cdot (p_2^*)(b_i)$ is a projector)\\
 $\pi_*(a_i\cdot b_i) = 1$ in $A^0(\Spec k)$ for all $i\in I$;
\item (projectors $(p_1^*)(a_i)\cdot (p_2^*)(b_i)$ and 
$(p_1^*)(a_j)\cdot (p_2^*)(b_j)$ are orthogonal)\\
 $\pi_*(a_i\cdot b_j) = 0$ for all $i,j\in I$, $i\neq j$;
\item (projectors $(p_1^*)(a_i)\cdot (p_2^*)(b_i)$ exhaust the motive of $X$) \\
$\Delta_{A} = \sum_{i\in I} (p_1^*)(a_i)\cdot (p_2^*)(b_i)$ in $A^{\dim X}(X\times X)$,
where $\Delta_A$ is the class of the diagonal of $X$.
\end{enumerate}
\end{Lm}

\begin{proof}[Proof of Corollary~\ref{cr:Kn_tate_small_dim}]
The topological filtration restricted to a graded component of $\Kn^*$
changes every $p^n-1$ steps (Prop. \ref{prop:top_morava_grading}),
 e.g.\ $\tau^2 \Kn^1=\tau^3 \Kn^1=\ldots =\tau^{p^n} \Kn^1$.
 Moreover, $\tau^1 \Kn^*$ splits off a constant summand, by the (CONST) property, Def.~\ref{def:const_axiom}, of free theories.
Thus, for a variety $X$ of dimension less or equal to $p^n-1$
we have $gr^{\dim X}_\tau \Kn^*(X\times X)=\Kn^{\dim X}(X\times X)$.

The canonical morphism $\rho_{\Kn}\colon \CH^{\dim X}(X\times X)\ot\Zp \rarr gr^{\dim X}_\tau \Kn^*(X\times X)$
is an isomorphism of $\Zp$-modules by Proposition~\ref{prop:morava_gamma_properties},~\ref{item:c_i-is-isom-tau}).
Since $\rho_{\Kn}$ is a morphism of theories by Lemma~\ref{lm:rho_morphism_of_theories},
i.e.\ commutes with pull-backs, push-forwards and multiplication,
it follows that $\rho_{\Kn}$ is an isomorphism of algebras of correspondences.
Note also that $\rho_{\Kn}$ preserves the class of the diagonal.
Thus, if $\dim X\le p^n-2$, then $\rho_{\Kn}\colon \CH^i(X)\rarr gr^i_\tau \Kn^i(X)=\Kn^i(X)$
is an isomorphism for all $i$. 
Moreover, this isomorphism for $i=\dim X$ agrees with
 the push-forwards to the point.

Assume that $\Kn^*$-motive of $X$ is Tate, 
and let $\Delta_{\Kn} = \sum_{i\in I} (p_1^*)(a_i)\cdot (p_2^*)(b_i)$ be the decomposition 
of the diagonal as given in Lemma~\ref{lm:tate_motive}. 
Then set $\alpha_i:= \rho_{\Kn}^{-1}(a_i)$, $\beta_i := \rho_{\Kn}^{-1}(b_i)$,
and by observations above 
we get
$$\Delta_{CH} = \sum_{i\in I} (p_1^*)(\alpha_i)\cdot (p_2^*)(\beta_i),$$
which yields the needed decomposition of the Chow motive of $X$ into Tate summands.

The proof in the other direction can be done by reversing the argument,
or one could use the well-known corollary of Vishik and Yagita \cite[Cr.~2.8]{VishYag}:
If the Chow motive of $X$ is Tate,
 then the $A^*$-motive of $X$ is Tate for every oriented cohomology theory $A^*$.
\end{proof}

\begin{Rk}
If one assumes that $X$ is a projective homogeneous variety 
(or, more generally, that the Chow motive of  $X\times_k \bar{k}$ is Tate
and Rost nilpotence holds for $M_{\CH\ot\Zp}(X)$),
then one can improve the bounds in the previous corollary.
Namely, under these assumptions if $\Kn^*$-motive of $X$ is Tate
and $\dim X\le p^n$, then the Chow motive of $X$ is Tate, see \cite[Cor.~7.12]{SechSem}.
\end{Rk}

Corollary~\ref{cr:Kn_tate_small_dim} shows that if one is looking 
for smooth projective varieties such that its $\Kn^*$-motive is Tate,
but Chow motive is not Tate, then one has to consider varieties of sufficiently big dimension. 
The search of such varieties is motivated by possible applications 
of the gamma filtration discussed at the end
of Section~\ref{sec:constant_chern}.

\subsection{Uniqueness of the gamma filtration}\label{sec:gamma_unique}

The gamma filtration turns out to have another, yet similar definition
which does not mention Chern classes explicitly.
It shows that the gamma filtration is the best approximation
of the topological filtration defined by values of poly-operations.

Recall that all internal poly-operations
from $\Kntilde^*$ to $\Kn^*$ are freely generated by products of Chern classes (Th.~\ref{th:main}).
It is clear that a series in Chern classes takes
values in $\tau^{\nu} \Kn^*$ where $\nu$ is the minimal degree
of non-trivial monomials in the series, $\deg c_i=i$. 
Together with the classification of all poly-operations
this gives a description of the set of internal poly-operations
which take values in $\tau^j \Kn^*$ for any $j\ge 0$. 

\begin{Prop}\label{prop:gamma_unique}
For $m\ge 0$ denote by $F^m$ the disjoint union of internal poly-operations
from $(\Kn^*)^{\times k}$ to $\tau^m \Kn^*$ for all $k\ge 1$.

For any $m\ge 0$ we have
$$\gamma^m \Kn^*(X) := \bigcup_{P \in F^m} \mathrm{Im\ } P|_{(\Kn^*)^{\times k}(X)}.$$

In particular, the gamma filtration does not depend
 on the choice of Chern classes $c_i$ 
(satisfying properties
of Theorem \ref{th:main}).
\end{Prop}
\begin{proof}
We need only to check that if $\phi$ is an internal poly-operation which takes values in $\tau^i \Kn^*$,
then it is expressible as a series in external products of Chern classes with the minimal degree of a summand being $i$.
Assume the contrary, i.e.\ without loss of generality a polynomial $P(c_1,\ldots, c_j)$ of degree $j<i$
 in external products of Chern classes takes values in $\tau^i \Kn^*$.

Then the following diagram commutes (poly-operation version of Prop. \ref{prop:op_mod_tau_vs_trunc}): 
\begin{center}
\begin{tikzcd}
(\Kntilde^*)^{\times r} \arrow[r,"P"] \arrow[d, "{P(c_1,..., c_j)}"]  &  \tau^i \Kn^* \arrow{d} \\
\CH^i \ot \Zp \arrow[r, "\rho_{\Kn}"] & gr^i_\tau \Kn^* \\
\end{tikzcd}
\end{center}

However, $P(c_1^{\CH},\ldots, c_{i-1}^{\CH})$ takes values in $\CH^{<i}\ot\Zp$
by assumptions on degrees. Thus, in this diagram  the left down arrow is zero, and therefore $P$ is zero as well.
\end{proof}

\begin{Prop}\label{prop:gamma_isom_compatible}
Let $\Kn^*$, $\overline{K}(n)^*$ be two $n$-th Morava K-theories,
and let $\phi\colon \Kn^*\rarr \overline{K}(n)^*$ an additive isomorphism between them
(which exists by Theorem~\ref{th:morava_unique}).

Then $\phi$ is strictly compatible with the gamma filtration.
\end{Prop}
\begin{proof}
It is enough to show that $\phi$ respects gamma-filtration,
since by symmetry it will follow that its inverse also respects it,
and therefore both $\phi$ and $\phi^{-1}$ are strictly compatible with it.

Let $P$ be a series of degree $i$ in Chern classes defining an $r$-ary poly-operation
from $(\Kn^*)^{\times r}$ to $\tau^m \Kn^*$. 
Using the description of the gamma filtration in Prop. \ref{prop:gamma_unique},
it is enough to show that $\phi \circ P = \bar{P}\circ \phi^{\times r}$, 
where $\bar{P}$ is an internal poly-operation from $(\overline{K}(n)^*)^{\times r}$ to $\tau^i\overline{K}(n)^*$.

Define $\bar{P}$ as the composition $\phi \circ P \circ \psi^{\times r}$
where $\psi$ is an inverse additive operation to $\phi$. 
The operation $\bar{P}$ takes values in $\tau^i\overline{K}(n)^*$
since $P$ does and $\phi$ preserves the topological filtration.
\end{proof}

\begin{Prop}\label{prop:gamma_K(1)_K0}
Denote by $\theta\colon \KK\ot\Zp\rarr K(1)$ an invertible multiplicative operation
defined by the Artin-Hasse exponential which gives an isomorphism of $\KK$ and a first Morava K-theory
with the logarithm $\log_{K(1)}=\sum_{i=0}^\infty \frac{x^{p^i}}{p^i}$ (cf. \cite[Rem. 4.1.4]{Sech}).
Denote by $\gamma^\bu \KK$ the classical gamma-filtration on $\KK$.

Then $\theta (\gamma^i \KK \ot\Zp) = F_\gamma^i \Kn$.
\end{Prop}
\begin{proof}
The proof is essentially the same as of Proposition~\ref{prop:gamma_isom_compatible}:
one needs only to note that the usual Chern classes $c^{\KK}_i\colon \KK\rarr \KK$ 
take values in $\tau^i \KK$, and the classification
of poly-operations in $\KK$ works integrally (Theorem~\ref{basis}).
\end{proof}

\subsection{Computational constants of Chern classes for the Riemann-Roch formula}\label{sec:constant_chern}

If we have a geometrically cellular variety $X$ such that
the pull-back morphism \mbox{$p^*\colon \Kn^*(X)\rarr \Kn^*(\overline X)$} is an isomorphism,
then $p^*$ induces an isomorphism
of abelian groups $gr^i_\gamma \Kn^*(X)\cong gr^i_\gamma \Kn^*(\overline{X})$.
Thus, we can compute the graded factors of the gamma filtration of $X$ 
using only the cellular variety $\overline{X}$.
Since $gr^i_\gamma \Kn^*(X)$ maps surjectively
 to $\CH^i(X)\ot\Zp$ for $i\le p^n$ by Proposition~\ref{prop:morava_gamma_properties},
the computation of the torsion in the abelian group $gr^i_\gamma \Kn^*(\overline{X})$ provides an estimate 
of the $p$-torsion in $\CH^i(X)$ for such variety.

The first example of a variety $X$ as above was found by Voevodsky based on the results of Rost:
it is a Pfister quadric corresponding to a pure symbol of degree $n+2$ (see \cite{Voe}\footnote{Note,
however, that the results of Voevodsky are stated a little bit differently, even though the spirit of 
the result is clearly the same. Note also that he worked with
 conjectural at that moment big Morava K-theories  and not their small parts directly as we do.
See below another proof of this property of Morava K-theories of Pfister quadrics 
which follows from \cite{VishYag}.}).
Other examples are provided in \cite{SechSem}.

In order to calculate the gamma filtration of a cellular algebraic variety
it is necessary to be able to compute values of Chern classes on the classes
of subvarieties. 
The only tool to do this at the moment
is the Riemann-Roch formula (Th. \ref{th:riemann_roch}). Let us recall how it applies.

Let $\phi\colon A^*\rarr B^*$ be an operation, let $X$ be a smooth variety,
and let $i\colon Z\hookrightarrow X$ be its smooth closed subvariety of codimension $c$.
It follows from the Riemann-Roch formula that the value $\phi(i_* 1_Z)$ 
is equal to $b\cdot 1_Z$ modulo $(c+1)$-th
part of the topological filtration,
where $b\in B$ is the coefficient of $z^B_1\cdots z^B_c$ in the series $\phi(z^A_1\cdots z^A_c)$ 
(for the notation see Section \ref{sec:cont}).
We compute these coefficients for some of the Chern classes 
in the following propositions. 
Partial calculations of the gamma filtration on split quadrics 
based on these numbers are performed in~\cite{SechSem}.

\begin{Prop}\label{prop:constant_cpn}
Let $c_{p^n}$ be the operation from $\Kn^1\rarr \tau^{p^n}\Kn^1$ constructed in Theorem~\ref{th:main}.

Denote by $e_j$, $j\ge 0$, the coefficient
of the monomial $z_1\cdots z_{1+j(p^n-1)}$ in the series
$c_{p^n}(z_1\cdots z_{1+j(p^n-1)}) \in \Kn^1((\mathbb{P}^\infty)^{\times 1+j(p^n-1)})$.

Then for all $j\ge 1$ we have $e_j\in \Zp^{\times}$.
\end{Prop}

The proof of this and of the following proposition uses the next Lemma.

\begin{Lm}\label{lm:add_op_induct_coef}
Let $\phi\colon \Kn^1\rarr \Kn^1$ be an additive operation, 
such that in the notation of Section \ref{sec:cont}
\begin{multline*}
G_{1+i(p^n-1)}(z_1,\ldots, z_{1+i(p^n-1)}) = \alpha_i z_1\cdots z_{1+i(p^n-1)}
+\beta_i \sum_{k=1}^{1+i(p^n-1)} z_1\cdots z_k^{p^n} \cdots z_{1+i(p^n-1)}\\
+\delta_i  \sum_{k,s=1\colon k<s}^{1+i(p^n-1)} z_1\cdots z_k^{p^n} \cdots z_s^{p^n}\cdots z_{1+i(p^n-1)}
+\mathrm{other \ terms}, \quad i\ge 0,
\end{multline*}
where $\alpha_i, \beta_i, \delta_i \in \Zp$.

Denote by $v_n \in \Zp^\times$ the coefficient in the following equation:
$p\cdot_{\Kn}x \equiv px+v_n x^{p^n} \mod x^{p^n+1}$.

Then $\alpha_{i+1} =\alpha_i -\frac{p^{p^n}-p}{v_n}\beta_i$ for $i\ge 0$,
$\beta_{i+1}=\beta_i+\frac{p^{p^n}-p}{v_n} \delta_i$ for $i\ge 1$.
\end{Lm}
\begin{proof}[Proof of the Lemma.]
The pull-back of the Veronese map $[p]$ of degree $p$ on $\mathbb{P}^\infty$ induces
the multiplicative map $z\rarr p\cdot_{\Kn} z$ on $\Kn^*(\mathbb{P}^\infty)$.
The action of the operation $\phi$ on products of projective spaces
has to commute with pull-backs along maps $[p]\times \id^{i(p^n-1)}$ for each $i\ge 0$
which gives a non-trivial relation on series $G_i$.
Explicitly, comparing the value of \mbox{$([p]\times \id^{i(p^n-1)})^*\circ \phi$} with 
$\phi\circ ([p]\times \id^{i(p^n-1)})^*$ on $z_1 z_2\cdots z_{1+i(p^n-1)}$ we obtain the following.
The series $G_{1+i(p^n-1)}(p\cdot_{\Kn}z_1,z_2,\ldots,z_{1+i(p^n-1)})$
has to be equal 
to $pG_{1+i(p^n-1)}(z_1,z_2,\ldots,z_{1+i(p^n-1)})+v_n G_{1+(i+1)(p^n-1)}(z_1^{\times p^n}, z_2,
\ldots,z_{1+i(p^n-1)})$
modulo $(z_1^{p^n+1})$.

Calculating the coefficient of the monomial $z_1^{p^n}z_2\cdots z_{1+i(p^n-1)}$ in this equation
we obtain: $v_n\alpha_{i+1}+p\beta_i=v_n\alpha_i+p^{p^n}\beta_i$.
Similarly calculating the coefficient of the monomial $z_1^{p^n}z_2^{p^n}z_3\cdots z_{1+i(p^n-1)}$ 
we obtain: $v_n\beta_{i+1}+p\delta_i=v_n\beta_i+p^{p^n} \delta_i$.
Note, however, that we have assumed in the last conclusion that $1+i(p^n-1)\ge 2$,
i.e.\ $i\ge 1$.
\end{proof}

\begin{proof}[Proof of Proposition~\ref{prop:constant_cpn}.]
The operation 
$c_{p^n}\colon \Kn^1\rarr \tau^{p^n}\Kn^1$ as defined by the construction of Section \ref{sec:construction_chern_classes}
acts on products of projective spaces as $\frac{c_1^{p^n}+\psi_{p^n}}{p}$
where $\psi_{p^n}$ is an additive endo-operation which takes values in $\tau^{p^n} \Kn^1$
and its truncation $tr_{p^n} \psi_{p^n}$ is a generator of additive integral operations
to $\CH^{p^n}\ot\Zp$. 

As $c_1^{p^n}$ sends $z_1\cdots z_{1+i(p^n-1)}$ to $p^n$-th power of a series divisible by 
$z_1\cdots z_{1+i(p^n-1)}$ for any $i\ge 1$,
the coefficient of $z_1\cdots z_{1+i(p^n-1)}$ in $c_1^{p^n}(z_1\cdots z_{1+i(p^n-1)})$ is zero for any $j$.
Thus, the coefficient of $z_1\cdots z_{1+i(p^n-1)}$ in $c_{p^n}(z_1\cdots z_{1+i(p^n-1)})$
is equal to the corresponding coefficient of the operation $\psi_{p^n}$
divided by $p$.

We are going to show that for the operation $\psi_{p^n}$
we have $\alpha_1\in p\Zp^\times$ and $\beta_1\in p\Zp$ in the notation 
of Lemma \ref{lm:add_op_induct_coef}.
It follows by induction on $i$ from relations there
 that $\beta_i \in p\Zp$ for any $i\ge 1$,
and therefore also by induction on $i$ that $\alpha_i\in p\Zp^\times$ for all $i\ge 1$.
Thus, we would have $\frac{\alpha_i}{p}\in\Zp^\times$ as needed.

To calculate the coefficient $\alpha_1$ we note that
it coincides with the corresponding coefficient of the truncation $tr_{p^n} \psi_{p^n}$
which is a generator of additive operations to $\CH^{p^n}\ot\Zp$.
By a straight-forward computation\footnote{or, for a more general discussion yielding this, see Appendix \ref{app:image_of_Chern_Chow}}
 one checks that a generator can be chosen to be $p\cdot ch_{p^n}$
where $ch\colon \Kn\rarr \CH^*_\QQ$ is the Chern character.
Since we have $p\cdot ch_{p^n}(z_1\cdots z_j)=pz_1\cdots z_j+\ldots$, 
it follows that $\alpha_1$ is proportional to $p$.

Since $pc_{p^n}=c_1^{p^n}-\psi_{p^n}$, we know that $\psi_{p^n} \equiv c_1^{p^n} \mod p$,
and therefore  all coefficients of those monomials of the series $\psi_{p^n}(z_1\cdots z_{1+i(p^n-1)})$
 which are not $p^n$-powers have to be zero modulo $p$.
 In particular,  $\beta_1$ is a coefficient of a monomial $z_1^{p^n}\cdots z_{p^n}$,
and therefore we have $\beta_1\in p\Zp$.  
\end{proof}

For the operation $c_{p^n}$ we have seen that
the $p$-valuation of the constant term of the series $\frac{c_{p^n}(z_1\cdots z_j)}{z_1\cdots z_j}$
does not depend on the choice of the operation itself.
However, this is not true for higher Chern classes.

In order to estimate the graded groups of the gamma filtration in the best way
one has to choose those operations in which the corresponding coefficients
have minimal $p$-valuation. 
Perhaps, this explains the importance of the following clumsy proposition
which is used in~\cite{SechSem}.

\begin{Prop}\label{prop:constant_chern_classes2}
Let $j\ge 0$, and let $p$ be the prime number corresponding to the $n$-th Morava K-theory,
i.e.\ $\Kn^*(\Spec k)=\Zp$.

There exist operations $\chi, \psi\colon \Kn^1\rarr \gamma^{2p^n-1}\Kn^1$
which satisfy the following. 

Denote by $h_j, f_j \in\Zp$, the coefficients
of the monomial $z_1\cdots z_{1+j(p^n-1)}$ in the series
$\chi(z_1\cdots z_{1+j(p^n-1)})$, 
 $\psi(z_1\cdots z_{1+j(p^n-1)}) \in \Kn^1((\mathbb{P}^\infty)^{\times 1+j(p^n-1)})$, respectively.
Then 

\begin{enumerate}[1)]
\item we have $h_j = f_j = 0$ for $j=0,1$.
\end{enumerate}
Let $j\ge 2$.
\begin{enumerate}[1)]
\setcounter{enumi}{1}
\item\label{item:coef_2pn-1}
\begin{itemize}
\item if $p\neq 2$ 
 we have $h_j \in p^{t_j}\Zp^\times$
where $t_j=\nu_p(j-1)+1$; 
\item if $p=2$ we have $t_j=\nu_2(j-1)+2$ if $j$ is odd
and $t_j=1$ if $j$ is even.
\end{itemize}
Here $\nu_p$ is the $p$-valuation on integers.
\item\label{item:coef_adams} We have $f_j \in p^{p^n}\Zp^{\times}$.
\end{enumerate}
\end{Prop}
\begin{proof}

Note that by Prop.~\ref{prop:gamma_unique} every operation $\phi\colon \Kn^*\rarr \tau^j \Kn^*$
takes values in $\gamma^j \Kn^*$. Thus, by Prop.~\ref{prop:operations_from_tilde}
it suffices to construct operations $\chi, \psi$
such that on products of projective spaces they take values in $\tau^{2p^n-1}$
and corresponding numbers $h_j, f_j$ satisfy the conditions above.
In particular, $f_j=h_j=0$ for $j=0,1$ will be automatically satisfied.

Let $k\in \NN$, let $\Psi_k\colon \Kn^*\rarr \Kn^*$ be the $k$-th Adams operation (see \cite[Th. 6.16]{Vish1}),
and consider an operation $\chi_k= (\Psi_k - k^{p^n}\cdot \id)\circ c_{p^n}$.
Both operations $\chi, \psi$ will be constructed as $\chi_k$ for specific $k$.

\ref{item:coef_2pn-1}) First, we claim that $\chi_k$ takes values in $\tau^{2p^n-1}$.
Using the Cartan formula one obtains that
$\pd \chi_k(u,v)=(\Psi_k - k^{p^n}\cdot \id)(-\frac{1}{p}\sum_{j=1}^{p^n-1}c_1(u)^ic_1(v)^{p^n-i})$.
However, $\Psi_k - k^{p^n}\cdot \id$ sends $z_1\cdots z_{1+j(p^n-1)}$ to a series in $\tau^{2p^n-1}$
for $j\ge 1$ (the claim is non-trivial only for $j=1$), and one deduces that this operation
sends a product of $p^n$ elements in $\Kn^1$ to an element of $\tau^{2p^n-1}\Kn^1$.
Thus, it is enough to check that the (non-additive) operation $\chi_k$ sends 
additive generators $z_1\cdots z_{1+j(p^n-1)}$ of 
$\Kn^1((\mathbb{P}^\infty)^{\times {1+j(p^n-1)}})$ to elements of $\tau^{2p^n-1}\Kn^1$.
This claim is non-trivial only for $j=0,1$ and can be checked by direct computations
using the fact that $c_{p^n}$ takes values in $\tau^{p^n}\Kn^1$.

Second, we claim that for specific $k$ its truncation 
 is a generator of primitive operations to $\CH^{2p^n-1}\ot\Zp$,
i.e.\ $tr_{2p^n-1} \chi_k = a\cdot c^{\CH}_{2p^n-1}+b(c_1^{\CH})^{p^n}c^{\CH}_{p^n}+d(c^{\CH}_1)^{2p^n-1}$
where $a\in \Zp^\times$, $b,d\in\Zp$. In order to prove this 
it is enough to show that $\chi_k$ sends $z_1\cdots z_{2p^n-1}$
to $ap\cdot z_1\cdots z_{2p^n-1}+$higher degree terms where $a\in\Zp^\times$,
because $c^{\CH}_{2p^n-1}$ is the only operation to $\CH^{2p^n-1}\ot\Zp$ which acts non-trivially on $z_1\cdots z_{2p^n-1}$.
Moreover, it acts proportionally to a generator of additive operations $\phi^{\CH}_{2p^n-1}$
 which can be chosen to be $p\cdot ch_{2p^n-1}$ (see Appendix~\ref{app:image_of_Chern_Chow}).
Since $c_{p^n}(z_1\cdots z_{2p^n-1})=e_2\cdot z_1\cdots z_{2p^n-1}$ 
and $e_2\in \Zp^\times$ by Proposition~\ref{prop:constant_cpn}, 
we calculate $\chi_k(z_1\cdots z_{2p^n-1})=e_2 k^{p^n}(k^{p^n-1}-1) \cdot z_1\cdots z_{2p^n-1}+$
higher degree terms. Thus, it is left to calculate the $p$-adic valuation of $k^{p^n}(k^{p^n-1}-1)$.

Let $k$ be coprime to $p$, then $\nu_p(k^{p^n})=0$, and $[k]\in (\ZZ/p^2)^\times$.
If $p\neq 2$, then \mbox{$(\ZZ/p^2)^\times \cong \ZZ/p\times \ZZ/(p-1)$,}
and since $p\nmid p^n-1$ we see that $p^2 \nmid k^{p^n-1}-1$
if $k$ is a generator of the multiplicative group $(\ZZ/p^2)^\times$.
If $p=2$, then $(\ZZ/4)^\times \cong \ZZ/2$,
and since $2^n-1$ is odd, we similarly obtain the claim for $k \equiv 3 \mod 4$.
Thus, if we choose $\chi$ to be $\chi_k$ for these $k$, its truncation is a primitive operation
and at least the $p$-adic valuation of coefficient $z_1\cdots z_{1+2(p^n-1)}$ of $\chi_k(z_1\cdots z_{1+2(p^n-1)})$ 
can not be reduced.

Third, we can now calculate constants $h_j$ of the operation $\chi_k$.
The operation $\chi_k$ sends $z_1\cdots z_{1+j(p^n-1)}$
to $e_j k^{p^n}(k^{(j-1)(p^n-1)}-1) z_1\cdots z_{1+j(p^n-1)}$+higher degree terms.
If $p\neq 2$, then $(\ZZ/p^l)^\times \cong \ZZ/p^{l-1}\times \ZZ/(p-1)$.
Thus, if $k$ is a generator of multiplicative groups $(\ZZ/p^l)^{\times}$ for all $l$ (e.g.\ $k=1+p$),
we need to look at the order of $j-1$ in $\ZZ/p^{l-1}$. 
If $l-1=\nu_p(j-1)+1$, then $p^l \nmid e_j k^{p^n}(k^{(j-1)(p^n-1)}-1)$
and $h_j \in p^{t_j}\Zp^\times$ where $t_j=\nu_p(j-1)+1$.

If $p=2$, then $(\ZZ/2^l)^\times \cong \ZZ/2^{l-2}\times \ZZ/2$ for $l\ge 2$.
Similarly, we need to look at the order of $j-1$ in $\ZZ/2^{l-2}$. 
However, if $j-1$ is odd, then $4\nmid k^{(j-1)(2^n-1)}-1$ for $k\equiv 3 \mod 4$.
If $j-1$ is even and $l-2=\nu_2(j-1)+1$, then $2^l \nmid e_j k^{p^n}(k^{(j-1)(p^n-1)}-1)$
and $h_j \in 2^{t_j}\Z{2}^\times$ where $t_j=\nu_2(j-1)+2$. This finishes the proof of \ref{item:coef_2pn-1}).

However, in the construction above we can also choose $k$ to be $p$, 
and then the $p$-adic valuation of $e_j k^{p^n}(k^{(j-1)(p^n-1)}-1)$ is $p^n$ for every $p$.
Thus, if we choose $\psi\colon =\chi_p$, this shows \ref{item:coef_adams}).
\end{proof}

\subsection{Application: estimates on the torsion in the Chow groups of quadrics}\label{section:application_quadrics}

\phantom{a}

The gamma filtration on $\Kn^*$ that we constructed in this paper 
provides a new tool for the computation of Chow groups in codimensions up to $p^n$,
due to Prop.~\ref{prop:morava_gamma_properties},~vi).
However, to achieve this for a smooth variety $X$,
 one needs to be able to compute $\Kn^*(X)$ (in the cases when $\CH^*(X)$ is not known).

In the rest of this section we consider only $\Kn$ for the prime $2$.
 
In~\cite{SechSem} Semenov jointly with the author of this paper show that there exists an abundance of quadrics $Q$,
for which $\Kn^*(Q)$ is as simple as it can be: $\Kn^*(Q)$ is isomorphic to $\Kn^*(Q\times_k \overline{k})$.
The latter group is a free $\Zp$-module with the well-known generators coming from the cellular structure of the split quadric
(see e.g.\ [loc.cit., Prop.~8.8]).
By applying the constructions and computations of this section the following result is then achieved.

\begin{Th}[{\cite[Th.~8.14]{SechSem}}]
Let $Q$ be a smooth quadric over a field $k$  such that 
the corresponding quadratic form $q$
 lies either in the ideal $I^{n+2}(k)$
or in the set $\langle c\rangle+I^{n+2}$ inside the Witt ring of $k$ for some $c\in F^\times$.

Let $D$ be the dimension of $Q$, $d:=[D/2]$,
and let $j\in[0,2^n-2]$ be the unique integer such that $d\equiv 1+j\mod 2^n-1$.

Then $\CH^{0\le * \le 2^n-1}(Q)=\ZZ$ and 
\begin{enumerate}
\item if $j\neq 0$, then $\CH^{2^n}(Q)=\ZZ$.

\item if $j=0$,  and the dimension of the quadric is odd,
 then the torsion in $\CH^{2^n}(Q)$ is at most $\ZZ/2$;
 
\item if $j=0$ and the dimension of the quadric is even,
$d=1+r(2^n-1)$,
then the torsion in $\CH^{2^n}(Q)$ is at most $\ZZ/2^s$
where $s=1$ if $r$ is even, $s=\min (\nu_2(r-1)+2, 2^n)$ otherwise.
\end{enumerate}
\end{Th}

So far, no other methods have reproduced the estimates obtained in this theorem. 
In other words, the gamma filtration on $\Kn^*$ remains the only tool that allows one to achieve this result.
We refer the reader to \cite[Sec.~8]{SechSem} for the context and history of this question,
 and for the proof of the above theorem.

\appendix
\section{}\label{app:disclaimer}

\subsection*{Disclaimer on the free theories with relation to the Landweber-exactness}

\phantom{a}

This expository appendix clarifies the relation of Definition~\ref{def:free_theory}
to the definitions of generalized cohomology theories in topology -- in particular,
the lack of the Landweber-exactness condition on the formal group law.
We recall what the latter means and survey the results in motivic homotopy theory
about the construction and the properties of oriented motivic spectra like Morava K-theory.
In particular, we review Levine's geometric Landweber exactness, which explains why the algebraic Morava 
K-theory studied in this paper is a meaningful analogue of the topological Morava K-theory.

We provide ample references for the motivic side of the picture
and refer the reader to \cite{Adams, Rav, Quillen} and references therein for the classical topological results.

\subsubsection*{Landweber-exactness.}

By the result of Quillen,
there exists a canonical isomorphism $\LL \xrarr{\cong} \MU^*(pt)$ of graded rings.
Thus, given an $\LL$-module $M$ one can ask whether $\MU^*(X)\otimes_{\LL} M$ 
is a generalized cohomology theory. 
Most of the Eilenberg--Steenrod axioms are easy to verify,
except for the ``exactness'' one, which is a long exact sequence of the cohomology groups.
For this axiom to hold, the long exact sequence for $\MU^*$ should stay exact after tensoring it with $M$ over $\LL$,
and this certainly holds if $M$ is a flat $\LL$-module.

However, $\MU^*(X)$ has an additional structure of $(\MU^*(pt), \MU^*(MU))$-comodule,
or, in more modern terms, of a quasi-coherent sheaf on the stack of formal groups.
It was shown by Landweber that $(\MU^*(pt), \MU^*(MU))$-comodules 
viewed as $MU^*(pt)$-modules (i.e.\ the same as $\LL$-modules)
must satisfy some severe restrictions on their structure.
In particular, if $X$ is a finite CW-complex, then the $\LL$-module $\MU^*(X)$ has a filtration
with quotients $\LL/I(p,n)$, where $I(p,n)$ are certain prime ideals in $\LL$,
known as Landweber ideals.
For $\MU^*(X)\otimes_{\LL} M$ to be a cohomology theory it thus suffices to check that $\Tor^{\LL}_{>0}(M,\LL/I(p,n))=0$,
the condition being generally known as {\sl Landweber-exactness}.

Thus, if we have a graded formal group law $F$ over a graded ring $R$,
and $R$ is Landweber-exact as $\LL$-module,
then we get a generalized cohomology theory $\MU^*(-)\otimes_\LL R$.
For example, the multiplicative formal group law over $\ZZ[\beta, \beta^{-1}]$ is Landweber-exact,
and the corresponding generalized cohomology theory is the complex K-theory (this result is known as
the Conner--Floyd isomorphism). Other Landweber-exact formal group laws include
the ones of the Brown--Peterson cohomology $\BPtop$ and the Johnson--Wilson cohomology $\Entop$,
as well as elliptic cohomology.
However, the formal group laws of Morava K-theories are {\it not} Landweber-exact.
In particular, the topological Morava K-theory $\Kntop(X)$ of a CW-complex $X$
is {\it not} isomorphic to $\MU^*(X)\ot_{\LL} \Kntop(pt)$.

The definition~\ref{def:free_theory}, however, does not require any Landweber-exactness of the formal group law,
and, in particular, the algebraic Morava K-theory $\Kn(X)$ of a smooth variety $X$ that we define in Section~\ref{sec_op_mor}
is given by the formula $\Omega^*(X)\otimes_{\LL} \Kn(pt)$.
As we explain below, this is not a ``wrong'' definition of the algebraic Morava K-theory,
but only a ``geometric'' part of the motivic analogue of $\Kntop$.

\subsubsection*{Constructing complex-oriented spectra without Landweber-exactness}

Topological Morava K-theory can be characterized
as a complex-oriented cohomology theory with coefficients $\Fp[v_n,v_n^{-1}]$
and a formal group law of height $n$.
One constructs the spectrum $\Kntop$ representing this cohomology theory
as a certain ``quotient'' of $\MU$ in the sense that one kills and inverts
some of the coefficients of $\MU(pt)$.

The ring map  $\LL \rarr \Fp[v_n,v_n^{-1}]$ that corresponds to the $p$-typical formal group law of $\Kntop$
kills $p$ and all the generators of $\LL$ except for $x_{p^n-1}$ (of topological degree $2(1-p^n)$)
and sends the latter to $v_n$.
To construct $\Kntop$ one performs the following on the level of spectra:
first, killing $p$ and the generators $x_i$ of $\MU(pt)$,  $i\neq p^n-1$, in $\MU$,
 and then inverting the generator $x_{p^n-1}$. If one wants to obtain a $p$-local version of $\Kntop$,
 one can do a similar procedure with inverting all primes except for $p$ instead of killing~$p$. 

Note that the procedures of killing coefficients in the base ring
do not preserve multiplicative structures on the underlying spectra, and the question what kind of multiplication $\Kntop$ admits
is a more subtle one (although well-studied, see e.g.\ \cite{Rob}). 
Below we recall that the construction of the motivic spectrum of the algebraic Morava K-theory $\Knspec$
follows the same procedure.

\subsubsection*{Motivic spectra and motivic Landweber exactness} 
Generalized cohomology theories in topology define spectra --
objects of the stable homotopy category $\SH^{top}$.
The algebro-geometric analogue of it 
is the stable motivic homotopy category $\SH(k)$ and it was constructed by
Morel and Voevodsky~\cite{VoeA1},
its objects are called motivic spectra. 
Given an inclusion of the field $k$ into complex numbers,
there exists a topological realization functor $\SH(k)\rarr \SH^{top}$,
which allows to compare motivic spectra to the topological ones. 
For the exposition purposes, we fix this functor in what follows.

The crucial difference between motivic spectra and the topological situation is that 
there are two types of ``circles'' in $\SH(k)$: the ``simplicial'' one $S^1$
and the ``geometric'' one $\mathbb{G}_m$. 
Both of these circles are sent to the topological circle $S^1_{top}$
by the topological realization, and so the effects coming from $S^1$ and $\mathbb{G}_m$ being different
cannot be seen topologically.

In particular, the existence of two circles leads to the definition of bi-graded motivic spheres $S^{p,q}$:
if $p,q\ge 0$, then $S^{p,q} = S^{p-q}\wedge \mathbb{G}_m^{q}$.
This makes the cohomology theories represented by motivic specta
bi-graded: for a smooth variety $X$ over $k$ and a motivic spectrum $E$ one defines
$$ E^{p,q}(X) = [S^{-p,-q}\wedge \Sigma^{\infty} X_+, E].$$

Let $E^{top}$ be the topological realization of the motivic spectrum $E$,
and denote by $X(\CC)$ the CW-complex underlying the complex manifold of $\CC$-points in $X$.
We get functorial maps on cohomology groups:
\begin{equation}\label{eq:top_realization}
E^{p,q}(X) \rarr (E^{top})^p(X(\CC)).
\end{equation}

There exist
 motivic Eilenberg--Maclane spectrum $\H \ZZ$ that represents motivic cohomology $\H^{p,q}$,
 algebraic K-theory spectrum $\KGL$ representing algebraic K-theory
 and motivic algebraic cobordism spectrum $\MGL$ (see \cite{VoeA1}),
 and the latter is the universal oriented motivic spectrum \cite{PanPimRoe}.
The topological realization sends these motivic spectra
to their topological analogues: Eilenberg--Maclane spectrum $\H^{top} \ZZ$, 
the complex K-theory $\mathrm{KU}$ and the complex cobordism $\mathrm{MU}$, respectively.
Thus, we get canonical morphisms from algebraic cohomology theories
to their topological analogues, as in (\ref{eq:top_realization}).

Similarly to the topological situation,
one can construct motivic spectra corresponding to the Landweber-exact formal group laws:
if $R$ is a Landweber-exact graded $\LL$-algebra,
then the functor on $\SH(k)^c$ $$\MGL^{*,*}(-)\otimes_{\LL} R$$ 
is represented by a motivic spectrum \cite{MotivicLandweber}.

\subsubsection*{Quotients of $\MGL$}

On the other hand, one can construct the motivic spectrum $\BPspec$
using motivic Landweber-exactness, 
and then one can repeat the topological procedure to construct the motivic spectrum $\Knspec$ -- first, kill $p$ and $v_m$, $m\neq n$,
and then invert $v_n$.
We are thus interested in what happens to the values of the bi-graded cohomology theory,
when we kill some of the elements in the coefficient ring on the level of the motivic spectrum.
It turns out that on some {\sl part} of such ``quotients'' of $\MGL$ this can be computed purely algebraically.

Recall that the Lazard ring $\LL$ is (non-canonically) isomorphic as a graded ring to $\ZZ[x_1, \ldots]$.
Given a collection $S$ of the elements $x_i$ and a collection $S_0$ of homogeneous elements
of the ring $\ZZ[x_i, x_i\notin S]$, Levine--Tripathi~\cite[Sections~1 and 6]{LevTri} construct the motivic spectrum
$$ \MGL/(S)[S_0^{-1}] = (\MGL/(x_i, x_i\in S))[y^{-1}, y\in S_0].$$

In particular, one can take $S$ to consist of $p$ and all $x_i$ with $i\neq p^n-1$,
 which yields Morava K-theory with mod-$p$ coefficients. 
 Similarly, one defines its $p$-local version.

\begin{Def}\label{def:Kn-motivic-spectrum}
We define the motivic spectrum $\Knspec$ of the $p$-local $n$-th algebraic Morava K-theory as
$$ \Knspec := \MGL/(x_i, i\neq p^n-1)[q^{-1}, q\neq p, q\text{ is a prime number}; x_{p^n-1}^{-1}].$$
\end{Def}

Note that the topological realization of $\Knspec/p$ is $\Kntop$,
which makes $\Knspec$ an algebro-geometric analogue of  the topological Morava K-theory.

\subsubsection*{Levine's geometric Landweber exactness}

The following fundamental result explains our Definitions~\ref{def:free_theory} and~\ref{mor}.

\begin{Th}[Levine--Tripathi {\cite[Prop~6.1, cf.~Cor.~6.3]{LevTri}}]\label{th:app_Knspec_comparison}
Let $k$ be a field of characteristic $0$. Then the following isomorphism of functors
on smooth quasi-projective varieties over $k$ holds:
\begin{equation}\label{eq:morava_geometric_landweber}
\Knspec^{2*,*}\cong \MGL^{2*,*}\otimes_\LL \Zp[v_n,v_n^{-1}].
\end{equation}
\end{Th}

To compare $\Knspec^{2*,*}$ with $\Kn:=\Omega^* \otimes_\LL \Zp[v_n,v_n^{-1}]$ that we study in this paper
it is left to recall the relation between $\MGL$ and $\Omega$.
For a smooth quasi-projective variety $X$ over $k$ Levine~\cite{LevComparison}
proved that $\MGL^{2*,*}(X)$ and $\Omega^*(X)$ are canonically isomorphic as graded rings.
Thus, the ``homotopically'' defined theory $\MGL$
has a more geometric description when restricted to the $(2*,*)$-graded part,
 sometimes called the ``geometric part'' (e.g.\ \cite[Def.~5.10]{LevTri}).
We can now reformulate the Theorem above as follows: $\Kn$ is the geometric part
of the $\Knspec$.

Let us recall some history behind Theorem~\ref{th:app_Knspec_comparison}.
Levine--Morel proved that $\CH^*$ and $\KK[\beta, \beta^{-1}]$ are free theories,
i.e.\ they are isomorphic to  $\Omega^*\otimes_\LL \ZZ$ and $\Omega^*\otimes_\LL \ZZ[\beta, \beta^{-1}]$, respectively,
where the morphisms $\LL\rarr \ZZ$ and $\LL\rarr \ZZ[\beta, \beta^{-1}]$ correspond to 
the additive and the (graded) multiplicative formal group laws.
In particular, combining it with the previous result of Levine,
 one sees that $\CH^*\cong \MGL^{2*,*}\otimes_\LL \ZZ$, 
even though the additive formal group law is not Landweber-exact.

It was then observed by Dai and Levine \cite{DaiLevine}
that a similar property also holds for the connective algebraic K-theory $\CKK$,
which has coefficients $\ZZ[\beta]$ and the multiplicative formal group law.
Although $\ZZ[\beta]$ is not Landweber-exact as $\LL$-algebra, 
there is an isomorphism \cite[Th.~6.3]{DaiLevine}:
$$ \CKK \cong \MGL^{2*,*}\otimes_\LL \ZZ[\beta].$$

Levine generalized this comparison to arbitrary effective covers 
(which is a motivic analogue of the notion of a connective cover)
of motivic Landweber-exact spectra in \cite{LevLandweber}.
Moreover, he introduced the notion of {\sl geometrically Landweber exact} motivic ring spectrum~[loc.cit., Def.~3.7].
This was later generalized to motivic spectra which are merely modules over some ring spectrum (e.g.\ $\MGL$)
by Levine--Tripathi~\cite[Def.~5.11]{LevTri}.
Without reproducing this notion let us mention that being geometrically Landweber exact $\MGL$-module
implies a comparison result for the geometric part of it, as in Theorem~\ref{th:app_Knspec_comparison},
see~\cite[Th.~6.2]{LevLandweber} and \cite[Th.~5.12]{LevTri}.
The main result of Levine--Tripathi is that $\MGL$-``quotients'' $\MGL/(S)[S_0^{-1}]$ are
geometrically Landweber-exact \cite[Prop.~6.1]{LevTri}.

Finally,
let us mention a few key ingredients that go into the proofs of the above mentioned motivic results.
First, as we already mentioned, to compare the motivically respresentable oriented theories
with the theories obtained from $\Omega^*$ 
one needs the isomorphism $\Omega^*(X)\xrarr{\cong} \MGL^{2*,*}(X)$
proved by Levine~\cite{LevComparison}.

Second, one tries to show that the canonical morphism of theories
\begin{equation}\label{eq:kn-comparison}
\Omega^*\otimes_\LL \Zp[v_n, v_n^{-1}] \xrarr{\Theta_{\Kn}} \Knspec^{2*,*}
\end{equation}
coming from the orientation of $\Knspec$ is an isomorphism.
The key idea is to use the localization axiom (\ref{eq:loc}), 
which is satisfied by $\Omega^*$.
For $\Knspec$ the localization sequence is a long exact sequence
coming from the cofiber sequence in the unstable motivic category $\mathcal{H}_\bullet(k)$ \cite[Th.~2.23]{MorVoe}:

$$ U_+ \rarr X_+ \rarr \Sigma^\infty Th(N_{Z/X}),$$

and the Thom isomorphisms allow to identify $\Knspec^{*,*'}$
of the Thom space $Th_Z(N_{Z/X})$ with $\Knspec^{*+2c,*'+c}(Z)$, $c$ being codimension of $Z$ in $X$.
In the range relevant for the comparison the localization sequence now looks as follows:
$$\ldots \rarr \Knspec^{2(n+c), n+c}(Z) \rarr \Knspec^{2n,n}(X) \rarr \Knspec^{2n,n}(U) \rarr \Knspec^{2n+1,n}(Z)\rarr \ldots $$

If $\Theta_\Kn$ of (\ref{eq:kn-comparison}) is an isomorphism,
then the morphism on the right has to be zero.
One proves it before showing that $\Theta_\Kn$ is an isomorphism,
and, more precisely, one shows that  
$\Knspec^{2n+1,n}(W)=0$ for all $n$ and any smooth $W$.
This result is based on the computation of the {\sl slices} of $\Knspec$ -- the associated subquotients
defined via the slice filtration -- as certain motivic Eilenberg--Maclane spaces.
The slices of ``quotients'' of $\MGL$ 
are computed by Levine and Tripathi \cite[Section~4]{LevTri} based on the ideas of Spitzweck~\cite{Spitzweck}.
Since the motivic cohomology $\H^{p,q}(X)$ vanishes for $p>2q$ \cite[Th.~19.3]{Mazza},
one can then use the slice spectral sequence to conclude the similar property for $\Knspec$.
We should also note that the computations of slices
require the validity of the Hopkins--Morel conjecture,
which was settled by Hoyois in \cite{Hoy}.

\section{}\label{app}
\subsection{Non-existence of some operations from $\Kn^*$}\label{app:non-exist-op}
In the paper we have constructed a subset of the set of oriented cohomology theories
for which Chern classes from $\Kn^*$ exist. Perhaps, vaguely speaking these $p^n$-typical theories
can be called Morava-orientable. 
It is reasonable to ask whether this subset can be expanded.
 We do not  answer this question completely here as we have no definition of Morava-orientability,
however the following results suggest 
that theories which are not $p^n$-typical (up to a change of orientation)
do not admit sufficiently good Chern classes from $\Kn^*$.

More precisely, Prop. \ref{prop:non_ex_m<n}
shows that free theories whose height is less than $n$ can not have a good theory
of Chern classes since they do not admit even additive operations from $\Kn^*$.
On the other hand Prop. \ref{prop:non_ex_m_nmid_n} shows 
that there could be no lifting (with respect to the truncation, see Prop. \ref{prop:op_mod_tau_vs_trunc})
 of Chern classes from $\Kn^*$ to $\CH^*\ot\Zp$
to operations with the target theory $K(m)^*$ when $m\nmid n$.

\begin{Prop}\label{prop:non_ex_m<n}
Let $A^*$ be a free theory s.t. $A$ is an $\F{p}$-algebra, 
and $p\cdot_A x \equiv a_k x^{p^k} \mod x^{p^k+1}$
where $a_k \in A$ is not a zero-divisor. 

If $k<n$, then there exist no non-trivial additive operations
from $\Kntilde^*/p$ to $A^*$.
In particular, there exist no additive operations from $\Kntilde^*$ to $BP\{k\}^*$ or $K(k)^*$ for $k\colon 1\le k <n$.
\end{Prop}
\begin{proof}
Let $\phi$ be a non-trivial additive operation from $\Kntilde^*$ to $A^*$.
Let us consider its action on products of projective spaces
which is non-trivial by Vishik's Theorem \ref{th:Vish_op}.
There exist $i>0$ s.t. $G_i:=\phi(z_1^{\Kn}\cdots z_i^{\Kn})\neq 0$
where $G_i$ is a symmetric series in $z_1^A, \ldots, z_i^A$ divisible by $\prod_{j=1}^i z_j^A$.
Let $d\ge 1$ be the minimal degree of $z_1$ in $G_i$.

The pull-back along the $m$-Veronese map $[m]$ on $\mathbb{P}^\infty$
acts on the first Chern class $z$ in an oriented theory $B^*$ by the formula
$z \mapsto m\cdot_B z$.
Thus, since $\phi$ has to commute with the pull-back along the map $[p^N]\times \id^{\times i-1}$
on $(\mathbb{P}^\infty)^{\times i}$
we have $\phi((p^N\cdot_{\Kn} z_1^{\Kn})z_2^{\Kn}\cdots z_i^{\Kn})$
is equal to $G_i|_{z_1^B = p^N\cdot_A z_1^B}$.
By the assumptions on the series $p\cdot_A z$ 
(note that $p^N \cdot_A z = p\cdot_A (p\cdot_A (\cdots (p\cdot_A z)))$)
we can see that the minimal degree of $z_1$ in the series 
$G_i|_{z_1^A = p^N\cdot_A z_1^A}$ equals to $dp^{kN}$.
On the other hand, the series $(p^N\cdot_{\Kn} z_1^{\Kn})z_2^{\Kn}\cdots z_i^{\Kn}$
has the minimal degree of $z_1$ equal to $p^{Nn}$.
By the continuity of operations
the minimal degree of $z_1$ in $\phi((p^N\cdot_{\Kn} z_1^{\Kn})z_2^{\Kn}\cdots z_i^{\Kn})$
is greater or equal to $p^{Nn}$ which is bigger than $dp^{kN}$ for sufficiently big $N$. Contradiction.

Every additive operation $\phi$ from $\Kntilde^*$ to $BP\{k\}^*$ or $K(k)^*$
factors through to an additive operation from $\Kntilde^*/p$ to $BP\{k\}^*/p$ or $K(k)^*/p$,
respectively. As follows from above this mod-$p$ operation has to be zero.
It follows from Vishik's theorem that one can canonically divide $\phi$ by $p$
to get a new additive operation. Again, $\frac{\phi}{p}$ has to be zero modulo $p$.
Continuing this we see that $\phi$ is zero modulo $p^N$ for every $N\ge 1$.
However, the action of $\phi$ on products of projective spaces
is defined by series with $BP\{k\}$- or $\Zp$-coefficients,
and therefore they are equal to zero if they are equal to zero modulo $p^N$ for every $N\ge 1$.
Thus, $\phi$ has to be a trivial operation.
\end{proof}

\begin{Rk}
One can show that between any two free theories there exist a non-trivial 
(in most cases, non-additive) operation.
For example, for all $n,m$ there exist an operation $c\colon \Kn^*\rarr K(m)^*$
defined as the composition of operations $c=c_1^{\KK\rarr K(m)}\circ \iota \circ c_1^{\Kn\rarr \CH}$
where $\iota$ denotes a non-additive map $\CH^1\cong \mathrm{Pic}\rarr \KK$.
\end{Rk}

\begin{Lm}\label{lm:comp_add_Kn_Km}
Assume that $m \nmid n$.
Let $\phi\colon \Kn^*\rarr K(m)^*$ be an additive operation.

Let $\psi\colon K(m)^*\rarr \CH^i\ot\Zp$ be any additive operation,
where $i\neq ip^m \mod p^n-1$.

Then the composition $\psi\circ \phi$ is 0 modulo $p$.
\end{Lm}
\begin{proof}
Without loss of generality we may assume that $\psi$ is a generator $\psi_i$ of additive operations
to some component $\CH^i\ot\Zp$. The composition $\psi\circ \phi$ has
to be supported on $\Kn^{i \mod p^n-1}=\Kn^i$ (\cite[Prop. 4.1.6]{Sech}), 
and we may also assume that $\phi$ is supported on $\Kn^i$.

By \cite[Cor. 4.3.5]{Sech} there is a relation $\psi^{p^m} \equiv \psi_{ip^m} \mod p$,
where $\psi_{ip^m}\colon K(m)^*\rarr \CH^{ip^m}\ot\Zp$ is a generator of additive operations.
Therefore we have $(\psi\circ \phi)^{p^m} \equiv \psi_{ip^m}\circ \phi \mod p$.
However, $\psi_{ip^m}\circ \phi$ is an additive operation from $\Kn^i$ to $\CH^{ip^m}\ot\Zp$.
If $ip^m \neq i \mod (p^n-1)$, this operation is zero by \cite[Prop. 4.1.6]{Sech},
 and therefore the composition $\psi_{ip^m}\circ \phi$ is zero modulo $p$.

If the operation $\psi\circ \phi$ is non-trivial modulo $p$,
then it acts non-trivially on products of projective spaces,
and since the theory $\CH^*/p$ has no nilpotents in the coefficient ring 
the $p^m$-th power of this operation also acts non-trivially on products of projective spaces.
Contradiction.
\end{proof}

\begin{Prop}\label{prop:non_ex_m_nmid_n}
Let $m\nmid n$. Then there exist $i$ s.t. $i\colon 1\le i\le p^m-1$ and $i p^m \neq i \mod (p^n-1)$.

For such $i$ the truncation map $tr_i\colon [\Kn^*, \tau^i K(m)^*]^{add} \rarr [\Kn^*, \CH^i\ot\Zp]^{add}$
is not surjective.
\end{Prop}
\begin{proof}
Let $\phi_i\colon K(m)^i\rarr \CH^i\ot\Zp$ be a generator of additive operations.
Then for $i$ in the given range
the composition of an operation from $\Kn^*$ to $\tau^iK(m)^*$ with the operation $\phi_i$
is the same as the truncation $tr_i$. One apples Lemma \ref{lm:comp_add_Kn_Km} and the claim 
follows.
\end{proof}

\subsection{Existence of $n$-th Morava K-theories which are not multiplicatively isomorphic}\label{app:mor_not_mult}

\phantom{a}

Recall that an $n$-th Morava K-theory $\Kn^*$ is defined as $\Omega^*\ot_\LL \Zp$
where the morphism of rings $\theta\colon \Omega^*(k)\cong\LL\rarr \Zp$
corresponds to a $p^n$-typical formal group law $F_{\Kn}$ over $\Zp$
s.t.\ $F_{\Kn} \mod p$ has height $n$.
However, we remarked after introducing Definition \ref{mor}
that if a cohomology theory is represented in the stable motivic category,
then its `geometric' part is a graded theory. For the case of Morava K-theories
this would mean investigating a theory $\Omega^*\ot_\LL \Zp[v_n,v_n^{-1}]$
where $\LL \rarr \Zp[v_n,v_n^{-1}]$ is a {\sl graded} lift of the ring map $\theta$ above.
Let us call this theory a graded $n$-th Morava K-theory throughout this section
 and denote it by $G\Kn^*$.
Note that $G\Kn^*\ot_{\Zp[v_n, v_n^{-1}]} \Zp=\Kn^*$ for some choice
of the ring map $\Zp[v_n,v_n^{-1}]\rarr\Zp$.

There is an isomorphism of presheaves of abelian groups
between $G\Kn^*$ and $\oplus_{i\in \ZZ} v_n^i\cdot \Kn^*$.
This allows to translate the classification of additive operations 
from or to $\Kn^*$ into the classification of additive operations from, resp. to, $G\Kn^*$.
In particular, it follows
from Th. \ref{th:morava_unique} that
any two graded $n$-th Morava K-theories are additively graded-isomorphic.

Comparing different graded and non-graded versions of Morava K-theories
as presheaves of rings is more subtle.
However, this question can be translated into a question about 
isomorphisms of formal group laws.

\begin{Prop}
Let $G\Kn^*$ be a graded $n$-th Morava K-theory,
let $\phi_1, \phi_2\colon \Zp[v_n,v_n^{-1}]\rarr \Zp$ be two ring maps
sending $v_n$ to $a_1, a_2 \in\Zp^\times$, respectively.
If $a_1\neq a_2 \mod p$, then $n$-th Morava K-theories 
$(\Kn)^*_i:=G\Kn^*\ot_{\phi_i} \Zp$, $i=1,2$ are not multiplicatively isomorphic (see Def.~\ref{def:mult_iso}).
\end{Prop}
\begin{proof}
 By \cite[Th. 6.9]{Vish1} the set of invertible multiplicative operations
 $\phi\colon  (\Kn)^*_1\rarr (\Kn)^*_2$
 is in bijective correspondence with the set of series
 $\gamma(x) \in \Zp^{\times}x+\Zp[[x]]x^2$ s.t.
 \mbox{$F_1(\gamma(x),\gamma(y))=\gamma F_2(x,y)$.}
 Since we are working over a torsion-free ring,
 the latter equation is equivalent 
 to $\gamma:=\log_1^{-1}(\log_2(x))$,
 where $\log_i(x)$ is the logarithm of the formal group law $F_i$ of the theory $(\Kn)^*_i$.

Using Araki relations we can write $\log_i(x)=x+\frac{a_i}{p-p^{p^n}}x^{p^n}+\ldots$, $i=1,2$,
and a direct computation shows that $\log_1^{-1}(\log_2(x))$
is not integral under the conditions on $a_1, a_2$. 
\end{proof}

Similarly,for the graded theories we have the following.

\begin{Prop}\label{prop:mor_mot-mult}
There exist two $n$-th graded Morava K-theories which 
are not graded multiplicatively isomorphic.
\end{Prop}
\begin{proof}
Two $n$-th graded Morava K-theory $(G\Kn)^*_i= BP\{n\}^*\ot_{BP\{n\}} \Zp[v_n, v_n^{-1}]$, $i=1,2$,
are defined by morphism of graded rings $\psi_i\colon BP\{n\}\rarr \Zp[v_n, v_n^{-1}]$, $i=1,2$, respectively,
which define two formal group laws $F_1, F_2$.
Let us denote $\psi_i(v_n)=a_iv_n$, $\psi_i(v_{2n})=b_iv_n^{p^n+1}$ for some
numbers $a_i \in \Zp^\times$, $b_i\in \Zp$.

By \cite[Th. 6.9]{Vish1} a multiplicative (graded) isomorphism between these two theories
consists of a graded isomorphism 
$\phi\colon (G\Kn)^*_1(\Spec k)\cong\Zp[v_n,v_n^{-1}]\rarr (G\Kn)^*_2(\Spec k)[v_n,v_n^{-1}]$
which sends $v_n$ to $\alpha v_n$ for some $\alpha \in\Zp^{\times}$
and a homogeneous series $\gamma\in \Zp^\times x+\Zp[v_n,v_n^{-1}][[x]]x^2$ of degree 1
s.t. $\phi(F_1)(\gamma(x),\gamma(y))=\gamma F_2(x,y)$.
Without loss of generality we may assume that \mbox{$\gamma(x)\equiv x \mod x^2$,}
since we can twist the isomorphism by an invertible Adams operation otherwise.
Thus, $\gamma(x)\equiv x+cv_nx^{p^n}\mod x^{p^n+1}$.

Using the Araki generators of $BP\{n\}$ (see Prop. \ref{prop:universal_pn-typical_fgl})
we can write for $i=1,2$:
$$p \cdot_{F_i}  x \equiv \psi_i(v_n)x^{p^n}+ \psi_i(v_{2n}) x^{p^{2n}}
\equiv a_i v_n x^{p^n}+ b_i v_n^{p^n+1}x^{p^{2n}} \mod (p, x^{p^{2n+1}}).$$

On the other hand it follows from the equation on $\gamma$ that
\begin{equation}\label{eq:FGL_morph_p}
\phi\left(p \cdot_{F_1} \gamma(x)\right) = \gamma (p\cdot_{F_2} x),
\end{equation}
where on the left hand side $\phi$ is applied to the series $p\cdot_{F_1}$,
and then $\gamma(x)$ is plugged in into it.
Rewriting this equation with given series we obtain

$$a_1\alpha v_n\gamma(x)^{p^n}+b_1 \alpha^{p^n+1} v_n^{p^n+1} \gamma(x)^{p^{2n}}
\equiv \gamma(a_2 v_n x^{p^n}+b_2v_n^{p^n+1} x^{p^2n}) \mod (p, x^{p^{2n+1}}),$$

$$ a_1\alpha v_n x^{p^n}+(a_1 \alpha c^{p^n} + b_1 \alpha^{p^n+1}) v_n^{p^n+1} x^{p^{2n}}
\equiv
 a_2 v_n x^{p^n}+(b_2+ c a_2^{p^n}) v_n^{p^n+1} x^{p^{2n}} \mod (p, x^{p^{2n+1}}).$$

We get two equations from which we obtain $\alpha \equiv \frac{a_1}{a_2} \mod p$,
$b_1 \alpha^2 \equiv b_2 \mod p$. However, since we can choose $b_1, b_2$ as we want,
e.g.\ $b_1=0, b_2\neq 0$, these equations can not always be satisfied,
and $(\phi, \gamma)$ does not always exist.
\end{proof}

\subsection{Image of Chern classes from $\Kn^*$ in Chow groups}\label{app:image_of_Chern_Chow}

One of the main results on Chern classes obtained in this paper
is that operations $c_i\colon \Kn^*\rarr \CH^i\ot\Zp$, $i\colon 1\le i\le p^n$ are surjective.
This is not true for Chern class $c_i$, $i>p^n$,
however, the image of this operation is always a subgroup of the form $b_i\CH^i\ot\Zp$.
In this section we provide an inductive way to compute numbers $b_i$.

Denote by $d_i\in p^\ZZ$ the number s.t. $d_i \cdot ch_i$ acts integrally on products of projective spaces
and this action is not zero modulo $p$. In other words, $d_i \cdot ch_i$
can be lifted to a generator $\phi_i$ of additive operations from $\Kn^*$ to $\CH^i\ot\Zp$. 
It is clear that $d_i$ is uniquely defined by the above.

\begin{Prop}\label{prop:app_add_chernchar}
We have $d_1=1$. 
\begin{itemize}
\item If $i \neq p^{sn}$ for some $s\in \NN$, then $d_i = \max_{j=1}^{i-1} (d_j d_{i-j})$.
\item If $i = p^{sn}$, then $d_i =p\cdot \max_{j=1}^{i-1} (d_j d_{i-j})$.
\end{itemize}
\end{Prop}
\begin{proof}
The operation $ch_1$ acts integrally on projective spaces,
 it is defined by series $G_1(z)=z$, $G_j(z)=0$ for $j\ge 2$. 
Thus, $d_1=1$.

\begin{Lm}\label{lm:polyadd_op}
The $\Zp$-module of poly-additive poly-operations of arity $r$ from $\Kn^*$
to $\CH^i\ot\Zp$ is a free $\Zp$-module generated by external products $\phi_{j_1}\odot \cdots \odot \phi_{j_r}$
where $\sum_{s=1}^{r} j_s =i$, $j_s\colon 0\le j_s \le i$.
\end{Lm}
\begin{proof}
It follows from Vishik's classification of poly-operations theorem
that this module is free. Thus, its rank is equal to the dimension of $\QQ$-vector space
of additive polyoperations from $\Kn^*\ot\QQ$ to $\CH^i\ot\QQ$.
However, using a multiplicative isomorphism $ch\colon \Kn^*\ot\QQ \cong \oplus \CH^i\ot\QQ$
one reduces the problem of calculating this dimension to the dimension of the vector space
of poly-operations of arity $r$ from Chow groups to Chow groups with rational coefficients.
One can show that the ring of poly-operations in Chow groups is generated by multiplications of components,
and, thus, the dimension is equal to the number of external products $\phi_{j_1}\odot \cdots \odot \phi_{j_r}$.

The number of poly-operations $\phi_{j_1}\odot \cdots \odot \phi_{j_r}$ coincides with 
the rank of the $\Zp$-module, 
and thus it is enough to show that there are no $\F{p}$-relations between them.
But this is precisely the claim of \cite[Lemma 4.6.2]{Sech}.
\end{proof}
\textsl{Proof of Proposition~\ref{prop:app_add_chernchar}, continued.} 
Denote by $m$ the multiplication in $\Kn^*$ which is a poly-additive poly-operation of arity 2.
As follows from Lemma \ref{lm:polyadd_op} for each $i$ the poly-operation $\phi_i\circ m$
is of the form 
$$\phi_i\odot \phi_0 + \phi_0\odot \phi_i+\sum_{j=1}^{i-1} b^{(i)}_j \phi_j \odot \phi_{i-j}$$
 where $b^{(i)}_j\in \Zp$ and $\phi_0$ is the additive operation $\Kn^*\rarr \CH^0\ot\Zp$
 which sends 1 to 1 and is zero on $\Kntilde^*$.

On the other hand, since $ch$ is a multiplicative operation,
we have $ch_i\circ m = \sum_{j=0}^{i} ch_{j}\odot ch_{i-j}$.
Combining these two equations together we obtain that
 $\phi_i \circ m = \phi_i\odot \phi_0 + \phi_0 \odot \phi_i
 + \sum_j \frac{d_i}{d_jd_{i-j}} \phi_j\circ \phi_{i-j}$,
and therefore $b^{(i)}_j = \frac{d_i}{d_jd_{i-j}}$.
The number $b^{(i)}_j$ has to be integral since $\phi_i \circ m$ is an integral polyoperation,
and therefore $d_i \ge \max_{j=1}^{i-1} (d_j d_{i-j})$.

Assume that $b^{(i)}_j\in p\Zp$ for all $j\colon 1\le j\le i-1$,
i.e.\ $\phi_i\circ m\equiv \phi_i\odot \phi_0 + \phi_0\odot \phi_i \mod p$. 
In particular, this means that $(\phi_i\mod p)(z_1\cdots z_j)=0$ for all $j\ge 2$.
However, $\phi_i\mod p\neq 0$ as it is a generator of additive operations,
and therefore $(\phi_i\mod p)(z)$ can not be equal to zero.
Thus, $(\phi_i\mod p)(z) = a z^i$ for some $a\in \F{p}^\times$,
and we see that $\phi_i \equiv a(\phi_1)^i \mod p$. 

In order for $(\phi_1)^i$ to be additive modulo $p$ we need $i$ to be a power of $p$.
However, if $i\neq p^{sn}$ for some $s\ge 0$, then $\phi_i$ and $\phi_1$ are supported on different
graded components of $\Kn^*$ by \cite[Prop. 4.1.6]{Sech}
 (or one could argue that $(\phi_1)^i \mod p$ is not $p^n$-gradable). 
Thus, we see that if $i \neq p^{sn}$, then there exist $j$ s.t. 
 $\frac{d_i}{d_jd_{i-j}}\in \Zp^\times$. The claim follows.

If $i = p^{sn}$, then by \cite[Cor. 4.3.5]{Sech} we have the relation $\phi_i \equiv (\phi_1)^i \mod p$,
and thus $\phi_i\circ m\equiv \phi_i\odot \phi_0 + \phi_0\odot \phi_i \mod p$. 
We claim, however, that $\phi_i \neq (\phi_1)^i \mod p^2$.
Without loss of generality we may choose $n$-th Morava K-theory
to have the logarithm $\log_{\Kn}(x)=\sum_{s=0}^\infty \frac{1}{p^s}x^{p^{sn}}$.

By the construction of Chern classes \cite[Lemma 4.4.1, Lemma 4.3.3 (2p)]{Sech}
we have $\log_{\Kn}(c_1t+c_2t^2+\ldots+c_{p^{sn}}t^{p^{sn}})[t^{p^{sn}}]=\frac{\phi_{p^{sn}}}{p^s}$.
Multiplying this equation by $p^s$ and taking it modulo $p^2$ we obtain

$$ p(c_1t+\ldots+c_{p^{sn}}t^{p^{sn}})^{p^{(s-1)n}}[t^{p^{sn}}]+
(c_1t+\ldots+c_{p^{sn}}t^{p^{sn}})^{p^{sn}}[t^{p^{sn}}]\equiv \phi_{p^{sn}} \mod p^2.$$

Using equality $c_1=\phi_1$ we may rewrite it as
$$  p Q(c_1, \ldots, c_{p^{sn}}) + \phi_1^{p^{sn}} \equiv \phi_{p^{sn}} \mod p^2,$$

where $Q$ is a polynomial in Chern classes. 
The polynomial $Q$ contains $(c_{p^n})^{p^{(s-1)n}}$ as a summand
which comes from the term $(c_1t+\ldots+c_{p^{sn}}t^{p^{sn}})^{p^{(s-1)n}}$
and can not be cancelled by any other summands coming from the other term.
Thus, $Q\neq 0$, and therefore it is not zero as an operation to Chow groups modulo $p$,
because Chern classes form a basis of operations and there are no relations between them modulo $p$.
Thus, we see that $\phi_i \neq (\phi_1)^i \mod p^2$,
and therefore as explained above $\phi_i\circ m \neq \phi_i\odot \phi_0 + \phi_0\odot \phi_i \mod p$.
Therefore the minimal $p$-valuation of coefficients $b^{(i)}_j$ is equal to 1,
and $d_i = p\cdot \max_{j=1}^{i-1} (d_j d_{i-j})$.
\end{proof}

\begin{Exl}\label{app:exl_add_image}
Prop. \ref{prop:app_add_chernchar} allows to calculate 
$d_2=\cdots = d_{p^n-1}=1$, $d_{p^n}=\cdots =d_{2p^n-1}=p$,
and, more generally, $d_{kp^n}=\ldots=d_{kp^n+p^n-1}=p^k$ for $k<p^n$.
\end{Exl}

Define numbers $b_i$ according to the rule
$b_i =\begin{cases} 
d_i, &\text{if $p^n \nmid i$;}\\
p^{-k}d_i, &\text{if $i=p^{nk}v$ where $p^n\nmid v$}.
\end{cases}$

\begin{Prop}
For each $i\ge 0$ numbers $b_i \in p^\NN$. 

The image of the map $c_i^{\CH}\colon  gr^i_\tau \Kn^*\rarr \CH^i\ot\Zp$
equals to $b_i\Zp(\CH^i\ot\Zp)$.
\end{Prop}
\begin{proof}
It follows from Prop. \ref{prop:operations_chow_topfilt} 
that the image is equal to the image of 
the operation $\Theta_i:=(c^{\CH}_i)^{\CH}_i\colon \CH^i\ot\Zp\rarr \CH^i\ot\Zp$
which is obtained as the composition of $c_i$ with 
the canonical morphism $\rho_{\Kn}\colon \CH^i\rarr gr^i_\tau \Kn^*$.
We are going to prove that this operation is multiplication by $b_i$.

The operation $\Theta_i$ is a composition of additive operations and therefore it is additive itself.
By Vishik's classification of additive operations \cite[Th. 6.2]{Vish1} $\Theta_i$
is determined by a symmetric polynomial $G_i(z_1,\ldots,z_i)$ 
of degree $i$ which is divisible by $\prod_{j=1}^i z_j$,
i.e.\ $G_i= \alpha z_1\cdots z_i$ for some $\alpha\in\Zp$, and it is clear
that $\Theta_i$ is multiplication by $\alpha$.
The polynomial $G_i$ for the operation $c^{\CH}_i$ is also equal to $\alpha z_1\cdots z_i$
which can be checked by comparing the action of $\Theta_i$ and $c^{\CH}_i$ on products of projective spaces using the definition of $\Theta_i$.
Thus, our goal is to show that $c^{\CH}_i(z_1\cdots z_i)= a b_i z_1\cdots z_i$ for some $a\in \Zp^\times$.
 
By construction of the operation $c^{\CH}_i$ (\cite[Cor. 4.3.1]{Sech})
we have $c_i - P_i(c_1,\ldots, c_{i-1})=\frac{\phi_i}{p^{\mu_i}}$ as operations to $\CH^i\ot\QQ$
where $P_i$ is a rational polynomial, $\phi_i$ is a generator of additive operations to $\CH^i\ot\Zp$.
For operations $c_j, j<i$ the polynomial $G_i$ is equal to 0 by degree reasons,
and therefore so it is for $P_i$. Thus, the polynomial $G_i$ of the operation $c_i$ 
coincides with the polynomial of the operation $\frac{\phi_i}{p^{\mu_i}}$.

The operation $\phi_i$ is equal to $a p^{d_i} \cdot ch_i$ where $a\in\Zp^\times$
by definition of numbers $d_i$ above.
Therefore $G_i$ of the operation $\phi_i$ is equal to $ap^{d_i} z^1\cdots z_i$,
and $\alpha = a p^{d_i-\mu_i}$.
To finish the proof recall that $\mu_i=\max (0, -\nu_p(P_i))$ (\cite[Cor. 4.3.1]{Sech})
and it follows from \cite[Lemma 4.3.3 (2, 2p)]{Sech} that $\mu_i = 0$ if $p^n \nmid i$
and $\mu_i=k$ if $i=p^{nk}v$ where $p^n\nmid v$, $v\in\Zp$. Thus, $\alpha = a b_i$.

Since $\Theta_i$ is an integral operation,
it follows that $b_i$ is integral, i.e.\ $b_i\in p^\NN$,
which is the first claim of the Proposition.
\end{proof}

\begin{Exl}
If $i=kp^n+j<p^{2n}$ where $j\colon 1\le j\le p^n$, 
then the image $c_i(gr_\tau^i \Kn^*)$ is equal to $p^{k} \CH^i\ot\Zp$.
\end{Exl}


\medskip

\medskip

\noindent

\end{document}